\documentclass[reqno]{amsart}

\usepackage{amsmath,amsthm,amssymb,amsfonts}

\usepackage{aligned-overset}
\usepackage{enumitem}
\usepackage[margin=35mm]{geometry}
\usepackage{mathtools}
\usepackage{tikz-cd}
\usepackage{xspace}
\usepackage{fixme}
\fxsetup{
  status=final, % show notes; change to final to hide
}

\newlist{equivalence}{enumerate}{1}
\setlist[equivalence]{label=(\alph*)}

\usepackage[colorlinks=true,
            linkcolor=blue,
            citecolor=blue,
            urlcolor=blue]{hyperref}

\newtheorem{theorem}{Theorem}[section]
\newtheorem{corollary}[theorem]{Corollary}
\newtheorem{lemma}[theorem]{Lemma}
\newtheorem{proposition}[theorem]{Proposition}
\newtheorem{example}[theorem]{Example}

\theoremstyle{definition}
\newtheorem{definition}[theorem]{Definition}

\theoremstyle{remark}
\newtheorem{remark}[theorem]{Remark}

\numberwithin{equation}{section}

\newcommand{\C}{\mathbb{C}}

\newcommand{\complex}{\mathbb{C}}

\newcommand{\integers}{\mathbb{Z}}

\DeclareMathOperator{\Hom}{Hom}
\DeclareMathOperator{\aut}{Aut}
\DeclareMathOperator{\der}{Der}

\DeclareMathOperator{\inn}{Inn}
\DeclareMathOperator{\Ad}{Ad}
\DeclareMathOperator{\Sym}{Sym}

\newcommand{\hd}{\smash{\hat{\delta}}}

\newcommand{\id}{\operatorname{id}}
\newcommand{\gau}{\operatorname{gau}}
\newcommand{\omegat}{\tilde{\omega}}

\newcommand{\Dergr}{\operatorname{Der}_{\textup{gr}}}
\newcommand{\Derext}{\operatorname{Der}_{\textup{ext}}}
\newcommand{\LConetwo}{L_{\complex}(1,2)}
\newcommand{\g}{\mathfrak{g}}
\newcommand{\deltat}{\tilde{\delta}}
\newcommand{\paraa}[1]{\big(#1\big)}
\newcommand{\parab}[1]{\Big(#1\Big)}

\newcommand{\kronr}{\triangleright}
\newcommand{\kronl}{\triangleleft}

\newcommand*{\Star}{$^\ast$\nobreakdash}

\newcommand{\cf}{\mbox{cf.}\xspace}				% confer 
\newcommand{\eg}{\mbox{e.\,g.}\xspace}			% exempli gratia	
\newcommand*{\ie}{\mbox{i.\,e.}\xspace}			% id est
\author{Joakim Arnlind}
\address{Department of Mathematics\\Link\"oping University\\Sweden}
\email{joakim.arnlind@liu.se}

\author{Stefan Wagner}
\address{Department of Mathematics and Natural Sciences \\ Blekinge Institute of Technology\\Sweden}
\email{stefan.wagner@bth.se}

\title[Factor Systems and geometric structures of strongly graded rings]{Factor Systems and geometric structures \\ of strongly graded rings}



\begin{document}

\begin{abstract}
Graded rings provide a natural algebraic framework for encoding symmetry via decompositions into homogeneous components indexed by a group, together with multiplication rules reflecting the group operation.
Among graded rings, strongly graded rings form a particularly well-behaved and structurally rich class.

In this paper we introduce a notion of factor systems for strongly graded rings, consisting of algebraic data that encode both the bimodule structure of the homogeneous components and their multiplication relations. 
In particular, this framework makes it possible to carry~out explicit computations.

We show that strongly graded rings with fixed principal component are classified, up to isomorphism, by conjugacy classes of such factor systems. 
Conversely, every abstract factor system gives rise to a strongly graded ring realizing it. 
In this way, the global structure of a strongly graded ring can be reconstructed from algebraic data on the principal component together with the grading group.

Factor systems also provide a convenient framework for studying the problem of lifting derivations from the principal component to graded derivations of the whole ring. 
We derive explicit compatibility conditions for the existence of such lifts and interpret the resulting obstruction in cohomological terms. 
This leads to an algebraic analogue of the Atiyah sequence for strongly graded rings and to curvature-type invariants measuring the failure of graded lifts to form Lie algebra homomorphisms. 

The theory is illustrated by applying it to Leavitt path algebras.
\end{abstract}

\maketitle

%\listoffixmes

\tableofcontents

\section{Introduction}

Graded rings, which throughout this paper are assumed to be graded by groups, play a~fundamental role in many areas of mathematics by encoding symmetry, decomposition, and additional algebraic structure. 
A grading decomposes a ring into components indexed by a group, allowing one to study the ring through its homogeneous pieces and the relations between them.

More precisely, if $G$ is a group, then a ring $S$ is called \emph{$G$-graded} if it admits a decomposition
\begin{equation*}
    S = \bigoplus_{g\in G} S_g
\end{equation*}
into additive subgroups $(S_g)_{g\in G}$, called the \emph{homogeneous components}, 
such that
\begin{equation*}
    S_g S_h \subseteq S_{gh}
\quad \text{for all } g,h \in G .
\end{equation*}
The component $S_e$ corresponding to the identity element $e\in G$ is a subring of $S$, called the \emph{principal component}.

For a comprehensive treatment of graded rings we refer to the monograph of 
C.~N\u{a}st\u{a}sescu and F.~van Oystaeyen, \emph{Graded Ring Theory}~\cite{NaOy82}.

Gradings provide a powerful structural tool: many global properties of a graded ring - such as ideals, primeness, simplicity, and module-theoretic behavior - can often be studied through the interaction of its homogeneous components.
This viewpoint has proved particularly fruitful in the study of strongly graded rings and their generalizations.

A $G$-graded ring $S$ is called \emph{strongly $G$-graded} if
\begin{equation*}
    S_g S_h = S_{gh}
\quad \text{for all } g,h \in G .
\end{equation*}
In this case each homogeneous component $S_g$ is an invertible $S_e$-bimodule. 
Strongly graded rings form a distinguished subclass of graded rings, as the strong grading condition forces maximal compatibility between the homogeneous components and leads to a rich structural theory (see, for instance,~\cite{NaOy82}).

Strong gradings arise naturally in a wide range of mathematical contexts. 
Important examples include crossed products and skew group rings in ring theory, saturated Fell bundles in operator algebra theory, and Hopf--Galois extensions in noncommutative geometry. 
%In many situations the grading is induced by symmetries or group actions, and the strong grading condition ensures that these symmetries are reflected faithfully in the algebraic structure.

Despite their ubiquity, the internal structure of strongly graded rings is often difficult to describe explicitly.  
The guiding idea of this work is that strongly graded rings admit a structural description in terms of algebraic data attached to the principal component and the grading group, thereby extending the classical notion of a factor system from group extension theory.

The problem of describing extensions of algebraic structures in terms of auxiliary data has a long history.  
In group theory, the classification of extensions of a group $G$ by a group $N$ was already studied in Schreier’s PhD thesis from 1923 (see also \cite{Schreier1925}) and in the work of Baer~\cite{Baer34} in the 1930s.  
A central insight of this theory is that a group extension can be encoded by a \emph{factor system} satisfying suitable compatibility conditions.  
This viewpoint later became a cornerstone of the cohomological theory of group extensions developed by Eilenberg and MacLane~\cite{MacEil42,EilenbergMacLane}.

An analogous philosophy appears in the $C^\ast$-algebraic setting of group actions.  
For free~actions of compact quantum groups on unital $C^\ast$-algebras, a generalized factor system theory has been developed, encoding such actions via data associated with the fixed point algebra and the representation theory of the quantum group~\cite{SchWa15,SchWa16,SchWa17}.  
The resulting framework yields a versatile and effective tool for analyzing noncommutative principal bundles.

The main contribution of this paper is the introduction of a notion of factor systems for strongly graded rings.
Roughly speaking, a factor system associated with a strongly graded ring encodes two fundamental aspects of the graded structure:

\pagebreak[3]
\begin{itemize}
\item 
    the bimodule structure of the homogeneous components over the principal component,
\item 
    the multiplication relations between these components.
\end{itemize}
In this way, the global structure of the graded ring can be reconstructed from algebraic data on the principal component together with the grading group.

This perspective provides a flexible framework for studying strongly graded rings. 
In particular, it enables classification results, the study of permanence properties, applications of $K$-theory, and explicit computations in concrete examples, as well as the construction of invariants and characteristic classes.

%This perspective provides a flexible framework for studying strongly graded rings. In particular, it enables
%\begin{itemize}
%\item 
%    classification results for strongly graded rings with fixed principal component,
%\item 
%    the study of permanence phenomena, \ie, which algebraic properties of the grading group and the principal component pass to the graded ring,
%\item 
%    the application of $K$-theoretic methods,
%\item 
%    the construction of invariants and characteristic classes, and
%\item 
%    explicit computations in concrete examples.
%\end{itemize}
Thus, factor systems allow one to reduce structural questions about strongly graded rings to algebraic data on the principal component and the grading group.

Furthermore, the geometry of strongly graded rings has received comparatively little systematic attention.
Their close relationship to Hopf--Galois extensions, which form a natural subclass, suggests that strongly graded rings may be viewed as algebraic models for noncommutative principal bundles.

In differential geometry, connections and curvature arise from the Atiyah sequence associated with a principal bundle.  
For a smooth principal bundle $q \colon P \to X$ with structure group $G$, the Atiyah sequence
\[
    0 \longrightarrow \mathfrak{gau}(P) \longrightarrow \mathcal{V}(P)^G \longrightarrow \mathcal{V}(X) \longrightarrow 0
\]
relates $G$-equivariant vector fields on $P$ to vector fields on the base manifold $X$.  
Sections of this sequence correspond to connections, and the associated curvature measures the failure of such sections to define Lie algebra homomorphisms.  
This structure underlies the classical Chern--Weil construction of characteristic classes.

Motivated by this picture, we develop an algebraic analogue of the Atiyah sequence in the setting of strongly graded rings (\cf~\cite[Sec.~7]{SchWa21}).
Using the factor system associated with a strongly graded ring, we study the problem of lifting derivations from the principal component to graded derivations of the entire ring.  
%We show that the obstruction to such lifts is given by a cohomological invariant associated with the factor system.
We show that the obstruction to such lifts can be expressed in terms of a cohomological invariant associated with the factor system.  
This leads naturally to a notion of curvature that measures the failure of graded lifts to define a Lie algebra homomorphism, thereby providing an algebraic counterpart to curvature in the classical setting.

\subsection*{Conventions and notation}

Throughout the paper, the identity element of a group is denoted by $e$.
All rings are assumed to be associative and unital with identity element denoted by $1$.
For a ring $R$ and integers $n,m \ge 1$, we write $M_{n,m}(R)$ for the set
of $n \times m$ matrices with entries in $R$.
Elements of $M_{n,1}(R)$ and $M_{1,n}(R)$ are interpreted as column and row
vectors, respectively.
For a matrix $A \in M_{n,m}(R)$, we denote by $A^t \in M_{m,n}(R)$ its transpose.
Matrix multiplication is used whenever the dimensions are compatible.

We also make use of Leavitt path algebras associated with directed graphs.
Since their detailed construction is not central to the main exposition,
we recall their definition and basic properties in Appendix~\ref{sec:LPA}.

\subsection*{Organization of the paper}

%In Section~\ref{sec:prelim} we collect the definitions, notation, and auxiliary results that will be used throughout the paper.

In Section~\ref{sec:facsys} we develop the theory of \emph{factor systems associated with strongly graded rings}.

We begin by recalling the motivating example of crossed products and then show how factor systems arise naturally from arbitrary strongly graded rings in Section~\ref{sec:facsys_assoc}. 
Using Dade’s characterization of strong gradings, we introduce \emph{module frame systems} (see Definition~\ref{def:mfs}) and construct the associated \emph{factor system} (see Definition~\ref{def:factorsystem}), which encodes both the bimodule structure of the homogeneous components and their multiplication. 
We establish the fundamental relations satisfied by factor systems (see Lemma~\ref{lem:factorsystem}) and prove that strongly graded rings with fixed principal component are classified, up to isomorphism, by their factor systems (see Theorem~\ref{thm:equivalence}).

In Section~\ref{sec:facsys_abs} we isolate the defining properties of factor systems arising from strongly graded rings, thereby introducing the notion of an \emph{abstract factor system} (see Definition~\ref{def:facsys_abs}).
Given a group $G$ and a unital ring $R$, we then construct a strongly $G$-graded ring from such data and show that the original abstract factor system is realized as an associated factor system of this ring (see Theorem~\ref{thm:facsys}).
This construction yields a correspondence between conjugacy classes of abstract factor systems for $(R,G)$ and equivalence classes of strongly $G$-graded rings with principal component~$R$ (see Corollary~\ref{cor:class}).

In Section~\ref{sec:involution}, we study the additional structure that arises when the graded ring carries a compatible involution. 
We introduce the notion of an \emph{algebraic Parseval property} (see Definition~\ref{def:parseval}), which allows module frame systems to be chosen in a symmetric form, and analyze the consequences for the associated factor systems (see Lemma~\ref{rem:parseval} and Lemma~\ref{lem:parseval}). 
In particular, this leads to a concrete description of the involution in terms of the factor system. 
The theory is illustrated with examples from complex Leavitt path algebras (see Section~\ref{sec:clpa}).

In Section~\ref{sec:lifting} we study the problem of lifting derivations from the principal component of a strongly graded ring to graded derivations of the whole ring. 
Using the factor system associated with a strongly graded ring, we derive explicit algebraic conditions characterizing when a derivation of the principal component admits a graded lift (see Theorem~\ref{thm:liftder}). 
We then specialize these conditions to crossed products and skew group rings, where they take a more concrete form and lead to a range of illustrative examples and consequences, including the case of compatible involutions.

In Section~\ref{cohomo} we reinterpret the lifting problem in cohomological terms. 
Starting from a family of covariant derivatives on the homogeneous components, we define defect maps and show that they determine a central $2$-cocycle, called the \emph{multiplicative curvature} (see Definition~\ref{def:mult_curv}). 
Its cohomology class measures the obstruction to lifting a derivation to a graded derivation. 
More precisely, we show that a graded lift exists precisely when there exist covariant derivatives $\nabla_{\delta,g}$ along $\delta$ on the homogeneous components and when the associated multiplicative curvature vanishes (see Theorem~\ref{thm:liftder_cohomo}). 
We further discuss consequences in the equivariant case and also obtain a lifting criterion for strongly $\mathbb{Z}$-graded rings (see Corollary~\ref{cor:Z_lift_from_generator}).

In Section~\ref{sec:Atiyah} we introduce the \emph{Atiyah sequence} associated with a strongly graded ring and describe its terms explicitly. 
In particular, we identify the gauge Lie ring with the Lie ring of crossed homomorphisms on the grading group with values in the center of the principal~component (see Theorem~\ref{thm:gau}). 
This provides a natural Lie-theoretic framework for the lifting problem and the study of derivations admitting graded lifts.

In Section~\ref{sec:atiyah_curvature} we study the \emph{Atiyah curvature} associated with a section of the Atiyah sequence. 
We show that it measures the failure of a family of graded lifts to define a Lie homomorphism, relate it to the curvature of the induced covariant derivatives on the homogeneous components, and place it in the context of Lecomte’s Chern--Weil theory. 
We also derive an explicit matrix description of this curvature in terms of an associated factor system.

Finally, in Appendix~\ref{sec:LPA} we recall the definition and basic properties of Leavitt path algebras, which include natural examples of unital strongly $\mathbb{Z}$-graded ${}^\ast$-algebras illustrating the theory.

\subsection*{Future directions}

We conclude by emphasizing that the present work forms part of a broader research program aimed at the systematic study of Levi--Civita--type structures in noncommutative spaces. 
A central objective of this program is the development of a general theory of derivation-based Levi--Civita--type connections for noncommutative spaces (see, \eg,~\cite{a:lc.kronecker,ah:on.the.existence.LC}). 
Within this framework, we plan to identify a natural Atiyah-type sequence for derivations, providing the algebraic infrastructure for defining connections, metric compatibility, and torsion in the noncommutative setting. 

The results of the present paper establish the necessary algebraic for this approach in the context of strongly graded rings, interpreted as algebraic models of noncommutative principal bundles. 
They thus provide a natural starting point for the systematic study of Levi--Civita--type structures in this framework.

\section{Factor systems of strongly graded rings}
\label{sec:facsys}

Let $G$ be a fixed group.
To understand the internal structure of strongly $G$-graded rings it is natural to look for data on the corresponding principal components together with $G$ that encode both the grading and the multiplication.
A guiding example is furnished by crossed~products, whose construction we recall for convenience:

Let $R$ be a ring.
A pair consisting of a group homomorphism $\alpha: G \to \aut(R)$ and a map $\omega: G \times G \to R^\times$, where $R^\times$ denotes the group of units of $R$, is called a \emph{factor system} if the following compatibility conditions are satisfied:
\begin{align}
    \alpha_g \circ \alpha_h &= \Ad(\omega(g,h)) \circ \alpha_{gh}, \label{eq:coaction_condition}
    \\
    \omega(g,h) \omega(gh,k) &= \alpha_g(\omega(h,k)) \omega(g,hk) \label{eq:cocycle_condition}
\end{align}
for all $g,h,k \in G$, where $\Ad(x)(r) := x r x^{-1}$ denotes conjugation in $R$.
Given such a factor~system $(\alpha,\omega)$, the \emph{crossed product} $R \rtimes_{(\alpha,\omega)} G$ is defined as the set of formal sums
\begin{equation}\label{eq:formal_sums}
    \sum_{g \in G} r_g u_g, \quad r_g \in R,
\end{equation}
with multiplication determined by
\begin{equation}\label{eq:crossed_product}
    (r_g u_g)(r_h u_h) = r_g \alpha_g(r_h) \omega(g,h) u_{gh}.
\end{equation}
The subspaces $S_g := R u_g$, $g \in G$, define a strong $G$-grading with principal component $S_e = R$.
Each $u_g$, $g \in G$, is invertible with inverse $\omega(g^{-1},g)^{-1} u_{g^{-1}}$, so that $S_g = R u_g$ is a free, rank-one, invertible $R$-bimodule.

More generally, the existence of invertible elements in the homogeneous components is decisive:
if a strongly $G$-graded ring $S = \bigoplus_{g \in G} S_g$ contains, for each $g \in G$, an invertible element $u_g \in S_g$ such that $S_g = S_e u_g = u_g S_e$, then $S$ is isomorphic to a crossed product $S_e \rtimes_{(\alpha,\omega)} G$ associated with a suitable factor system $(\alpha,\omega)$.

Indeed, writing $R := S_e$ and fixing such a family $(u_g)_{g \in G}$, the corresponding factor system $(\alpha,\omega)$ is given as follows.
The action $\alpha: G \to \aut(R)$ is defined by conjugation,
\begin{equation*}
    \alpha_g(r) := u_g r u_g^{-1}, \quad r \in R,
\end{equation*}
and the cocycle $\omega: G \times G \to R^\times$ is determined by the relations
\begin{equation*}
    u_g u_h = \omega(g,h) u_{gh}, \quad g,h \in G.
\end{equation*}

The existence of such invertible elements already follows under a
simple module-theoretic condition on the homogeneous components.
The following observation, due to J.~Öinert (personal communication),
will be useful later.

\begin{lemma}\label{lem:j.oi}
    Let $S=\bigoplus_{g\in G} S_g$ be a unital strongly $G$-graded ring with principal component $R:=S_e$.
    If $S_g$ is a free right $R$-module of rank one for every $g\in G$, then each $S_g$ contains an invertible element.
    The analogous statement holds if each $S_g$ is a free left $R$-module of rank one.
\end{lemma}
\begin{proof}
We prove the right-module case. 
The left-module case is analogous.
Let $g\in G$ and let $u\in S_g$ be a basis of $S_g$ as a free right $R$-module,
\ie, $S_g=uR$.
Since $S$ is strongly graded,
\begin{equation*}
    1=\sum_i x_i y_i
\end{equation*}
for suitable $x_i \in S_g$ and $y_i \in S_{g^{-1}}$.
Writing $x_i = u r_i$ with $r_i \in R$ gives
\begin{equation*}
    1= \sum_i (u r_i) y_i = u \left(\sum_i r_i y_i\right) = uv
\end{equation*}
for some $v \in S_{g^{-1}}$.
To conclude, note that the map $R \to uR$, $r \mapsto ur$, is injective.
Using $uv=1$, one obtains $u(vu)=u$, and hence $vu=1$.
Thus $u$ is invertible with inverse $v$.
\end{proof}

Even within this structured setting, crossed products give rise to a rich variety of interesting examples and capture a broad spectrum of noncommutative phenomena.
Most importantly, the factor system framework provides a transparent algebraic formalism that makes the underlying structure explicit and computations remarkably accessible.

Motivated by this, we now extend the framework from crossed products to arbitrary strongly $G$-graded rings. 
%The exposition follows the approach of \cite{SchWa17}.

\subsection{Factor systems associated with strongly \texorpdfstring{$G$}{G}-graded rings}
\label{sec:facsys_assoc}

As a preliminary step, we recall Dade’s characterization of strongly graded
rings \cite[Prop.~1.6]{Dade1980} in a form convenient for our purposes:

\begin{lemma}\label{lem:dade}
    A $G$-graded ring $S = \bigoplus_{g \in G} S_g$ is strongly $G$-graded if and only if for each $g \in G$ there exist a positive integer $n_g$ and elements 
    \begin{equation*}
        x_g \in M_{n_g,1}(S_g)
        \quad \text{and} \quad
        y_g \in M_{n_g,1}(S_{g^{-1}})
    \end{equation*}
    such that $y_g^t x_g = 1$.
\end{lemma}

The lemma shows that strong $G$-grading can be witnessed by suitable matrix pairs. 

Leavitt path algebras form an important class of $\mathbb Z$-graded rings arising from directed graphs.
Under suitable conditions on the graph they are strongly graded and provide natural examples of strongly graded rings that are not crossed products.
For the convenience of the reader, we recall their definition and the relevant structural properties in Appendix~\ref{sec:LPA}.

The next result shows that for Leavitt path algebras one can construct elements $x_g,y_g$ satisfying the conditions of Lemma~\ref{lem:dade} in all negative homogeneous degrees (and in degree $0$).
The corresponding construction in positive degree is more delicate and
will be obtained in Theorem~\ref{thm:xdaggerx.eeast}.

\begin{lemma}\label{lem:neg.degree.x}
    Let $L_k(E)$ be a strongly graded unital Leavitt path algebra.
    For $n \ge 1$, let $P_n(E)$ denote the finite set of real paths of length $n$.
    If
    \begin{equation*}
        x_n := (p^\ast)_{p\in P_n(E)}^t
        \quad \text{and} \quad
        y_n := (p)_{p\in P_n(E)}^t,
    \end{equation*}
    then $y_n^t x_n = 1$.
\end{lemma}
\begin{proof}
For $v\in E_0$ and $n \ge 1$, let $P_n(E,v)$ denote the set of all real paths of length $n$ with source $v$. 
It is first shown that
\begin{equation}\label{eq:ppstar.one}
    \sum_{p \in P_n(E,v)} pp^\ast = v.
\end{equation}

For $n=1$ this is the Cuntz--Krieger relation
\begin{equation*}
    \sum_{p \in P_1(E,v)} p p^\ast
    =
    \sum_{e: s(e)=v} e e^\ast
    = v.
\end{equation*}

Assume that~\eqref{eq:ppstar.one} holds for some $n\ge 1$. 
A direct computation gives
\begin{gather*}
    \sum_{p \in P_{n+1}(E,v)} pp^\ast
    =
    \sum_{e \in s^{-1}(v)} \sum_{p \in P_{n}(E,r(e))}(ep) (ep)^\ast
    =
    \sum_{e \in s^{-1}(v)} \sum_{p \in P_{n}(E,r(e))} e(p p^\ast) e^\ast
    \\
    =
    \sum_{e \in s^{-1}(v)}e \left(\sum_{p\in P_{n}(E,r(e))} p p^\ast\right) e^\ast
    =
    \sum_{e \in s^{-1}(v)}er(e)e^\ast
    =
    \sum_{e \in s^{-1}(v)}ee^\ast
    = v,
\end{gather*}
which proves~\eqref{eq:ppstar.one} by induction.

Fix $n\ge 1$. Then
\begin{equation*}
    y_n^t x_n
    = \sum_{p\in P_n(E)} pp^\ast
    = \sum_{v\in E_0}\sum_{p\in P_n(E,v)} pp^\ast
    = \sum_{v\in E_0} v
    = 1,
\end{equation*}
since $L_k(E)$ is unital and $E$ has no sinks.
\end{proof}

We formalize this matrix-pair realization of as follows:

\begin{definition}\label{def:mfs}
    Let $S=\bigoplus_{g\in G} S_g$ be a strongly $G$-graded ring.
    A \emph{module frame system} for $S$ is a family $(x_g,y_g)_{g\in G}$
    such that for each $g\in G$ there exists a positive integer $n_g$ with 
    \begin{equation*}
        x_g \in M_{n_g,1}(S_g),
        \quad
        y_g \in M_{n_g,1}(S_{g^{-1}})
        \quad \text{and} \quad
        y_g^t x_g = 1.
    \end{equation*}
    In addition, $n_e = 1$ and $x_e = y_e = 1$.
\end{definition}

By Lemma~\ref{lem:dade}, every strongly $G$-graded ring admits a module frame system.

%\begin{definition}\label{def:mfs}
%    A \emph{module frame system} for $S$ is a family $(x_g,y_g)_{g\in G}$ as in Lemma~\ref{lem:dade} with $n_e = 1$ and $x_e = y_e = 1$.
%\end{definition}

Let $S=\bigoplus_{g\in G} S_g$ be a strongly $G$-graded ring, and let $(x_g,y_g)_{g\in G}$ be a module frame system for $S$.
With this choice, each homogeneous component $S_g$, for $g \in G$, admits an explicit description in terms of the principal component $S_e$. 
Indeed, for each $g \in G$,
\begin{equation}\label{eq:hom_comp}
    S_g = M_{1,n_g}(S_e) x_g,
\end{equation}
so that any element $s \in S_g$ can be written as
\begin{align*}
    s = u^t x_g
\end{align*}
for some $u \in M_{n_g,1}(S_e)$. 
Moreover, $u^t$ may be chosen uniquely in $M_{n_g,1}(S_e) x_g y_g^t$.

This description allows both the bimodule structure and the multiplication on the homogeneous components to be expressed entirely in terms of $S_e$ and the group $G$.
We make this~explicit by introducing the following auxiliary maps and notation:

For each $g\in G$, define a ring homomorphism
\begin{equation*}
    \alpha_g: M_{n,m}(S) \to M_{n_g n,n_g m}(S)
\end{equation*}
by
\begin{align*}
    \alpha_g(A) =
    \begin{pmatrix}
        x_g s_{11} y_g^t & x_g s_{12} y_g^t & \cdots & x_g s_{1m} y_g^t 
        \\
        \vdots & \vdots & & \vdots 
        \\
        x_g s_{n1} y_g^t & x_g s_{n2} y_g^t & \cdots & x_g s_{nm} y_g^t
    \end{pmatrix},
\end{align*}
for $A=(s_{ij})\in M_{n,m}(S)$.
It is immediate that $\alpha_g$ is graded, in the sense that it restricts to
\begin{equation*}
    M_{n,m}(S_h) \to M_{n_g n,n_g m}(S_h)
\end{equation*}
for each $h\in G$.
Moreover, for $s \in S$, one has
\begin{equation}\label{eq:commrel1}
    x_g s = x_g s y_g^t x_g = \alpha_g(s) x_g.
\end{equation}

We also fix notation for Kronecker products:
Let $R$ be a ring, and let $A=(a_{ij})\in M_{n_1,n_2}(R)$ and $B=(b_{kl})\in M_{m_1,m_2}(R)$. 
The right and left Kronecker products of $A$ and $B$ are the matrices
\begin{equation*}
    A \kronr B \in M_{n_1 m_1,n_2 m_2}(R),
    \quad
    A \kronl B \in M_{n_1 m_1,n_2 m_2}(R)
\end{equation*}
defined by
\begin{equation*}
    A\kronr B =
    \begin{pmatrix}
        A b_{11} & A b_{12} & \cdots & A b_{1m_2} \\
        A b_{21} & A b_{22} & \cdots & A b_{2m_2} \\
        \vdots & \vdots & & \vdots \\
        A b_{m_1 1} & A b_{m_1 2} & \cdots & A b_{m_1 m_2}
    \end{pmatrix},
    \quad
    A\kronl B =
    \begin{pmatrix}
        a_{11} B & a_{12} B & \cdots & a_{1n_2} B \\
        a_{21} B & a_{22} B & \cdots & a_{2n_2} B \\
        \vdots & \vdots & & \vdots \\
        a_{n_1 1} B & a_{n_1 2} B & \cdots & a_{n_1 n_2} B
    \end{pmatrix}.
\end{equation*}

For all $g,h \in G$, define
\begin{align*}
    \omega(g,h) &:= (x_g \kronr x_h) y_{gh}^t \in M_{n_g n_h, n_{gh}}(S_e),
    \\
    \omegat(g,h) &:= x_{gh}(y_h \kronl y_g)^t \in M_{n_{gh},n_g n_h}(S_e).
\end{align*}
Then
\begin{equation}\label{eq:commrel2}
    x_g\kronr x_h = \omega(g,h) x_{gh}
    \quad\text{and}\quad
    y_{gh}^t \omegat(g,h) = (y_h\kronl y_g)^t .
\end{equation}

\begin{lemma}\label{lem:commrel}
    For $v \in M_{n_h,1}(S_e)$ one has
    \begin{equation*}
        x_g \kronr (v^t x_h) = \alpha_g(v)^t \omega(g,h) x_{gh} \in M_{n_g,1}(S_{gh}).
    \end{equation*}
\end{lemma}
\begin{proof}
    Let $v \in M_{n_h}(S)$.
    Then
    \begin{align*}
        x_g \kronr (v^t x_h) &=  
        \begin{pmatrix}
            x_g v_1 & x_g v_2 & \cdots & x_g v_{n_h}
        \end{pmatrix}
        x_h
        \overset{\eqref{eq:commrel1}}{=}
        \begin{pmatrix}
            \alpha_g(v_1)x_g & \alpha_g(v_2) x_g & \cdots & \alpha_g(v_{n_h}) x_g
        \end{pmatrix}x_h
        \\
        &= 
        \sum_{i=1}^{n_h} \alpha_g(v_i) x_g(x_h)_i
        = 
        \alpha_g(v)^t(x_g \kronr x_h)
        \overset{\eqref{eq:commrel2}}{=}
        \alpha_g(v)^t\omega(g,h)x_{gh}.
        \qedhere
    \end{align*}
\end{proof}

With this notation, the bimodule structure of $S_g$, for $g \in G$, is given by
\begin{equation}\label{eq:bimod}
    r \cdot (u^t x_g) = r u^t x_g,
    \quad 
    (u^t x_g) \cdot r \overset{\eqref{eq:commrel1}}{=} u^t \alpha_g(r) x_g,
\end{equation}
for all $r \in S_e$ and $u \in M_{n_g,1}(S_e)$.
Moreover, by Lemma~\ref{lem:commrel}, the product of homogeneous~elements satisfies
\begin{equation}\label{eq:multi}
    (u^t x_g) \cdot (v^t x_h) = u^t \alpha_g(v)^t \omega(g,h) x_{gh},
    \quad
    g,h \in G,
\end{equation}
for all $u \in M_{n_g,1}(S_e)$ and $v \in M_{n_h,1}(S_e)$.

The preceding discussion leads to the following definition:

\begin{definition}\label{def:factorsystem}
    Let $S = \bigoplus_{g\in G} S_g$ be a strongly $G$-graded ring, and let
    $(x_g,y_g)_{g\in G}$ be a module frame system for $S$.
    For $g,h \in G$, define
    \begin{equation*}
        \alpha_g : S_e \to M_{n_g}(S_e),
        \quad
        \alpha_g(r) := x_g r y_g^t,
    \end{equation*}
    and
    \begin{equation*}
        \omega(g,h) := (x_g \kronr x_h) y_{gh}^t \in M_{n_gn_h,n_{gh}}(S_e).
    \end{equation*}
    The triple
    \begin{equation*}
        (n,\alpha,\omega) := (n_g,\alpha_g,\omega(g,h))_{g,h \in G}
    \end{equation*}
    is called the \emph{factor system} associated with $(x_g,y_g)_{g\in G}$.
    We may simply speak of a factor system of $S$ without specifying the
    underlying module frame system.
\end{definition}

The basic properties of a factor system are summarized in the following lemma:

\begin{lemma}\label{lem:factorsystem}
    Let $S$ be a strongly $G$-graded ring, and let $(n,\alpha,\omega)$ be a factor system of $S$.
    Then the following assertions hold:
    \begin{enumerate}
        \item\label{it:omegae}
            One has $\omega(e,e) = 1$ and $\alpha_e = \id_{S_e}$.
        \item\label{it:omega_dagger} 
            For all $g,h \in G$,
            \begin{equation}\label{eq:paruni}
                \omegat(g,h) \omega(g,h) = \alpha_{gh}(1)
                \quad \text{and} \quad
                \omega(g,h) \omegat(g,h) = \alpha_g(\alpha_h(1)).
            \end{equation}
        \item\label{it:coaction}  
            For all $g,h \in G$ and $r \in S_e$,
            \begin{align}
                 \omega(g,h) \alpha_{gh}(r) 
                 = \alpha_g(\alpha_h(r)) \omega(g,h). 
                 \label{eq:coaction_condition_new}
            \end{align}
            %(see~\eqref{eq:cocycle_condition}).
         \item\label{it:cocycle} 
            For all $g,h,k \in G$,
            \begin{align}
                 (\omega(g,h) \kronr 1_{n_k}) \omega(gh,k) 
                 = \alpha_g(\omega(h,k)) \omega(g,hk). \label{eq:cocycle_condition_new}
            \end{align}
            %(see~\eqref{eq:coaction_condition}).
        \item\label{it:conjugate}  
            Let $(n',\alpha',\omega')$ be another factor system of $S$.
            Then there exist families of elements $v_g \in M_{n'_g,n_g}(S_e)$, $w_g \in M_{n_g,n'_g}(S_e)$, $g \in G$, satisfying
            \begin{align}
                &v_g w_g = \alpha'_g(1)
                \quad \text{and} \quad
                w_g v_g = \alpha_g(1), \label{eq:v1}
                \\
                &\alpha'_g(r) v_g = v_g \alpha_g(r), \label{eq:v2}
                \\
                &\omega'(g,h) v_{gh} = v_g \alpha_g(v_h) \omega(g,h) \label{eq:v3}
            \end{align}
            for all $g,h \in G$ and $r \in S_e$.
        \item
            Conversely, let $v_g \in M_{n'_g,n_g}(S_e)$, $w_g \in M_{n_g,n'_g}(S_e)$, $g \in G$, be families of elements satisfying $w_g v_g = \alpha_g(1)$ for all $g \in G$.
            Then the triple $(n',\alpha',\omega')$ defined as follows is a factor system of $S$:
            \begin{align*}
                &\alpha'_g(r) := v_g \alpha_g(s) w_g,
                \\
                &\omega'(g,h) := v_g \alpha_g(v_h) \omega(g,h) w_{gh} 
            \end{align*}
            for all $g,h \in G$ and $r \in S_e$.
    \end{enumerate}
\end{lemma}

\begin{remark}
    The elements $\tilde\omega(g,h)$, $g,h \in G$, are not included in the definition of a factor system; instead they are regarded as auxiliary elements for formulating the relations.
\end{remark}

\begin{proof}[Proof of Lemma~\ref{lem:factorsystem}]
    Let $(x_g,y_g)_{g\in G}$ be a module frame system for $S$ implementing the factor system $(n,\alpha,\omega)$.
    \setlist[enumerate]{wide=0pt, leftmargin=*, label=(\arabic*)}
    \begin{enumerate}
        \item 
            This follows immediately from the normalization $n_e = 1$, $x_e = 1$, and $y_e = 1$.
        \item
            Let $g,h \in G$. 
            A direct computation shows that
            \begin{gather*}
                \omegat(g,h) \omega(g,h) 
                = 
                x_{gh}(y_h \kronl y_g)^t (x_g \kronr x_h) y_{gh}^t 
                = 
                x_{gh}y_{gh} = \alpha_{gh}(1),
                \\
                \omega(g,h) \omegat(g,h) 
                = 
                (x_g \kronr x_h) y_{gh}^t x_{gh}(y_h \kronl y_g)^t
                = 
                (x_g \kronr x_h) (y_h \kronl y_g)^t 
                = 
                \alpha_g(x_hy_h^t) = \alpha_g(\alpha_h(1)).
            \end{gather*}
        \item
            Let $g,h \in G$ and $r \in S_e$. 
            Using that $\alpha_g(\alpha_h(s)) = (x_g \kronr x_h) r (y_h \kronl y_g)^t$, one obtains
            \begin{align*}
                \omega(g,h) \alpha_{gh}(r)
                &= 
                (x_g \kronr x_h) y_{gh}^t x_{gh} r y_{gh}^t
                = 
                (x_g \kronr x_h) r y_{gh}^t
                \\
                &= 
                (x_g \kronr x_h) r (y_h \kronl y_g)^t (x_g \kronr x_h) y_{gh}^t
                = 
                \alpha_g(\alpha_h(r)) \omega(g,h).
            \end{align*}
        \item
            Let $g,h,k \in G$.
            Then
            \begin{align*}
                (\omega(g,h) \kronr 1_{n_k}) \omega(gh,k)
                &= ((x_g \kronr x_h) \kronr 1_{n_k}) (y_{gh}^t \kronr 1_{n_k}) (x_{gh} \kronr x_k) y_{ghk}^t
                =
                (x_g \kronr x_h \kronr x_k) y_{ghk}^t.
            \end{align*}
            On the other hand, since $\alpha_g(\omega(h,k)) = (x_g \kronr x_h \kronr x_k) (y_{hk} \kronl y_g)^t$, one finds that
            \begin{equation*}
                \alpha_g(\omega(h,k)) \omega(g,hk)
                = (x_g \kronr x_h \kronr x_k) (y_{hk} \kronl y_g)^t (x_g \kronr x_{hk}) y_{ghk}^t
                = (x_g \kronr x_h \kronr x_k) y_{ghk}^t.
            \end{equation*}
            This establishes the formula.
        \item 
            Let $(x'_g,y'_g)_{g\in G}$ be a module frame system for $S$ implementing the factor system $(n',\alpha',\omega')$.
            For each $g \in G$, set 
            \begin{equation*}
                v_g := x'_g y_g^t
                \quad
                \text{and}
                \quad
                w_g := x_g (y'_g)^t
            \end{equation*}
            Now, fix $g,h \in G$ and $r \in S_e$.
            Then
            \begin{align*}
                v_g w_g = x'_g y_g^t x_g (y'_g)^t = x'_g (y'_g)^t = \alpha'_g(1),
                \quad \text{and} \quad
                w_g v_g = x_g (y'_g)^t x'_g y_g^t = x_g y_g^t = \alpha_g(1),
                \quad
            \end{align*}
            which proves~\eqref{eq:v1}.
            \eqref{eq:v2} is verified as follows:
            \begin{align*}
                \alpha'_g(r) v_g = x'_g r (y'_g)^t x'_g y_g^t = x'_g r y_g^t = x'_g y_g^t x_g r y_g = v_g \alpha_g(r),
            \end{align*}
            Finally, for~\eqref{eq:v3}, it is noted that $v_g \alpha_g(v_h) = (x'_g \kronr x'_h) (y_h \kronl y_g)^t$, so that
            \begin{align*}
                v_g \alpha_g(v_h) \omega(g,h) 
                &= (x'_g \kronr x'_h) (y_h \kronl y_g)^t (x_g \kronr x_h) y_{gh}^t
                \\
                &= (x'_g \kronr x'_h) y_{gh}^t
                = (x'_g \kronr x'_h) (y'_{gh})^t x'_{gh} y_{gh}^t
                = \omega'(g,h) v_{gh}.
            \end{align*}
        \item 
            It is straightforward to check that $(n',\alpha',\omega')$ is the factor system of $S$ associated with the elements $v_g x_g \in M_{n'_g,1}(S_g)$, $y_g^t w_g \in M_{1,n'_g}(S_{g^{-1}})$, $g \in G$.
            \qedhere
    \end{enumerate}
\end{proof}

\begin{remark}
    For crossed products, Definition~\ref{def:factorsystem} reproduces the classical notion of a factor system (see~\eqref{eq:coaction_condition} and~\eqref{eq:cocycle_condition}).
\end{remark}

\begin{example}\label{ex:L(1,2)_factorsystem}
We illustrate the above construction for the Leavitt path algebra $L_{\mathbb C}(1,2)$ associated with the graph $E$ with one vertex $v$ and two loop edges $e_1,e_2$, which is strongly $\mathbb Z$-graded but not a crossed product (see Section~\ref{sec:L(1,2)}).

%For $n=0$, set $x_0 = y_0 = 1$.
For $n=-1$, set
\begin{equation*}
    x_{-1} :=
    \begin{pmatrix}
        e_1^\ast
        \\
        e_2^\ast
    \end{pmatrix}
    \quad \text{and} \quad
    y_{-1} :=
    \begin{pmatrix}
        e_1
        \\
        e_2
    \end{pmatrix},
\end{equation*}
so that
\begin{equation*}
    y_{-1}^t x_{-1} = e_1e_1^\ast + e_2e_2^\ast = 1.
\end{equation*}
For $n<0$, define recursively
\begin{equation*}
    x_{n-1} := x_{-1} \kronr x_n
    \quad \text{and} \quad
    y_{n-1} := y_n \kronl y_{-1}.
\end{equation*}
Then $y_n^t x_n = 1$ for all $n<0$, as is easily checked.

For $n>0$, define
\begin{equation*}
    x_n := e_1^n,
    \quad \text{and} \quad
    y_n := (e_1^\ast)^n.
\end{equation*}
Then $y_n^t x_n = y_n x_n = (e_1^\ast)^n e_1^n = 1$.
Hence $(x_n,y_n)_{n\in \mathbb Z}$ is a module frame system for $L_{\mathbb C}(1,2)$.

For $n,m \ge 0$, the associated factor system can be described explicitly.
Indeed, since $x_n = e_1^n$ and $y_n = (e_1^\ast)^n$, one obtains
\begin{equation*}
    \alpha_n(r) = e_1^n r (e_1^\ast)^n
\end{equation*}
for $r \in L_0$, and
\begin{equation*}
    \omega(n,m)
    =
    (x_n \kronr x_m)y_{n+m}^t
    =
    e_1^{n+m}(e_1^\ast)^{n+m}.
\end{equation*}
In particular, the cocycle $\omega(n,m)$ is the range projection of the path
$e_1^{n+m}$.
\end{example}

\begin{remark}
The module frame system used in Example~\ref{ex:L(1,2)_factorsystem} is not
canonical. 
For instance, for $n>0$ one could equally choose
\begin{equation*}
    \tilde{x}_n :=
    \frac{1}{\sqrt{2}}
    \begin{pmatrix}
        e_1^n\\
        e_2^n
\end{pmatrix}.
\end{equation*}
Such choices lead to different module frame systems.
%but yield conjugate factor systems as described in Lemma~\ref{lem:factorsystem}.
\end{remark}

We now introduce a notion of equivalence for factor systems:

\begin{definition}\label{def:conjugate}
    Let $(n,\alpha,\omega)$ and $(n',\alpha',\omega')$ be factor systems associated with strongly $G$-graded rings $S$ and $S'$, respectively, having the same principal component $R = S_e = S'_e$.
    They are \emph{conjugate} if there exist families of elements
    \begin{equation*}
        v_g \in M_{n'_g,n_g}(R), 
        \quad
        w_g \in M_{n_g,n'_g}(R), 
        \quad
        g \in G,
    \end{equation*}
    such that for all $g,h \in G$ and $r \in S_e$ one has
    \begin{align*}
        &v_g w_g = \alpha'_g(1)
        \quad \text{and} \quad
        w_g v_g = \alpha_g(1). 
        \\
        &\alpha'_g(r) v_g = v_g \alpha_g(r),
        \\
        &\omega'(g,h) v_{gh} = v_g \alpha_g(v_h) \omega(g,h).
    \end{align*}
\end{definition}

\begin{theorem}\label{thm:equivalence}
    Let $S = \bigoplus_{g \in G} S_g$ and $S' = \bigoplus_{g \in G} S_g'$ be strongly $G$-graded rings with principal component $R$, and let $(n,\alpha,\omega)$ and $(n',\alpha',\omega')$ be associated factor systems, respectively. 
	Then the following statements are equivalent:
	\begin{equivalence}
	\item\label{en:equiv}
		The strongly $G$-graded rings $S$ and $S'$ are isomorphic.
	\item\label{en:conj}
		The factor systems $(n,\alpha,\omega)$ and $(n',\alpha',\omega')$ are conjugate.
	\end{equivalence}
\end{theorem}

\begin{proof}
We distinguish the data associated with $S' = \bigoplus_{g \in G} S'_g$ by adding a prime to all corresponding symbols.  
To establish that~\ref{en:equiv} implies~\ref{en:conj}, it suffices to verify that, for a fixed strongly $G$-graded ring, different choices of module frame systems yield conjugate factor systems.  
This is precisely the content of Lemma~\ref{lem:factorsystem}~(\ref{it:conjugate}).

Conversely, let $(x_g,y_g)_{g\in G}$ and $(x'_g,y'_g)_{g\in G}$ be module frame system for $S$ and~$S'$, respectively, implementing $(n,\alpha,\omega)$ and $(n',\alpha',\omega')$.
Moreover, let $v_g \in M_{n'_g,n_g}(S_e)$, $w_g \in M_{n_g,n'_g}(S_e)$, $g \in G$, be elements satisfying $w_g v_g = \alpha_g(1)$ for all $g \in G$, and which implement the conjugation between the two factor systems.

For each $g \in G$, define
\begin{equation*}
    \varphi_g : S'_g \to S_g, 
    \quad
    \varphi_g(u^t x'_g) := u^t v_g x_g,
    \quad u \in M_{n'_g,1}(R).
\end{equation*}
Then $\varphi_g$ is an isomorphism of $R$-bimodules with inverse 
$\varphi_g^{-1}(v x_g) = v w_g x'_g$, $v \in M_{n_g,1}(R)$.
Furthermore, for all $g,h \in G$, $u \in M_{n'_g,1}(R)$ and $v \in M_{n'_h,1}(R)$,
\begin{align*}
    \varphi_g(u^t x'_g) \cdot \varphi_h(v^t x'_h)
        &= 
        u^t v_g \alpha_g(v)^t \alpha_g(v_h) \omega(g,h) x_{gh} 
        \\
        &\overset{\eqref{eq:v2}}{=} 
        u^t \alpha'_g(v)^t v_g \alpha_g(v_h) \omega(g,h) x_{gh} 
        \\
        &\overset{\eqref{eq:v3}}{=} 
        u^t \alpha'_g(v)^t \omega'(g,h) v_{gh} x_{gh} 
        \overset{\eqref{eq:multi}}{=}
        \varphi_{gh}((u^t x'_g)(v^t x'_h)).
\end{align*}
Hence $\bigoplus_{g \in G} \varphi_g$ implements an equivalence between $S = \bigoplus_{g \in G} S_g$ and $S' = \bigoplus_{g \in G} S_g'$.
\end{proof}

\begin{remark}\label{rem:K0_factor_sys}
    Let $S$ be a strongly $G$-graded ring with principal component $R$, and let $(n,\alpha,\omega)$ be an associated factor system.
    For each $g \in G$, $\alpha_g: R \to M_{n_g}(R)$ induces a homomorphism
    $K_0(\alpha_g): K_0(R) \to K_0(R)$.
    \eqref{eq:paruni} and~\eqref{eq:coaction_condition_new} imply that
    $K_0(\alpha_{gh}) = K_0(\alpha_g) \circ K_0(\alpha_h)$ for all $g,h \in G$,
    and hence define a group homomorphism $G \to \aut(K_0(R))$.
    This assignment depends only on the conjugacy class of $(n,\alpha,\omega)$, and therefore yields an invariant of $S$.
\end{remark}

\subsection{Abstract factor systems}
\label{sec:facsys_abs}

%By Theorem~\ref{thm:equivalence}, strongly $G$-graded rings are uniquely determined by their factor systems up to conjugacy.
%We now address the reconstruction problem in the opposite direction, starting from a triple $(n,\alpha,\omega)$ satisfying the relations of Lemma~\ref{lem:factorsystem} and constructing an associated strongly $G$-graded ring.

The module frame description of strongly graded rings also yields an
explicit realization of the homogeneous components as finitely generated
projective modules over the principal component.
Indeed, let $S = \bigoplus_{g\in G} S_g$ be a strongly $G$-graded ring and
let $(x_g,y_g)_{g\in G}$ be a module frame system for $S$.
Then for each $g \in G$ the element
\begin{equation*}
    p_g := x_g y_g^t \in M_{n_g}(S_e)
\end{equation*}
is an idempotent.
Moreover, identifying $S_e^{n_g}$ with $M_{1,n_g}(S_e)$, the homogeneous component $S_g$ can be identified with the projective left $S_e$-module $S_e^{n_g} p_g$ via the map
\begin{equation*}
    S_e^{n_g} p_g \to S_g,
    \quad
    u \mapsto u x_g,
\end{equation*}
whose inverse is given by $s_g \mapsto s_g y_g^t$.

This observation will also be useful in the reconstruction of strongly
$G$-graded rings from abstract factor systems carried out below.
More precisely, we address the reconstruction problem in the opposite
direction: starting from a triple $(n,\alpha,\omega)$ satisfying the
relations of Lemma~\ref{lem:factorsystem}, we construct an associated
strongly $G$-graded ring.

Since a factor system consists solely of data on the principal component
and the group $G$, it is convenient to encode this information in the
following definition:

\begin{definition}\label{def:facsys_abs}
    Let $R$ be a unital ring and let $G$ be a group.  
    A \emph{factor system for $(R,G)$} is a triple $(n,\alpha,\omega)$ consisting of  
    \begin{itemize}
        \item 
            a family $(n_g)_{g \in G}$ of positive integers, normalized by $n_e = 1$;
        \item 
            a family $(\alpha_g : R \to M_{n_g}(R))_{g \in G}$ of ring homomorphisms, normalized by $\alpha_e = \id_R$;
        \item 
            a family $(\omega(g,h) \in M_{n_g n_h,\,n_{gh}}(R))_{g,h \in G}$ of elements, normalized by $\omega(g,1) = \omega(1,g) = \alpha_g(1)$ for all $g \in G$.
    \end{itemize}
    In addition, for all $g,h \in G$ there exists an auxiliary element $\tilde\omega(g,h) \in M_{n_{gh},n_g n_h}(R)$, and these data are required to satisfy, for all $g,h,k \in G$ and $r \in R$,
    \begin{gather*}
        \omegat(g,h) \omega(g,h) = \alpha_{gh}(1)
        \quad \text{and} \quad
        \omega(g,h) \omegat(g,h) = \alpha_g(\alpha_h(1)),
        \\
        \omega(g,h) \alpha_{gh}(r) = \alpha_g(\alpha_h(r)) \omega(g,h),
        \\
        (\omega(g,h) \kronr 1_{n_k}) \omega(gh,k) = \alpha_g(\omega(h,k)) \omega(g,hk).
    \end{gather*}
 %   In addition, for all $g,h \in G$ there exists an auxiliary element $\tilde\omega(g,h) \in M_{n_{gh},n_g n_h}(R)$, and with this notation the above data are required to satisfy~\eqref{eq:paruni}, \eqref{eq:coaction_condition_new}, and \eqref{eq:cocycle_condition_new} for all $g,h,k \in G$ and $r \in R$.
\end{definition}

\begin{remark}
    Each factor system of a strongly $G$-graded ring with principal component $R$ is, in particular, a factor system for $(R,G)$.
\end{remark}

Henceforth, let $R$ be a unital ring, let $G$ be a group, and let $(n,\alpha,\omega)$ be a factor system for $(R,G)$.  
By convention, for each $g \in G$, the map $\alpha_g: R \to M_{n_g}(R)$ is extended entrywise to matrix algebras over $R$, and this extension is denoted by the same symbol.

For each $g \in G$, define
\begin{equation*}
    S_g := M_{1,n_g}(R) \alpha_g(1),
\end{equation*}
and endow $S_g$ with its natural left $R$-module structure and the right $R$-module structure $x \cdot r := x \alpha_g(r)$ for all $x \in S_g$ and $r \in R$.
Then $S_g$ is an $R$-bimodule.  

\begin{lemma}\label{lem:multi}
    For $g,h\in G$, define the $R$-bilinear map
    \begin{equation*}
        m_{g,h}: S_g \times S_h \to S_{gh},
        \quad 
        m_{g,h}(s_g,s_h) := s_g \alpha_g(s_h) \omega(g,h).
    \end{equation*}
    The image of $m_{g,h}$ generates $S_{gh}$ as a left $R$-module.
\end{lemma}

The proof of Lemma~\ref{lem:multi} is reduced to two auxiliary results.

\begin{lemma}\label{lem:multi_1}
    For $g,h\in G$, define the $R$-bilinear map
    \begin{equation*}
        \mu_{g,h}: S_g \times S_h \to M_{1,n_gn_h}(R) \alpha_g(\alpha_h(1)),
        \quad 
        \mu_{g,h}(s_g,s_h) := s_g \alpha_g(s_h).
    \end{equation*}
    The image of $\mu_{g,h}$ generates $M_{1,n_gn_h}(R) \alpha_g(\alpha_h(1))$ (as a left $R$-module).
\end{lemma}
\begin{proof}
Fix $g,h \in G$, set $m := n_g$, $n := n_h$, and write
\begin{equation*}
    p := \alpha_g(\alpha_h(1)) \in M_{mn}(R).
\end{equation*}
Since the rows $e_{(i,j)}p \in M_{1,mn}(R)p$ generate $M_{1,mn}(R)p$
as a left $R$-module, it suffices to show that each such row lies in the
image of $\mu_{g,h}$.

Let $e_i \in M_{1,m}(R)$ and $e_j \in M_{1,n}(R)$ be the standard basis rows, and set
\begin{equation*}
    s_g := e_i \alpha_g(1) \in S_g
    \quad \text{and} \quad
    s_h := e_j \alpha_h(1) \in S_h.
\end{equation*}
By construction, $\alpha_g(e_j) \in M_{m,mn}(R)$ is the row block matrix with $\alpha_g(1)$ in the $j$th block and zeros elsewhere.
Consequently,
\begin{equation*}
    s_g \alpha_g(s_h)
    =
    e_i \alpha_g(1) \alpha_g(e_j) p
    =
    (0 | \cdots | e_i \alpha_g(1) | \cdots | 0) p.
\end{equation*}

Write $\alpha_h(1) = (q_{r\ell})$ and identify $p = (\alpha_g(q_{r\ell}))$ as an $n \times n$ block matrix with $m \times m$ blocks.
Then
\begin{equation*}
    (s_g \alpha_g(s_h))_{(k,\ell)}
    =
    \sum_{k'=1}^m (\alpha_g(1))_{i,k'}
    (\alpha_g(q_{j\ell}))_{k',k}
    =
    (\alpha_g(q_{j\ell}))_{i,k},
\end{equation*}
using $\alpha_g(1)\alpha_g(q_{j\ell})=\alpha_g(q_{j\ell})$.

On the other hand,
\begin{equation*}
    (e_{(i,j)}p)_{(k,\ell)}
    =
    p_{(i,j),(k,\ell)}
    =
    \big(\alpha_g(q_{j\ell})\big)_{i,k}.
\end{equation*}
Thus $s_g \alpha_g(s_h) = e_{(i,j)}p$, as required.
\end{proof}

The next lemma is a direct consequence of~\eqref{eq:paruni}:

\begin{lemma}\label{lem:multi_2}
For $g,h \in G$, the $R$-bimodule map
\begin{equation*}
    \rho_{\omega(g,h)}:
    M_{1,n_g n_h}(R) \alpha_g(\alpha_h(1)) \to S_{gh},
    \quad
    y \mapsto y\,\omega(g,h),
\end{equation*}
is an isomorphism of $R$-bimodules.
\end{lemma}

Since $m_{g,h}=\rho_{\omega(g,h)}\circ\mu_{g,h}$ for all $g,h\in G$, 
Lemma~\ref{lem:multi} now follows directly from
Lemmas~\ref{lem:multi_1} and~\ref{lem:multi_2}.

With this in place, we set
\begin{equation*}
    S := \bigoplus_{g\in G} S_g
\end{equation*}
and define a multiplication $m: S\times S\to S$ as follows:
for $g,h\in G$ and homogeneous elements $s_g\in S_g, s_h\in S_h$, set
\begin{equation*}
    m(s_g,s_h) := m_{g,h}(s_g,s_h),
\end{equation*}
and extend $m$ to all of $S\times S$ by $R$-bilinearity.

\begin{lemma}\label{lem:multiass}
    The map $m$ is associative.
\end{lemma}
\begin{proof}
Let $g,h,k \in G$ and let $s_g\in S_g, s_h\in S_h, s_k\in S_k$.
Then
\begin{align*}
    m(m(s_g,s_h),s_k)
    &=
    s_g \alpha_g(s_h) \omega(g,h) \alpha_{gh}(s_k) \omega(gh,k)
    \\
    &\overset{\eqref{eq:coaction_condition_new}}{=}
    s_g \alpha_g(s_h) \alpha_g(\alpha_h(s_k)) (\omega(g,h) \kronr 1_{n_k}) \omega(gh,k)
    \\
    &\overset{\eqref{eq:cocycle_condition_new}}{=}
    s_g \alpha_g(s_h) \alpha_g(\alpha_h(s_k)) \alpha_g(\omega(h,k)) \omega(g,hk)
%    \\
%    &=
%    x_g \alpha_g(x_h \alpha_h(x_k) \omega(h,k)) \omega(g,hk)
    =
    m(s_g,m(s_h,s_k)).
    \qedhere
\end{align*}
\end{proof}

Thus $S$ is a strongly $G$-graded ring with principal component $R$.

We proceed to construct a factor system of $S$:
For each $g\in G$, define
\begin{equation*}
    x_g
    :=
    (e_1\alpha_g(1),\ldots,e_{n_g}\alpha_g(1))^t
    \in M_{n_g,1}(S_g),
\end{equation*}
where $e_i \in M_{1,n_g}(R)$ denote the standard basis rows.
To obtain compatible elements in degree $g^{-1}$, write
$\omegat(g^{-1},g) \in M_{1,n_{g^{-1}}n_g}(R)$ in block form
\begin{equation*}
    \omegat(g^{-1},g)
    =
    (t_{g,1} | \cdots | t_{g,n_g}),
    \quad
    t_{g,i} \in M_{1,n_{g^{-1}}}(R),
\end{equation*}
and define
\begin{equation*}
    y_g
    :=
    (t_{g,1}\alpha_{g^{-1}}(1),\ldots,t_{g,n_g}\alpha_{g^{-1}}(1))^t
    \in M_{n_g,1}(S_{g^{-1}}).
\end{equation*}

\begin{lemma}
    The family $(x_g,y_g)_{g\in G}$ is a module frame system for $S$.
\end{lemma}
\begin{proof}
Observe that $n_e=1$ and $x_e = y_e = 1$.
Let $g \neq e $.
Then
\begin{align*}
    y_g^t x_g
    =
    \sum_{i=1}^{n_g} t_{g,i} \alpha_{g^{-1}}(1) \alpha_{g^{-1}}(e_i\alpha_g(1)) \omega(g^{-1},g)
    =
    \sum_{i=1}^{n_g} (0 | \cdots | t_{g,i} \alpha_{g^{-1}}(1) | \cdots | 0) \omega(g^{-1},g).
\end{align*}
By construction, this sum is precisely $\omegat(g^{-1},g)\omega(g^{-1},g)$, which equals $1$ by~\eqref{eq:paruni}.
\end{proof}

We claim that the factor system of $S$ associated with $(x_g,y_g)_{g\in G}$ coincides with the factor system $(n,\alpha,\omega)$.
The following auxiliary identity will be used to verify this.
Let $g\in G$.
Applying $\alpha_g$ to $\omegat(g^{-1},g)\omega(g^{-1},g)=1$
and using~\eqref{eq:cocycle_condition_new} in the special case
$(g,h,k)=(g,g^{-1},g)$, one finds that
\begin{equation}\label{eq:alpha-omega-id}
    \alpha_g(\omegat(g^{-1},g)) (\omega(g,g^{-1}) \kronr 1_{n_g})
    =
    \alpha_g(1).
\end{equation}
Taking the $j$th column of~\eqref{eq:alpha-omega-id} gives
\begin{equation}\label{eq:column-id}
    \alpha_g(t_{g,j} \alpha_{g^{-1}}(1)) \omega(g,g^{-1})
    =
    \alpha_g(1)e_j^t.
\end{equation}

\begin{lemma}\label{lem:recover-alpha}
    For each $g\in G$, the map 
    \begin{equation*}
        \alpha_g': R \to M_{n_g}(R),
        \quad
        \alpha_g'(r) := x_g r y_g^t
    \end{equation*}
    coincides with $\alpha_g$.
\end{lemma}
\begin{proof}
Let $g\in G$ and let $r\in R$.
For all $1 \leq i,j \leq n_g$,
\begin{equation*}
    (\alpha_g'(r))_{ij}
    =
    e_i\alpha_g(r) \cdot t_{g,j} \alpha_{g^{-1}}(1)
    =
    e_i\alpha_g(r) \alpha_g(t_{g,j} \alpha_{g^{-1}}(1)) \omega(g,g^{-1})
    \overset{\eqref{eq:column-id}}{=}
    (\alpha_g(r))_{ij}.
\end{equation*}
Hence $\alpha_g'(r)=\alpha_g(r)$.
\end{proof}

\begin{lemma}\label{lem:recover-omega}
    For all $g,h \in G$, the element
    \begin{equation*}
        \omega'(g,h) := (x_g \kronr x_h) y_{gh}^t \in M_{n_g n_h, n_{gh}}(R)
    \end{equation*}
    coincides with $\omega(g,h)$.
\end{lemma}
\begin{proof}
Let $g,h\in G$, and let $1 \le i \le n_g$, $1 \le j \le n_h$, and $1\le k\le n_{gh}$.
Equality is verified by computing the $((i,j),k)$-entry of $\omega'(g,h)$.

As in the proof of Lemma~\ref{lem:multi_1},
\begin{equation*}
    (x_g \kronr x_h)_{(i,j)} 
    = 
    e_i \alpha_g(e_j\alpha_h(1)) \omega(g,h)
    =
    e_{(i,j)} \omega(g,h),
\end{equation*}
and by definition $(y_{gh})_k = t_{gh,k} \alpha_{gh^{-1}}(1) \in S_{gh^{-1}}$.
Hence
\begin{equation*}
    \omega'(g,h)_{((i,j),k)}
    =
    (x_g \kronr x_h)_{(i,j)} \cdot (y_{gh})_k
    =
    e_{(i,j)} \omega(g,h) \alpha_{gh}(t_{gh,k} \alpha_{gh^{-1}}(1)) \omega(gh,gh^{-1}).
\end{equation*}
Applying~\eqref{eq:column-id} with $g=gh$ and $j=k$ yields
\begin{equation*}
    \omega'(g,h)_{((i,j),k)}
    =
    e_{(i,j)} \omega(g,h) \alpha_{gh}(1) e_k^t
    =
    e_{(i,j)} \omega(g,h) e_k^t
    =
    \omega(g,h)_{((i,j),k)}.
\end{equation*}
Since the indices were arbitrary, $\omega'(g,h)=\omega(g,h)$.
\end{proof}

We summarize this construction as follows:

\begin{theorem}\label{thm:facsys}
    Each factor system for $(R,G)$ gives rise to a strongly $G$-graded ring with principal component $R$,
    for which it appears as an associated factor system.
\end{theorem}

We conclude by addressing the classification problem through an extension of the notion of equivalence for factor systems for $(R,G)$, in analogy with Definition~\ref{def:conjugate}:

\begin{definition}\label{def:conjugate_abstract}
    Let $(n,\alpha,\omega)$ and $(n',\alpha',\omega')$ be factor systems for $(R,G)$.
    We say that they are \emph{conjugate} if there exist families of elements
    \begin{equation*}
        v_g \in M_{n'_g,n_g}(R),
        \quad
        w_g \in M_{n_g,n'_g}(R),
        \quad
        g\in G,
    \end{equation*}
    satisfying~\eqref{eq:v1}–\eqref{eq:v3} for all $g,h\in G$ and $s \in R$.
\end{definition}

Taken together with the preceding constructions and results, this yields the following classification statement.

\begin{corollary}\label{cor:class}
    Let $R$ be a unital ring and $G$ a group.
    There is a one-to-one correspondence between conjugacy classes of factor systems for $(R,G)$
    and equivalence classes of strongly $G$-graded rings with principal component $R$.
\end{corollary}

\begin{remark}
    Each factor system determines a Picard homomorphism together with a concrete factor set, thereby placing the classification in
    Corollary~\ref{cor:class} within the well-developed framework of Picard formalism; see~\cite[Cor.~1.3.18 and the surrounding discussion]{NaOy82}.
\end{remark}

\subsection{Factor systems for strongly graded \texorpdfstring{$^\ast$}{Star}-rings}
\label{sec:involution}

In many concrete situations a $G$-graded ring $S = \bigoplus_{g\in G} S_g$ is equipped with a compatible involution ${}^\ast$, in the sense that $(S_g)^\ast = S_{g^{-1}}$ for all $g\in G$.
In this case, $(S,{}^\ast)$ is referred to as a $G$-graded $^\ast$-ring.
A particularly rich source of examples is provided by actions of compact Abelian groups on unital C\Star-algebras, where the corresponding spectral subspaces yield a grading by the unitary dual.

For the remainder of this section, we fix a strongly $G$-graded $^\ast$-ring $(S,{}^\ast)$. 
For $A \in M_{n,m}(S)$, we denote by $A^\dagger \in M_{m,n}(S)$ the matrix obtained by transposing $A$ and applying ${}^\ast$ entrywise.
We further write $A^\ast$ for the matrix obtained by applying ${}^\ast$ entrywise,
without transposition.

A natural question is whether there exists a module frame system for $S$ of the form $(z_g,z_g^\ast)_{g\in G}$.
This motivates the following terminology:

\begin{definition}\label{def:parseval}
    A strongly $G$-graded $^\ast$-ring $(S,{}^\ast)$ has the \emph{algebraic Parseval property} if there exists a module frame system for $S$ of the form $(z_g,z_g^\ast)_{g\in G}$.
\end{definition}

\begin{remark}
    The terminology is inspired by Parseval frames in Hilbert C\Star modules, where an analogous identity characterizes tight frames.
\end{remark}

The following lemma records a simple consequence of the algebraic Parseval property:

\begin{lemma}\label{rem:parseval}
    If $(S,{}^\ast)$ satisfies the algebraic Parseval property, then one may choose a module frame system for $S$ of the form $(z_g,z_g^\ast)_{g\in G}$.
    With this choice, the maps $\alpha_g$, $g \in G$, respect the involution, \ie, $\alpha_g(r^\dagger) = \alpha_g(r)^\dagger$ for all $g \in G$ and $r\in R$, and the generalized cocycle satisfies $\omegat(g,h) = \omega(g,h)^\dagger$ for all $g,h \in G$.
\end{lemma}

The algebraic Parseval property does not hold in general, and can already fail for group rings:

\begin{example}
Consider the ring $\mathbb{C}[t,t^{-1}]$ of complex Laurent polynomials, which is strongly $\mathbb{Z}$-graded with homogeneous components $\mathbb{C} t^n$ for $n\in\mathbb{Z}$.
Equip $\mathbb{C}[t,t^{-1}]$ with the involution determined by $z^\ast := \overline{z}$ for $z \in \mathbb{C}$ and  $t^\ast := -t^{-1}$, so that $(t^{-1})^\ast := -t$.

We claim that $(\mathbb{C}[t,t^{-1}],{}^\ast)$ fails the algebraic Parseval property.
Indeed, let $m\in\mathbb{N}$ and let 
\begin{equation*}
    z \in M_{m,1}(\mathbb{C}[t,t^{-1}]_1) = M_{m,1}(\mathbb{C}t).
\end{equation*}
Writing $z=(q_1,\dots,q_m)^t$ with $q_i \in \mathbb{C}t$, we have $q_i = \lambda_i t$ for some  $ \lambda_i \in \mathbb{C}$, and therefore
\begin{equation*}
    z^\dagger z = \sum_{i=1}^m q_i^\ast q_i
    = \sum_{i=1}^m (\lambda_i t)^\ast (\lambda_i t)
    = \sum_{i=1}^m (-\overline{\lambda_i}t^{-1})(\lambda_i t)
    = -\sum_{i=1}^m |\lambda_i|^2.
\end{equation*}
Thus $z^\dagger z$ is a non-positive real multiple of $1$, and in particular cannot equal $1$.
Consequently, there is no choice of $m$ and $z\in M_{m,1}(\mathbb{C}t)$ with $z^\dagger z=1$, so the algebraic Parseval property fails.
\end{example}

We now characterize when $(S,{}^\ast)$ has the algebraic Parseval property.
Fix a module frame~system $(x_g,y_g)_{g\in G}$ for $S$ and set $p_g := x_g y_g^t \in M_{n_g}(S_e)$.
Then $p_g$ is an idempotent and $p_g x_g = x_g$.

Suppose that $(S,{}^\ast)$ satisfies the algebraic Parseval property, and let $g\in G$.
Then there exists $z_g\in M_{m_g,1}(S_g)$ for some $m_g\in\mathbb{N}$ such that $z_g^\dagger z_g = 1$.
By~\eqref{eq:hom_comp}, each component $(z_g)_i$ can be written as $(z_g)_i = u_i^t x_g$ for some $u_i\in M_{n_g,1}(S_e)$, hence
\begin{equation*}
    z_g = r_g x_g,
    \quad
    r_g := (u_1^t,\ldots,u_{m_g}^t)^t \in M_{m_g,n_g}(S_e).
\end{equation*}
Replacing $r_g$ by $r_g p_g$ does not change $z_g$, so we may and do assume $r_g = r_g p_g$.
A direct~computation gives
\begin{equation*}
    z_g^\dagger z_g 
    = 
    (r_g x_g)^\dagger (r_g x_g)
    = 
    x_g^\dagger r_g^\dagger r_g x_g.
\end{equation*}
Thus $z_g^\dagger z_g = 1$ holds if and only if $x_g^\dagger r_g^\dagger r_g x_g = 1$, and multiplying by $(y_g^t)^\dagger$ on the left and $y_g^t$ on the right yields $r_g^\dagger r_g = (y_g^t)^\dagger y_g^t$.

Conversely, if $r_g\in M_{m_g,n_g}(S_e)$ satisfies $r_g^\dagger r_g = (y_g^t)^\dagger y_g^t$ for some $m_g\in\mathbb{N}$, then $z_g := r_g x_g$ satisfies $z_g^\dagger z_g = 1$.
Therefore:

\begin{lemma}\label{lem:parseval}
    With notation as above, the following are equivalent:
    \begin{enumerate}[label=(\alph*)]
        \item 
            $(S,{}^\ast)$ has the algebraic Parseval property.
        \item 
            For each $g\in G$, the element $(y_g^t)^\dagger y_g^t \in M_{n_g}(S_e)$ admits a factorization
            \begin{equation*}
                (y_g^t)^\dagger y_g^t = r_g^\dagger r_g
            \end{equation*}
            for some $m_g\in\mathbb{N}$ and $r_g\in M_{m_g,n_g}(S_e)$.
    \end{enumerate}
\end{lemma}

\begin{remark}
    By Lemma~\ref{lem:factorsystem}~\emph{(\ref{it:conjugate})}, different choices of module frame systems for $S$ yield conjugate elements $(y_g^t)^\dagger y_g^t$.
    In particular, the validity of condition~\emph{(b)} in Lemma~\ref{lem:parseval} is independent of the particular choice of module frame system.
\end{remark}
%\begin{proof}
%Let $g \in G$.
%If $(x_g',y_g')_{g \in G}$ is another choice of a module frame system for $S$ giving rise to aconjugate factor system, then by Lemma~3.6\emph{(5)} one has $y_g'{}^\ast y_g' = w_g^\ast (y_g^\ast y_g) w_g$ for suitable elements $w_g \in M_{n_g,n_g'}(S_e)$. 
%Thus condition~\emph{(b)} for $y_g^\ast y_g$ holds if and only if it holds for $y_g'{}^\ast y_g'$, and the conclusion follows from Lemma~\ref{lem:frame-vs-matrix}.
%\end{proof}

\begin{corollary}\label{cor:parseval}
    Let $R$ be a ring, let $G$ be a group, let $S := R \rtimes_{(\alpha,\omega)} G$ be the crossed product associated with a factor system $(\alpha,\omega)$ (see~\eqref{eq:coaction_condition}--\eqref{eq:crossed_product}), and suppose that ${}^\ast$ is a compatible involution.
    Then $(S,{}^\ast)$ satisfies the algebraic Parseval property if and only if, for each $g\in G$, there exists $r_g \in R$ such that $r_g^\ast r_g = (u_g^{-1})^\ast u_g^{-1}$.
\end{corollary}

\begin{remark}
    The condition in Lemma~\ref{lem:parseval}~\emph{(b)} holds automatically whenever $S_e$ embeds as a unital $C^\ast$-subalgebra of a $C^\ast$-algebra $A$ containing $S$ as a unital ${}^\ast$-subalgebra.
    In this case $y_g\in M_{1,n_g}(A)$, so $y_g^\ast y_g$ is positive in $M_{n_g}(A)$ and lies in $M_{n_g}(S_e)$ by strong grading.
\end{remark}

\subsubsection{Complex Leavitt path algebras}\label{sec:clpa}

The following lemmas describe the adjoint structure in strongly graded complex Leavitt path algebras and provide explicit ``Parseval elements'' in each homogeneous component.

\begin{lemma}\label{lem:xn.from.x1}
    Let $L_\complex(E)$ be a unital Leavitt path algebra.
    Let $x_1\in M_{n,1}(L_\complex(E))$ and define $x_{n+1} := x_1 \kronr x_n$ for $n \ge 1$. 
    If $x_1^\dagger x_1 =1$, then $x_n^\dagger x_n=1$ for all $n \ge 1$.
\end{lemma}
\begin{proof}
Let $x_1\in M_{n,1}(L_\complex(E))$ such that $x_1^\dagger x_1 = 1$, 
%For $u,v\in M_{1,n}(L_\complex(E))$ and $w,\tilde w\in M_{1,m}(L_\complex(E))$, one verifies that
%\begin{equation}\label{eq:uvdaggerkron}
%    (u\kronr w)^\dagger(v\kronr \tilde w)
%    = 
%    w^\dagger (u^\dagger v)\tilde w .
%\end{equation}
and assume $x_n^\dagger x_n = 1$ for some $n \ge 1$. 
Then
\begin{equation*}
    x_{n+1}^\dagger x_{n+1}
    =
    (x_1\kronr x_n)^\dagger(x_1\kronr x_n)
    = x_n^\dagger (x_1^\dagger x_1) x_n
    = x_n^\dagger x_n = 1,
\end{equation*}
and the claim follows by induction.
\end{proof}

\begin{lemma}\label{lem:xdaggerx.eeast}
    Let $L_\complex(E)$ be a strongly graded unital Leavitt path algebra.
    For each $e \in E_1$ and each integer $n \ge 1$ there exist an integer $\tilde n \ge 1$ and $x \in M_{\tilde n,1}(L_\complex(E)_n)$ such that $x^\dagger x = ee^\ast$.
\end{lemma}
\begin{proof}
Since $L_\complex(E)$ is unital and strongly graded, the vertex set $E_0$ is finite and has no sinks. 
Consequently, every vertex connects to a cycle by a real path (following edges indefinitely forces a repetition in the finite graph).
%Moreover, if $v \in E_0$ lies on a cycle and $N \ge 1$, there exists a real path $\alpha$ with $|\alpha| = N$ and $r(\alpha) = v$, obtained by traversing the cycle sufficiently many times.

For $v\in E_0$, define $d(v)$ as the maximal length of a cycle-avoiding
real path starting at $v$, with $d(v)=0$ if $v$ lies on a cycle.
Then $d(v)$ is finite for all $v\in E_0$, because a cycle-avoiding~path visits distinct vertices and the vertex set $E_0$ is finite.

The statement is proved by induction on $d(s(e))$:

\medskip
\noindent
\emph{Case $d(s(e))=1$.}
Then $r(e)$ lies on a cycle. 
Since vertices on a cycle admit return paths of arbitrary length, for every $n \ge 1$ there exists a real path $\alpha$ with $|\alpha| = n+1$ and $r(\alpha) = r(e)$. 
Set $x = \alpha e^\ast$.
Then $x \in L_\complex(E)_n = M_{1,1}(L_\complex(E)_n)$ and
\begin{equation*}
    x^\dagger x
    =
    (\alpha e^\ast)^\ast \alpha e^\ast
    =
    e\alpha^\ast\alpha e^\ast
    =
    er(e)e^\ast
    =ee^\ast.
\end{equation*}

\medskip
\noindent
\emph{Induction step.}
Assume the statement holds whenever $d(s(e)) \le N$.
Let $e \in E_1$ with $d(s(e)) = N+1$.
Then $d(r(e)) \le N$.
For each $f\in s^{-1}(r(e))$ and integer $n \ge 1$, the induction hypothesis yields $\tilde n \ge 1$ and $x_f \in M_{\tilde n_f,1}(L_{n+1}(E))$ such that $x_f^\dagger x_f = ff^\ast$.

Since $E$ is row-finite, the set $s^{-1}(r(e))$ is finite; write $s^{-1}(r(e)) = \{f_1,\dots,f_m\}$ for some integer $m \ge 1$.
Then
\begin{equation*}
    x=(x_{f_1}e^\ast,\ldots,x_{f_m}e^\ast)^t
\in M_{\tilde n_{f_1}+\ldots+\tilde n_{f_m},1}(L_N(E)),
\end{equation*}
and hence
\begin{equation*}
    x^\dagger x
    =\sum_{i=1}^m e x_{f_i}^\dagger x_{f_i} e^\ast
    =\sum_{i=1}^m e f_i f_i^\ast e^\ast
    =e \parab{\hspace{-2mm}\sum_{f\in s^{-1}(r(e))}\hspace{-4mm}ff^\ast} e^\ast
    =e r(e) e^\ast = e e^\ast.
\end{equation*}

\medskip
\noindent
\emph{Cycle case.}
If $s(e)$ lies on a cycle, the previous argument applies when $r(e)$ does not lie on a cycle.
If both lie on cycles, the initial construction $x = \alpha e^\ast$ (with $r(\alpha)=r(e)$) again yields $x^\dagger x = e e^\ast$.
%Finally, let $v\in E_0$ be a vertex in a cycle and assume that $e\in E_1$ such that $s(e)=v$. If $r(e)$ is in a cycle, we can find  $\alpha\in P(E)$ such that $r(\alpha)=r(e)$ and $|\alpha|=N+1$. Setting $x=\alpha e^\ast$ it follows that $(\alpha e^\ast)^\ast\alpha e^\ast=ee^\ast$. If $r(e)$ is not in a cycle, i.e. $d(r(e))\geq 1$, one can use the previous induction argument to find $x\in M_{\tilde{N},1}(L_{N}(E))$ such that $x^\dagger x = ee^\ast$. Namely, let $s^{-1}(r(e))=\{f_1,\ldots,f_M\}$ and let $x_i\in M_{\tilde{N}_i,1}(L_{N+1}(E))$ be such that $x_i^\dagger x_i=f_if_i^\ast$. Setting $x=(x_1e^\ast\,\, x_2e^\ast\,\,\cdots\,\, x_Me^\ast)^t$ one finds that $x^\dagger x = ee^\ast$. 
\end{proof}

\begin{corollary}\label{cor:xdaggerx.vertex}
    Let $L_\complex(E)$ be a strongly graded unital Leavitt path algebra.
    For each $v \in E_0$ and each integer $n \ge 1$ there exist an integer $\tilde n \ge 1$ and $x\in M_{\tilde n,1}(L_\complex(E)_n)$ such that $x^\dagger x = v$.
\end{corollary}
\begin{proof}
Let $v \in E_0$ and let $n \ge 1$. 
Since $E$ is row-finite, the set of edges emitted by $v$ is finite; write $s^{-1}(v)=\{f_1,\dots,f_m\}$ for some integer $m\ge 1$. 
By Lemma~\ref{lem:xdaggerx.eeast}, for each $1\le i\le m$ there exists $x_i \in M_{\tilde n_i,1}(L_\complex(E)_n)$ such that $x_i^\dagger x_i = f_i f_i^\ast$.
Define
\begin{equation*}
    x := (x_1,\ldots,x_m)^t \in M_{\tilde n_1+\cdots+\tilde n_m,1}(L_\complex(E)_n),
\end{equation*}
Then
\begin{equation*}
    x^\dagger x
    = \sum_{i=1}^m x_i^\dagger x_i
    = \sum_{f\in s^{-1}(v)} ff^\ast
    = v.
    \qedhere
\end{equation*}
\end{proof}

\begin{theorem}\label{thm:xdaggerx.eeast}
    A unital Leavitt path algebra $L_\complex(E)$ has the algebraic Parseval property if and only if it is strongly graded.
\end{theorem}

\begin{proof}
Let $L_\complex(E)$ be a unital Leavitt path algebra. 
If $L_\complex(E)$ has the algebraic Parseval property, then it admits a module frame system; by Lemma~\ref{lem:dade}, this implies that $L_\complex(E)$ is strongly graded.

Conversely, suppose that $L_\complex(E)$ is strongly graded. 
To establish the algebraic Parseval property, it must be shown that for each $n \in \integers$ there exist an integer $\tilde n \ge 1$ and $x_n \in M_{\tilde n,1}(L_\complex(E)_n)$ such that $x_n^\dagger x_n = 1$.

For $n=0$ one may take $x_0=1$. 
If $n<0$, Lemma~\ref{lem:neg.degree.x} yields a column vector $x_n$ whose entries are all ghost paths of length $-n$, and $x_n^\dagger x_n=1$.
Now let $n>0$ and~write $E_0 = \{v_1,\dots,v_m\}$. 
By Corollary~\ref{cor:xdaggerx.vertex}, for each $1 \le i \le m$ there exists $\tilde n_i \ge 1$ and $x_i \in M_{\tilde n_i,1}(L_\complex(E)_n)$ such that $x_i^\dagger x_i = v_i $.
Define
\begin{equation*}
    x_n := (x_1,\ldots,x_m)^t \in M_{\tilde n_1+\cdots+\tilde n_m,1}(L_\complex(E)_n).
\end{equation*}
Then
\begin{equation*}
    x_n^\dagger x_n
    =\sum_{i=1}^m x_i^\dagger x_i
    =\sum_{i=1}^m v_i = 1.
\end{equation*}
Thus $L_\complex(E)$ has the algebraic Parseval property.
\end{proof}

\subsubsection*{Frame description of the involution}

In the rest of this section, suppose that $(S,{}^\ast)$ has the algebraic Parseval property and fix a module frame $(z_g,z_g^\dagger)_{g\in G}$, which will be used to give an explicit description of the involution.
Let $(n,\alpha,\omega)$ be the associated factor system.

For each $g\in G$, expand $z_g$ with respect to the standard column basis
$f_i\in M_{n_g,1}(S_e)$:
\begin{equation*}
    z_g = \sum_{i=1}^{n_g} (v_i^t z_g) f_i,
\end{equation*}
where $v_i^t \in M_{1,n_g}(S_e) \alpha_g(1)$.
The condition $z_g^\dagger z_g = 1$ then takes the form
\begin{equation}
    \sum_{i=1}^{n_g} u_i^t \alpha_{g^{-1}}(v_i)^t \omega(g^{-1},g) = 1,\label{eq:involution}
\end{equation}
where $u_i^t := z_g^\dagger v_i^\ast z_{g^{-1}}^\dagger \in M_{1,n_{g^{-1}}}(S_e) \alpha_{g^{-1}}(1)$ for $i=1,\dots,n_g$.

\begin{lemma}\label{lem:involution}
    With notation as above, for each $g \in G$ and all $w^t \in  M_{1,n_g}(S_e) \alpha_g(1)$, one has
    \begin{equation*}
        (w^t z_g)^\ast = J_g(w^t) z_{g^{-1}},
        \quad
        J_g(w^t) := \sum_{i=1}^{n_g} u_i^t \alpha_{g^{-1}}(v_i^t w^\ast) \in M_{1,n_{g^{-1}}}(R)\alpha_{g^{-1}}(1).
    \end{equation*}
\end{lemma}
\begin{proof}
Let $g \in G$ and let $w^t \in  M_{1,n_g,}(S_e) \alpha_g(1)$.
Then
\begin{equation*}
    J_g(w^t) z_{g^{-1}} = \sum_{i=1}^{n_g} u_i^t z_{g^{-1}} (v_i^t w^\ast) = \sum_{i=1}^{n_g} ((u_i^t z_{g^{-1}}) (v_i^t z_g)) (z_g^\dagger w^\ast) \overset{\eqref{eq:involution}}{=} z_g^\dagger w^\ast = (w^t z_g)^\dagger = (w^t z_g)^\ast,
\end{equation*}
    which proves the claim.
\end{proof}

\section{Lifting derivations}\label{sec:lifting}

Let $S$ be a strongly $G$-graded ring with principal component $R$, and let
$\delta: R \to R$ be a derivation.
In this section we study the problem of \emph{lifting} $\delta$ to a graded
derivation $\hd: S \to S$, \ie, a graded derivation satisfying
\begin{equation*}
    \hd \vert_{S_e} = \delta.
\end{equation*}

In general, such a lift does not exist, and its existence is constrained by the graded structure of~$S$.
Using a factor system associated with $S$, we formulate explicit compatibility conditions characterizing when $\delta$ admits a graded lift.

Notably, inner derivations of $R$ always admit graded lifts to $S$.
Thus, the obstructions~identified in Theorem~\ref{thm:liftder} below are genuinely outer phenomena.

Throughout this section, a derivation $\delta: R \to R$ is extended to
matrix algebras over $R$ by acting entrywise.

\begin{lemma}\label{lem:liftder1}
    Let $S$ be a ring and let $\delta: S \to S$ be a derivation.
    If $p \in S$ is an idempotent, then $p\delta(p)p=0$.
\end{lemma}
\begin{proof}
	The Leibniz rule yields $\delta(p)=\delta(p^2)=p\delta(p)+\delta(p)p$. 
 Hence $p\delta(p)p=p\delta(p)p+p\delta(p)p$, and so $p\delta(p)p=0$ as claimed.
\end{proof}

\begin{lemma}\label{lem:liftder2}
    Let $S$ be a strongly $G$-graded ring with principal component $R$, let $(n,\alpha,\omega)$ be a factor system of $S$, and let $\delta: R \to R$ be a derivation.
	Then, for all $g,h \in G$, the following identities hold:
	\begin{enumerate}
		\item 
			$\alpha_g(1) \delta(\alpha_g(1)) \omega(g,h) = 0$,\label{it:liftder2simple}
		\item
			$\omega(g,h) \delta(\alpha_{gh}(1)) \alpha_{gh}(1) = 0$.\label{it:liftder2double}
	\end{enumerate}
\end{lemma}
\begin{proof}
Let $g,h \in G$.
Since $\omega(g,h)=\alpha_g(1)\omega(g,h)$, applying the Leibniz rule yields
\begin{equation*}
    \delta(\omega(g,h)) = \delta(\alpha_g(1)) \omega(g,h) + \alpha_g(1) \delta(\omega(g,h)).
\end{equation*}
Left multiplication by $\alpha_g(1)$ gives $\alpha_g(1) \delta(\alpha_g(1)) \omega(g,h)=0$, proving the first identity.

The proof of second identity is analogous, using the identity
$\omega(g,h)=\omega(g,h) \alpha_{gh}(1)$.
\end{proof}

\begin{theorem}\label{thm:liftder}
    Let $S$ be a strongly $G$-graded ring with principal component $R$, let
    $(n,\alpha,\omega)$ be a factor system associated with a module frame system  $(x_g,y_g)_{g \in G}$ for $S$, and let $\delta: R \to R$ be a derivation.
    Then the following assertions hold:
    \begin{enumerate}
        \item
            Let $\hd: S \to S$ be a graded derivation lifting $\delta$.
            For each $g \in G$, define
            \begin{equation*}
                \eta(g) := \hd(x_g) y_g^t \in M_{n_g}(R)\alpha_g(1).
           %     \quad
            %    \heta(g) := x_g \hd(y_g^t) \in \alpha_g(1) M_{n_g}(R).
            \end{equation*}
            Then, for all $g \in G$ and $r \in R$,
            \begin{equation}
                [\delta,\alpha_g](r)
                =
                [\eta(g),\alpha_g(r)]
                +
                \alpha_g(r)\delta(\alpha_g(1)).
                \label{eq:deriv_cond_1}
            \end{equation}
%            \begin{equation}
%                \delta(\alpha_g(r))
%                = 
%                \eta(g) \alpha_g(r) + \alpha_g(\delta(r)) + \alpha_g(r) (\delta(\alpha_g(1)) - \eta(g)),
%                \label{eq:deriv_cond_1}
%            \end{equation}
            and, for all $g,h \in G$,
            \begin{equation}
                \hspace{3em}
                \delta(\omega(g,h))
                =
                (\eta(g) \kronr 1_{n_h}) \omega(g,h) + \alpha_g(\eta(h)) \omega(g,h) + \omega(g,h) (\delta(\alpha_{gh}(1)) - \eta(gh))
                \label{eq:deriv_cond_2}
            \end{equation}
        \item
            Conversely, suppose that for each $g \in G$ there exist $\eta(g) \in M_{n_g}(R)$ with $\eta(e)=0$ such that~\eqref{eq:deriv_cond_1} holds for all $g \in G$ and $r \in R$, and~\eqref{eq:deriv_cond_2} holds for all $g,h \in G$.
            Define a map $\hd: S \to S$ by additively extending
            \begin{equation}
                \hd(u^t x_g) := \delta(u)^t x_g + u^t \eta(g) x_g
                \label{eq:deriv_formula}
            \end{equation}
            for all $g \in G$ and $u \in M_{n_g,1}(R)$.
            Then $\hd$ is a graded derivation lifting $\delta$.
    \end{enumerate}
\end{theorem}
\begin{proof}%[Proof of Theorem~\ref{thm:liftder}]
(1)
Let $g \in G$.
Since $\hd$ preserves homogeneous components, $\eta(g) \in M_{n_g}(R)$.
Moreover, 
\begin{equation*}
    \eta(g) \alpha_g(1) = \hd(x_g)y_g^t x_gy_g^t = \hd(x_g) y_g^t = \eta(g),
\end{equation*}
so $\eta(g)\in M_{n_g}(R)\alpha_g(1)$. %; similarly, $\heta(g)\in \alpha_g(1)M_{n_g}(R)$.

For $g \in G$ and $r \in R$, applying the Leibniz rule to
$\alpha_g(r) = x_g r y_g^t$ yields
\begin{align*}
    \delta(\alpha_g(r))
    &=
    \hd(x_g) r y_g^t + x_g \delta(r) y_g^t + x_g r \hd(y_g^t)
    \\
    &= 
    \eta(g)\alpha_g(r) + \alpha_g(\delta(r)) + \alpha_g(r) x_g \hd(y_g^t)
    \\
    &= \eta(g) \alpha_g(r) + \alpha_g(\delta(r)) + \alpha_g(r) (\delta(\alpha_g(1)) - \eta(g)).
\end{align*}
Rearranging terms, this is equivalent to
\begin{equation*}
    [\delta,\alpha_g](r)
    =
    [\eta(g),\alpha_g(r)]
    +
    \alpha_g(r)\delta(\alpha_g(1)),
\end{equation*}
which proves~\eqref{eq:deriv_cond_1}.
  
For $g,h \in G$, using $\omega(g,h)=(x_g\kronr x_h)y_{gh}^t$ and the Leibniz rule, a standard computation yields
\begin{align*}
    \delta(\omega(g,h))
    &=
    (\hd(x_g) \kronr x_h) y_{gh}^t + (x_g \kronr \hd(x_h)) y_{gh}^t + (x_g \kronr x_h) \hd(y_{gh}^t)
    \\
    &=
    (\eta(g) x_g \kronr x_h) y_{gh}^t + (x_g \kronr (\eta(h) x_h)) y_{gh}^t + \omega(g,h) (\delta(\alpha_{gh}(1)) - \eta(gh))
    \\
    &=
    (\eta(g) \kronr 1_{n_h}) \omega(g,h) + \alpha_g(\eta(h)) (x_g \kronr x_h) y_{gh}^t +  \omega(g,h) (\delta(\alpha_{gh}(1)) - \eta(gh))
    \\
    &=
    (\eta(g) \kronr 1_{n_h}) \omega(g,h) + \alpha_g(\eta(h)) \omega(g,h) + \omega(g,h) (\delta(\alpha_{gh}(1)) - \eta(gh)),
\end{align*}
which establishes~\eqref{eq:deriv_cond_2}.

(2)
The map $\hd$ defined by~\eqref{eq:deriv_formula} is well defined.
Indeed, if $u^t x_g=0$ for some $g\in G$ and~$u\in M_{n_g,1}(R)$, then
$u^t \alpha_g(1) = 0$, and applying $\delta$ to $u^t\alpha_g(1)=0$ yields
\begin{equation*}
    \delta(u)^t \alpha_g(1) + u^t \delta(\alpha_g(1)) = 0.
\end{equation*}
Since $\delta(\alpha_g(1)) = \eta(g) \alpha_g(1) + \alpha_g(1)(\delta(\alpha_g(1))-\eta(g))$, it follows that
\begin{equation*}
    \delta(u)^t \alpha_g(1) + u^t \eta(g) \alpha_g(1) = 0, 
\end{equation*}
and multiplying on the right by $x_g$ completes the argument.

Clearly, $\hd$ is graded, and $\eta(e)=0$ ensures that $\hd$ lifts $\delta$.
It therefore remains to show that $\hd$ is a derivation, which reduces to verifying
the Leibniz rule on products of homogeneous elements.
Let $g,h \in G$ and let $u \in M_{n_g,1}(R)$ and $v \in M_{n_h,1}(R)$.
Then
\begin{align*}
    \hd((u^t x_g) (v^t x_h)) 
    &=
    \hd(u^t \alpha_g(v)^t \omega(g,h) x_{gh})
    \\
    &=
    \delta(u^t \alpha_g(v)^t \omega(g,h)) x_{gh}
    + u^t \alpha_g(v)^t \omega(g,h) \eta(gh) x_{gh}
    \\
    &=
    \delta(u)^t \alpha_g(v)^t \omega(g,h) x_{gh} 
    + u^t \delta(\alpha_g(v))^t \omega(g,h) x_{gh}
    + u^t \alpha_g(v)^t \delta(\omega(g,h)) x_{gh} 
    \\
    &\quad 
    + u^t \alpha_g(v)^t \omega(g,h) \eta(gh) x_{gh}.
\end{align*}
Applying~\eqref{eq:deriv_cond_1} and~\eqref{eq:deriv_cond_2} to the second and third summands, respectively, yields
\begin{align*}
    \hd((u^t x_g) (v^t x_h)) 
    &=
    \delta(u)^t \alpha_g(v)^t \omega(g,h) x_{gh}
    + u^t \eta(g) \alpha_g(v)^t \omega(g,h) x_{gh}
    + u^t \alpha_g(\delta(v)^t) \omega(g,h) x_{gh}
    \\
    &\quad
    + u^t \alpha_g(v^t \eta(h)) \omega(g,h) x_{gh} 
    + u^t \alpha_g(v)^t (\delta(\alpha_g(1)) \kronr 1_{n_h}) \omega(g,h) x_{gh}
    \\
    &\quad
    + u^t \alpha_g(v)^t \omega(g,h) \delta(\alpha_{gh}(1)) x_{gh}.
\end{align*}
By Lemma~\ref{lem:liftder2}, the last two terms vanish, so that
\begin{align*}
    \hd((u^t x_g) (v^t x_h)) 
    &=
    \delta(u)^t \alpha_g(v)^t \omega(g,h) x_{gh}
    + u^t \eta(g) \alpha_g(v)^t  \omega(g,h) x_{gh}
    + u^t \alpha_g(\delta(v)^t) \omega(g,h) x_{gh}
    \\
    &\quad
    + u^t \alpha_g(v^t \eta(h)) \omega(g,h) x_{gh}.
\end{align*}
On the other hand,
\begin{align*}
    &\hd(u^t x_g) (v^t x_h) + (u^t x_g) \hd(v^t x_h)
    %&\overset{\eqref{eq:deriv_formula}}{=}
    =
    (\delta(u)^t x_g + u^t \eta(g) x_g) (v^t x_h) + (u^t x_g) (\delta(v)^t x_h + v^t \eta(h) x_h)
    \\
    &=
    (\delta(u)^t \alpha_g(v)^t + u^t \eta(g) \alpha_g(v)^t + u^t \alpha_g(\delta(v)^t) + u^t \alpha_g(v)^t \alpha_g(\eta(h))) \omega(g,h) x_{gh},
\end{align*}
which coincides with the previous expression.
Thus $\hd$ satisfies the Leibniz rule.
\end{proof}

\begin{remark}\label{lem:commH}
    Under the hypotheses of Theorem~\ref{thm:liftder}, if $S$ is commutative and $\delta: R \to R$ is a derivation, then~\eqref{eq:deriv_cond_1} holds for the choice $\eta(g) := \alpha_g(1) \delta(\alpha_g(1))$, $g \in G$.
    Indeed, in the commutative case one has $\alpha_g(r) = r\alpha_g(1)$ for all $r\in R$, and \eqref{eq:deriv_cond_1} therefore follows from the idempotent identity
    \begin{equation*}
        \delta(\alpha_g(1)) = \alpha_g(1)\delta(\alpha_g(1)) + \delta(\alpha_g(1))\alpha_g(1).
    \end{equation*}
    together with Lemma~\ref{lem:liftder1}.
\end{remark}

\begin{corollary}\label{cor:liftder_crossed}
    Let $R$ be a ring, let $G$ be a group, let $S := R \rtimes_{(\alpha,\omega)} G$ be the crossed product associated with a factor system $(\alpha,\omega)$ (see~\eqref{eq:coaction_condition}--\eqref{eq:crossed_product}), and let $\delta: R \to R$ be a derivation.
    Then the following assertions hold:
    \begin{enumerate}
        \item
            Let $\hd: S \to S$ be a graded derivation lifting $\delta$.
            For each $g \in G$, define
            \begin{equation*}
                \eta(g) := \hd(u_g) u_{g^{-1}} \in R.
            \end{equation*}
            Then, for all $g \in G$ and $r \in R$,
            \begin{equation}\label{eq:deriv_cond_1_crossed}
                [\delta,\alpha_g](r)
                =
                [\eta(g),\alpha_g(r)],
            \end{equation}
            and, for all $g,h \in G$,
            \begin{equation}\label{eq:deriv_cond_2_crossed}
                \delta(\omega(g,h))
                =
                \eta(g) \omega(g,h) + \alpha_g(\eta(h)) \omega(g,h) - \omega(g,h) \eta(gh).
            \end{equation}
        \item
            Conversely, suppose that there exists a map $\eta: G \to R$ with $\eta(e)=0$ such that~\eqref{eq:deriv_cond_1_crossed} holds for all $g \in G$ and $r \in R$, and~\eqref{eq:deriv_cond_2_crossed} holds for all $g,h \in G$.
            Define a map $\hd: S \to S$ by additively extending
            \begin{equation}\label{eq:deriv_formula_crossed}
                \hd(r u_g) := \delta(r) u_g + r \eta(g) u_g
            \end{equation}
            for all $g \in G$ and $r \in R$.
            Then $\hd$ is a graded derivation lifting $\delta$.
    \end{enumerate}
\end{corollary}

\begin{remark}\label{rem:primary_obstruction}
    Under the hypotheses of Corollary~\ref{cor:liftder_crossed}, if $\delta$ admits a
    graded lift, then for each $g \in G$ the derivation $\alpha_g^{-1} \circ \delta \circ \alpha_g - \delta$ is inner by~\eqref{eq:deriv_cond_1_crossed}.
    Equivalently, the class of $\delta$ in $\der(R)/\inn(R)$ is fixed under the action of $G$ induced by~$\alpha$.
\end{remark}

Henceforth, we refer to the case $\omega(g,h)=1$ for all $g,h\in G$
as the skew group ring case.

\begin{remark}\label{rem:skew}
    Under the hypotheses of Corollary~\ref{cor:liftder_crossed}, 
    in the skew group ring case~\eqref{eq:deriv_cond_2_crossed} reduces~to
    \begin{equation*}
        \eta(gh) = \eta(g) + \alpha_g(\eta(h)),
    \end{equation*}
    so that $\eta$ is a \emph{crossed $R$-valued homomorphisms on $G$}.
\end{remark}

\begin{example}\label{ex:skew_lift}
    Let $S = R \rtimes_\alpha G$ be a skew group ring.
    Suppose that $\delta \in \der(R)$ satisfies $\delta \circ \alpha_g = \alpha_g \circ \delta$ for all $g \in G$.
    Then~\eqref{eq:deriv_cond_1_crossed} holds with $\eta(g)=0$.
    Hence, by Corollary~\ref{cor:liftder_crossed} and Remark~\ref{rem:skew}, $\delta$ admits a graded lift
    $\hd$ given for $g \in G$ and $r \in R$ by $\hd(r u_g) = \delta(r) u_g$.
\end{example}

\begin{example}\label{ex:quantum_torus}
    Let $q\in\mathbb C^\times$ and let $A_q$ be the (discrete) quantum torus generated by invertible elements $u^{\pm1}, v^{\pm1}$ subject to $uv=qvu$.
    Then $A_q$ is a skew group ring
    \begin{equation*}
        A_q \cong \mathbb C[u,u^{-1}] \rtimes_\alpha \mathbb Z,
        \quad
        \alpha_1(u) := qu.
    \end{equation*}

    Let $\delta$ be the canonical derivation on $\mathbb C[u,u^{-1}]$ given by $\delta(u^n) = n u^n$, $n \in \mathbb Z$.
    Since $\delta$ commutes with the action $\alpha$, it follows from Example~\ref{ex:skew_lift} that $\delta$ admits a graded lift $\hd \in \der(A_q)$, explicitly determined by $\hd(u)=u$ and $\hd(v)=0$. 
\end{example}

For the next result, recall that a derivation $\delta$ on a $^\ast$-ring $(S,{}^\ast)$ is called a \emph{${}^\ast$-derivation} if $\delta(s^\ast) = \delta(s)^\ast$ for all $s \in S$.

\begin{corollary}
\label{cor:star_derivations}
    Let $(S,{}^\ast)$ be a strongly $G$-graded $^\ast$-ring with principal component $R$
    satisfying the algebraic Parseval property.
    Choosing a module frame system for $S$ of the form $(z_g,z_g^\ast)_{g\in G}$,
    Theorem~\ref{thm:liftder} extends to ${}^\ast$-derivations.
    More precisely, for a ${}^\ast$-derivation $\delta: R \to R$ the following assertions hold:
    \begin{enumerate}
        \item
            Let $\hd: S \to S$ be a graded ${}^\ast$-derivation lifting $\delta$.
            For each $g \in G$, define
            \begin{equation*}
                \eta(g) := \hd(z_g) z_g^\dagger \in M_{n_g}(R)\alpha_g(1).
            \end{equation*}
            Then, for all $g \in G$ and $r \in R$,
            \begin{equation}
                \delta(\alpha_g(r))
                = 
                \eta(g) \alpha_g(r) + \alpha_g(\delta(r)) + \alpha_g(r) \eta(g)^\dagger,
                \label{eq:deriv_cond_1_parseval}
            \end{equation}
            and, for all $g,h \in G$,
            \begin{equation}
                \hspace{3em}
                \delta(\omega(g,h))
                =
                (\eta(g) \kronr 1_{n_h}) \omega(g,h) + \alpha_g(\eta(h)) \omega(g,h) + \omega(g,h) \eta(gh)^\dagger
                \label{eq:deriv_cond_2_parseval}
            \end{equation}
        \item
            Conversely, suppose that for each $g \in G$ there exist $\eta(g) \in M_{n_g}(R)$ with $\eta(e)=0$ such that~\eqref{eq:deriv_cond_1_parseval} holds for all $g \in G$ and $r \in R$, and~\eqref{eq:deriv_cond_2_parseval} holds for all $g,h \in G$.
            Define a map $\hd: S \to S$ by additively extending
            \begin{equation}
                \hd(u^t z_g) := \delta(u)^t z_g + u^t \eta(g) z_g
                %\label{eq:deriv_formula_parseval}
            \end{equation}
            for all $g \in G$ and $u^t \in M_{n_g,1}(R) \alpha_g(1)$.
            Then $\hd$ is a graded ${}^\ast$-derivation lifting $\delta$.
    \end{enumerate}
\end{corollary}
\begin{proof}
\eqref{eq:deriv_cond_1_parseval} and \eqref{eq:deriv_cond_2_parseval} follow from the same argument as in the proof of Theorem~3.3. 
In that proof the expressions are rearranged to obtain the stated form; in the present ${}^\ast$-ring setting the symmetric form on the right-hand side is retained. 
The remainder of the argument is analogous, and it only remains to verify that the lift preserves the involution. 
Recall that the factor system associated with $(z_g,z_g^\dagger)$ is ${}^\ast$-compatible (see~Lemma~\ref{rem:parseval}).

Let $g \in G$ and let $w^t \in M_{1,n_g}(R) \alpha_g(1)$.
It suffices to show that
\begin{equation*}
    \hd((w^t z_g)^\ast) = \hd(w^t z_g)^\ast.
\end{equation*}
By Lemma~\ref{lem:involution},
\begin{align*}
    \hd\bigl((w^t z_g)^\ast\bigr)
    &= 
    \hd(J_g(w^t) z_{g^{-1}})
    = \delta(J_g(w^t)) z_{g^{-1}} + J_g(w^t)\eta(g^{-1}) z_{g^{-1}}, 
    \\
    \hd(w^t z_g)^\ast
    &= 
    (\delta(w^t) z_g + w^t\eta(g) z_g)^\ast
    = 
    J_g(\delta(w^t)) z_{g^{-1}} + J_g(w^t\eta(g)) z_{g^{-1}}.
\end{align*}
Thus, to show that these two expressions are equal, it is enough to establish that
\begin{equation}\label{eq:star_check}
    \delta(J_g(w^t)) + J_g(w^t)\eta(g^{-1})
    = 
    J_g(\delta(w^t)) + J_g(w^t\eta(g)) + J_g(w^t)\delta(\alpha_{g^{-1}}(1)).
\end{equation}
%Indeed, multiplying by $z_{g^{-1}}$ kills the last term:
Indeed, multiplying by $z_{g^{-1}}$ eliminates the last term:
%Indeed, the extra term disappears after multiplying by $z_{g^{-1}}$, since
\begin{equation*}
    J_g(w^t)\delta(\alpha_{g^{-1}}(1)) z_{g^{-1}}
    =
    J_g(w^t)\alpha_{g^{-1}}(1)\delta(\alpha_{g^{-1}}(1))\alpha_{g^{-1}}(1) z_{g^{-1}}
    = 0,
\end{equation*}
where the last equality uses Lemma~\ref{lem:liftder1}.

To establish~\eqref{eq:star_check}, recall that
\begin{equation*}
    J_g(w^t) = \sum_{i=1}^{n_g} u_i^t \alpha_{g^{-1}}(v_i^t w^\ast),
\end{equation*}
where $v_1^t,\ldots,v_{n_g}^t \in M_{1,n_g}(S_e) \alpha_g(1)$ and $u_1^t,\ldots u_{n_g}^t \in M_{1,n_{g^{-1}}}(S_e) \alpha_{g^{-1}}(1)$ satisfy
\begin{equation}
    \sum_{i=1}^{n_g} u_i^t \alpha_{g^{-1}}(v_i)^t \omega(g^{-1},g) = 1.\label{eq:product_recall}
\end{equation}

Applying $\delta$ to $J_g(w^t)$ and using~\eqref{eq:deriv_cond_1_parseval}, one obtains
\begin{align*} 
    \delta(J_g(w^t)) &= \sum_{i=1}^{n_g} \big( \delta(u_i)^t \alpha_{g^{-1}}(v_i^t w^\ast) + u_i^t \alpha_{g^{-1}}(\delta(v_i)^t w^\ast) + (u_i^t \eta(g^{-1})) \alpha_{g^{-1}}(v_i^t w^\ast) \big) 
    \\ 
    &\quad 
    + J_g(\delta(w^t)) - J_g(w^t) \eta(g^{-1}) + J_g(w^t) \delta(\alpha_{g^{-1}}(1)).
\end{align*}

Next, applying $\delta$ to the identity~\eqref{eq:product_recall} and using~\eqref{eq:deriv_cond_1_parseval} and~\eqref{eq:deriv_cond_2_parseval},
together with Lemma~\ref{lem:liftder2}\emph{(1)}, yields
\begin{equation*}
    \sum_{i=1}^{n_g} \big( \delta(u_i)^t \alpha_{g^{-1}}(v_i)^t 
    + u_i^t \alpha_{g^{-1}}(\delta(v_i))^t 
    + (u_i^t \eta(g^{-1})) \alpha_{g^{-1}}(v_i)^t  
    + u_i^t \alpha_{g^{-1}}(v_i^t \eta(g)) \big)\omega(g^{-1},g) = 0.
\end{equation*}
By Lemma~\ref{lem:multi_2}, this implies
\begin{equation*}
    \sum_{i=1}^{n_g} \big(
        \delta(u_i)^t \alpha_{g^{-1}}(v_i)^t
        + u_i^t \alpha_{g^{-1}}(\delta(v_i))^t
        + (u_i^t \eta(g^{-1})) \alpha_{g^{-1}}(v_i)^t
        + u_i^t \alpha_{g^{-1}}(v_i^t \eta(g))
        \big) = 0.
\end{equation*}
Multiplying from the right by $\alpha_{g^{-1}}(w^\ast)$ then gives
\begin{equation*}
    \sum_{i=1}^{n_g} \big(
        \delta(u_i)^t \alpha_{g^{-1}}(v_i^t w^\ast)
        + u_i^t \alpha_{g^{-1}}(\delta(v_i)^t w^\ast)
        + (u_i^t \eta(g^{-1})) \alpha_{g^{-1}}(v_i^t w^\ast)
        + u_i^t \alpha_{g^{-1}}(v_i^t \eta(g) w^\ast)
    \big) = 0.
\end{equation*}

From the preceding identities it follows that
\begin{equation*}
    \delta(J_g(w^t)) + J_g(w^t) \eta(g^{-1}) = J_g(\delta(w^t)) + J_g(w^t) \delta(\alpha_{g^{-1}}(1)) - \sum_{i=1}^{n_g} u_i^t \alpha_{g^{-1}}(v_i^t \eta(g) w^\ast).
\end{equation*}
Note that
\begin{equation*}
    \sum_{i=1}^{n_g} u_i^t \alpha_{g^{-1}}(v_i^t \eta(g) w^\ast) 
    =
    \sum_{i=1}^{n_g} u_i^t \alpha_{g^{-1}}(v_i^t (w^t \eta(g)^\dagger)^\dagger)
    =
    \sum_{i=1}^{n_g} u_i^t \alpha_{g^{-1}}(v_i^t (w^t \delta(\alpha_g(1)))^\dagger) - J_g(w^t \eta(g)).
\end{equation*}
Since $\alpha_g(1)\delta(\alpha_g(1))\alpha_g(1)=0$ by Lemma~\ref{lem:liftder1}, 
\begin{equation*}
    \delta(J_g(w^t)) + J_g(w^t)\eta(g^{-1})
    =
    J_g(\delta(w^t)) + J_g(w^t\eta(g)) + J_g(w^t)\delta(\alpha_{g^{-1}}(1)).
\end{equation*}
This establishes~\eqref{eq:star_check} and completes the proof.
\end{proof}

\subsection{Cohomological methods}\label{cohomo}

Let $S$ be a strongly $G$-graded ring with principal component $R$, let
$(n,\alpha,\omega)$ be a factor system associated with a module frame system
$(x_g,y_g)_{g\in G}$, and let $\delta: R \to R$ be a derivation.

Assume that for each $g\in G$ we are given an element $\eta(g)\in M_{n_g}(R)$ with
$\eta(e)=0$ such that~\eqref{eq:deriv_cond_1} holds for all $r\in R$, \ie, 
\begin{equation*}
    [\delta,\alpha_g](r)
    =
    [\eta(g),\alpha_g(r)]
    +
    \alpha_g(r)\delta(\alpha_g(1)).
\end{equation*}
%\begin{equation*}
%    \delta(\alpha_g(r))
%    = 
%    \eta(g) \alpha_g(r) + \alpha_g(\delta(r)) + \alpha_g(r) (\delta(\alpha_g(1)) - \eta(g)).
%\end{equation*}
In this situation,~\eqref{eq:deriv_formula} defines, for each $g\in G$,
an additive map
\begin{equation}\label{eq:connection_ng}
    \nabla_{\delta,g}: S_g \to S_g,
    \quad
    \nabla_{\delta,g}(u^t x_g) := \delta(u)^t x_g + u^t \eta(g) x_g,
    \quad u\in M_{n_g,1}(R),
\end{equation}
so that $\hd=\bigoplus_{g\in G}\nabla_{\delta,g}$ is the candidate graded lift.

\begin{lemma}\label{lem:connection_bimodule}
    For each $g \in G$, the map $\nabla_{\delta,g}$ satisfies the two-sided Leibniz rule
    \begin{equation}\label{eq:two_sided_leibniz}
        \nabla_{\delta,g}(r s r')
        =
        \delta(r) s r' + r \nabla_{\delta,g}(s) r' + r s \delta(r')
    \end{equation}
    for all $r,r'\in R$ and $s\in S_g$.
    In particular, $\nabla_{\delta,g}$ is a covariant derivative along $\delta$.
\end{lemma}
\begin{proof}
    Let $g \in G$, let $r,r'\in R$, and let $s=u^t x_g \in S_g$ with $u \in M_{n_g,1}(R)$.

    The left Leibniz rule is verified by a direct computation:
    \begin{align*} 
    \nabla_{\delta,g}(r(u^t x_g)) &= \nabla_{\delta,g}((ru)^t x_g) = \delta(ru)^t x_g + (ru)^t \eta(g) x_g 
    \\ 
    &= \delta(r)u^t x_g + r\delta(u)^t x_g + r u^t \eta(g) x_g = \delta(r)\cdot (u^t x_g) + r\cdot \nabla_{\delta,g}(u^t x_g). 
    \end{align*}

    For the right Leibniz rule,
    \begin{align*}
        \nabla_{\delta,g}((u^t x_g) r')
        &=
        \nabla_{\delta,g}(u^t \alpha_g(r') x_g)
        =
        \delta(u^t \alpha_g(r')) x_g + u^t \alpha_g(r') \eta(g) x_g
        \\
        &=
        \delta(u)^t \alpha_g(r') x_g + u^t \delta(\alpha_g(r')) x_g
        + u^t \alpha_g(r') \eta(g) x_g.
    \end{align*}
    Applying~\eqref{eq:deriv_cond_1} to the second term and using $x_g=\alpha_g(1)x_g$ together with Lemma~\ref{lem:liftder1}, one obtains
    \begin{align*}
        u^t \delta(\alpha_g(r')) x_g
        &=
        u^t \eta(g)\alpha_g(r') x_g + u^t \alpha_g(\delta(r')) x_g + u^t \alpha_g(r') \delta(\alpha_g(1)) x_g
        - u^t \alpha_g(r') \eta(g) x_g
        \\
        &=
        u^t \eta(g)\alpha_g(r') x_g + u^t \alpha_g(\delta(r')) x_g
        - u^t \alpha_g(r') \eta(g) x_g.
    \end{align*}
    It follows that
    \begin{align*}
        \nabla_{\delta,g}((u^t x_g) r')
        &=
        \delta(u)^t \alpha_g(r') x_g + u^t \eta(g) \alpha_g(r') x_g + u^t \alpha_g(\delta(r')) x_g
        \\
        &=
        (\delta(u)^t x_g + u^t \eta(g) x_g) r' + (u^t x_g) \delta(r')
        =
        \nabla_{\delta,g}(u^t x_g) r' + (u^t x_g) \delta(r').
        \qedhere
    \end{align*}
\end{proof}

\begin{remark}\label{rem:grassmann}%[The Grassmann connection]
For $g\in G$, set $\eta(g):=\delta(\alpha_g(1))\in M_{n_g}(R)$.
With this choice, a direct computation using Lemma~\ref{lem:liftder1} shows that~\eqref{eq:deriv_cond_1} holds for all $r\in R$ if and only if $y_g^t[\delta,\alpha_g](r)x_g = 0$ for all $r \in R$.

Notably, in this situation, the corresponding covariant derivative 
\begin{equation*}
    \nabla_{\delta,g}: S_g\to S_g,
    \quad
    \nabla_{\delta,g}(u^t x_g)
    = \delta(u)^t x_g + u^t \delta(\alpha_g(1)) x_g,
    \quad u\in M_{n_g,1}(R),
\end{equation*}
coincides with the covariant derivative of the Grassmann connection on the finitely generated projective left $R$-module $S_g$ along the derivation $\delta$.
\end{remark}

We now define the defect of the family $(\nabla_{\delta,g})_{g \in G}$ with respect to the multiplication in $S$.
For $g,h\in G$, let
\begin{equation}\label{eq:defect_map}
    \Delta_\delta(g,h): S_{gh} \to S_{gh}
\end{equation}
be the additive map determined by
\begin{equation}\label{eq:defect_def}
    \Delta_\delta(g,h)(s_g s_h)
    :=
    \nabla_{\delta,g}(s_g) s_h + s_g \nabla_{\delta,h}(s_h) - \nabla_{\delta,gh}(s_g s_h),
\end{equation}
for homogeneous elements $s_g\in S_g$ and $s_h\in S_h$, and extended additively to $S_{gh}$.

\begin{lemma}
    For all $g,h \in G$, the map $\Delta_\delta(g,h): S_{gh} \to S_{gh}$ is well defined.
\end{lemma}
\begin{proof}
Let $g,h \in G$.
Since $S$ is strongly $G$-graded, multiplication induces an $R$-bimodule isomorphism
\begin{equation*}
    m_{g,h}: S_g \otimes_R S_h \to S_{gh},
    \quad
    m_{g,h}(s_g \otimes s_h) := s_g s_h.
\end{equation*}
Consider the $R$-bilinear map
\begin{equation*}
    D_\delta(g,h): S_g \times S_h \to S_{gh},
    \quad
    D_\delta(g,h)(s_g,s_h)
    :=
    \nabla_{\delta,g}(s_g) s_h + s_g\nabla_{\delta,h}(s_h) - \nabla_{\delta,gh}(s_g s_h).
\end{equation*}
Using the two-sided Leibniz rule~\eqref{eq:two_sided_leibniz} for $\nabla_{\delta,g}$, $\nabla_{\delta,h}$, and $\nabla_{\delta,gh}$, one verifies that $D_\delta(g,h)$ is~balanced, \ie,
\begin{equation*}
    D_\delta(g,h)(s_g r,s_h) = D_\delta(g,h)(s_g,r s_h)
\end{equation*}
for all $r \in R$, $s_g \in S_g$, and $s_h \in S_h$.
Consequently, $D_\delta(g,h)$ factors through a unique $R$-bimodule homomorphism
\begin{equation*}
    D_\delta(g,h): S_g \otimes_R S_h \to S_{gh},
\end{equation*}
which, by abuse of notation, is denoted by the same symbol.
The composition
\begin{equation*}
    D_\delta(g,h) \circ m_{g,h}^{-1}: S_{gh} \to S_{gh}
\end{equation*}
is therefore well defined and, by construction, coincides with $\Delta_\delta(g,h)$ as defined in~\eqref{eq:defect_def}.
\end{proof}

\begin{lemma}\label{lem:defect_def_dc}
    For all $g,h\in G$, $u\in M_{n_g,1}(R)$, and $v\in M_{n_h,1}(R)$,
    \begin{equation}\label{eq:defect_def_dc}
        \Delta_\delta(g,h)((u^t x_g)(v^t x_h))
        =
        u^t\alpha_g(v)^t \Theta_\delta(g,h) x_{gh},
    \end{equation}
    where
    \begin{equation*}
        \Theta_\delta(g,h):=
        (\eta(g)\kronr 1_{n_h})\omega(g,h)
        +
        \alpha_g(\eta(h))\omega(g,h)
        +
        \omega(g,h)(\delta(\alpha_{gh}(1))-\eta(gh))
        -
        \delta(\omega(g,h)).
    \end{equation*}
\end{lemma}
\begin{proof}
Let $g,h\in G$, let $u\in M_{n_g,1}(R)$, let $v\in M_{n_h,1}(R)$, and set $s_g:=u^t x_g\in S_g$ and $s_h:=v^t x_h\in S_h$, so that $s_gs_h=u^t\alpha_g(v)^t\omega(g,h)x_{gh}$.

The three terms in~\eqref{eq:defect_def} are computed as follows:
\begin{align*}
    \nabla_{\delta,g}(s_g) s_h
    &=
    (\delta(u)^t x_g + u^t \eta(g) x_g) (v^t x_h)
    =
    \delta(u)^t \alpha_g(v)^t \omega(g,h) x_{gh}
    +
    u^t\eta(g) \alpha_g(v)^t \omega(g,h) x_{gh},
    \\
    s_g \nabla_{\delta,h}(s_h)
    &=
    (u^t x_g)(\delta(v)^t x_h + v^t \eta(h) x_h)
    =
    u^t \alpha_g(\delta(v))^t \omega(g,h) x_{gh}
    +
    u^t \alpha_g(v^t\eta(h)) \omega(g,h) x_{gh},
    \\
    \shortintertext{and}
    \nabla_{\delta,gh}(s_gs_h)
    &=
    \delta(u^t \alpha_g(v)^t \omega(g,h)) x_{gh}
    +
    u^t \alpha_g(v)^t \omega(g,h) \eta(gh) x_{gh}
    \\  
    &=
    \delta(u)^t \alpha_g(v)^t \omega(g,h) x_{gh}
    +
    u^t \delta(\alpha_g(v))^t \omega(g,h) x_{gh}
    +
    u^t \alpha_g(v)^t \delta(\omega(g,h)) x_{gh}
    \\
    &\quad
    +
    u^t \alpha_g(v)^t \omega(g,h) \eta(gh) x_{gh}.
\end{align*}

Putting these pieces together, one obtains
\begin{align*}
    \Delta_\delta(g,h)(s_gs_h)
    &=
    u^t \eta(g) \alpha_g(v)^t \omega(g,h) x_{gh}
    +
    u^t \alpha_g (\delta(v))^t \omega(g,h) x_{gh}
    +
    u^t \alpha_g (v^t\eta(h)) \omega(g,h) x_{gh}
    \\
    &\quad
    -
    u^t \delta(\alpha_g(v))^t \omega(g,h) x_{gh}
    -
    u^t \alpha_g(v)^t \delta(\omega(g,h)) x_{gh}
    -
    u^t \alpha_g(v)^t \omega(g,h) \eta(gh) x_{gh}.
\end{align*}
Substituting~\eqref{eq:deriv_cond_1} into the previous formula cancels the first two summands, leaving
\begin{align*}
    \Delta_\delta(g,h)(s_gs_h)
    &=
    u^t \alpha_g(v^t \eta(h)) \omega(g,h) x_{gh}
    -
    u^t \alpha_g(v)^t ((\delta(\alpha_g(1))-\eta(g))\kronr 1_{n_h}) \omega(g,h) x_{gh}
    +
    \\
    &\quad
    -
    u^t \alpha_g(v)^t \delta(\omega(g,h)) x_{gh}
    -
    u^t \alpha_g(v)^t \omega(g,h) \eta(gh) x_{gh}.
\end{align*}
The piece with $\delta(\alpha_g(1))$ vanishes after multiplying by $\omega(g,h)$
by Lemma~\ref{lem:liftder2}\emph{(1)}, so
\begin{equation*}
    -\alpha_g(v)^t ((\delta(\alpha_g(1))-\eta(g))\kronr 1_{n_h}) \omega(g,h) x_{gh}
    =
    \alpha_g(v)^t (\eta(g)\kronr 1_{n_h}) \omega(g,h) x_{gh}.
\end{equation*}
Moreover, by Lemma~\ref{lem:factorsystem}~\emph{(2)}-\emph{(3)} one has $\alpha_g(\eta(h))b\omega(g,h)b=b\alpha_g(\eta(h))b\omega(g,h)b\alpha_{gh}(1)$,
so the projection $\alpha_{gh}(1)$ may be inserted freely; then adding and subtracting
$\omega(g,h)\delta(\alpha_{gh}(1))$, which vanishes on $S_{gh}$ by
Lemma~\ref{lem:liftder2}~\emph{(2)}, gives~\eqref{eq:defect_def_dc}.
\end{proof}

\begin{corollary}\label{cor:delta_0}
    Under the above assumptions, one has $\Delta_\delta(g,h)=0$ for all $g,h\in G$
    if and only if~\eqref{eq:deriv_cond_2} holds for all $g,h\in G$.
\end{corollary}

Beyond their explicit coordinate description, the defect maps enjoy an important structural property:
%A key structural feature of the defect maps is the following.

\begin{lemma}\label{lem:defect_bimodule}
    For all $g,h\in G$, the defect map $\Delta_\delta(g,h)$ is an $R$-bimodule homomorphism.
\end{lemma}
\begin{proof}
    Let $g,h\in G$. 
    By Lemma~\ref{lem:connection_bimodule}, for any $r\in R$ one has
    \begin{align*}
        \Delta_\delta(g,h)(r(s_g s_h))
        &=
        \Delta_\delta(g,h)((r s_g)s_h)
        =
        \nabla_{\delta,g}(r s_g)s_h + (r s_g)\nabla_{\delta,h}(s_h) - \nabla_{\delta,gh}((r s_g)s_h)
        \\
        &=
        (\delta(r)s_g + r\nabla_{\delta,g}(s_g))s_h + r s_g\nabla_{\delta,h}(s_h)
        - (\delta(r)s_g s_h + r\nabla_{\delta,gh}(s_g s_h))
        \\
        &=
        r(\nabla_{\delta,g}(s_g)s_h + s_g\nabla_{\delta,h}(s_h) - \nabla_{\delta,gh}(s_g s_h))
        =
        r \Delta_\delta(g,h)(s_g s_h).
    \end{align*}
    Right $R$-linearity is verified analogously.
\end{proof}

By~\cite[Lem.~I.3.11.2]{NaOy82}, for all $g,h\in G$ there exists a unique element
$\Delta_\delta(g,h)\in Z(R)$ implementing the map $\Delta_\delta(g,h)$
by left multiplication; we use the same notation for both.
%Thus the family $(\Delta_\delta(g,h))_{g,h\in G}$ may be viewed as a $2$-cochain
%\begin{equation*}
%    \Delta_\delta \in C^2(G,Z(R)).
%\end{equation*}

In~fact, more structure is present.
The ring $S$ induces a $G$-module structure on $Z(R)$ via a homomorphism
$\beta: G \to \aut(Z(R))$, characterized by
\begin{equation}\label{eq:mod_str}
    s z = \beta_g(z) s,
\end{equation}
for all $g\in G$, $s\in S_g$, and $z\in Z(R)$ (see, e.g.,~\cite[Lem.~I.3.12]{NaOy82}).
Notably, with respect to the module frame system $(x_g,y_g)_{g\in G}$, this action is explicitly given by $\beta^{-1}_g(z) = y_g^t z x_g$.
In the special case of a crossed product $R \rtimes_{(\alpha,\omega)} G$, the induced action coincides with the restriction of $\alpha$ to $Z(R)$.

This leads naturally to the group cohomology of $G$ with coefficients in the
$G$-module $Z(R)$. 
For convenience we briefly recall the relevant definitions:

For $p\ge 0$, let
\begin{equation*}
    C^p(G,Z(R)) := \{f: G^p \to Z(R)\}
\end{equation*}
denote the group of $p$-cochains.
The group differential
\begin{equation*}
    d_\beta : C^p(G,Z(R)) \to C^{p+1}(G,Z(R))
\end{equation*}
is given by
\begin{align*}
    (d_\beta f)(g_0,\ldots,g_p)
    &=
    \beta_{g_0}(f(g_1,\ldots,g_p))
    +
    \sum_{j=1}^{p}(-1)^j
    f(g_1,\ldots,g_j g_{j+1},\ldots,g_{p+1})
    \\
    &\quad
    + (-1)^{p+1} f(g_0,\ldots,g_{p-1}).
\end{align*}
The groups of $p$-cocycles and $p$-coboundaries are defined by
\begin{equation*}
    Z^p_\beta(G,Z(R)) := \ker(d_\beta),
    \quad \text{and} \quad
    B^p_\beta(G,Z(R)) := \operatorname{im}(d_\beta),
\end{equation*}
and the $p$-th cohomology group is
\begin{equation*}
    H^p_\beta(G,Z(R)) := Z^p_\beta(G,Z(R)) / B^p_\beta(G,Z(R)).
\end{equation*}
For further details on group cohomology we refer to~\cite[Chap.~IV]{MacLane95}.

%This leads naturally to the group cohomology of $G$ with coefficients in the $G$-module $Z(R)$. 
%In particular, we consider the goups
%\begin{equation*}
%    Z^p_\beta(G,Z(R))
%    \quad \text{and} \quad
%    H^p_\beta(G,Z(R)),
 %   \quad p \ge 1,
%\end{equation*}
%(see, \eg,~\cite[Chap.~IV]{MacLane95}). 
The following discussion shows how the elements $\Delta_\delta(g,h)$ can be interpreted in terms of group cohomology.
In particular, the family $(\Delta_\delta(g,h))_{g,h\in G}$ defines a $2$-cochain
\begin{equation*}
    \Delta_\delta \in C^2(G,Z(R)).
\end{equation*}

\begin{lemma}
    $\Delta_\delta\in Z^2_\beta(G,Z(R))$, \ie, $\Delta_\delta$ is a
    $Z(R)$-valued $2$-cocycle on $G$ with respect to $\beta$.
\end{lemma}
\begin{proof}
Let $g,h,k\in G$ and let $s_g\in S_g$, $s_h\in S_h$, and $s_k\in S_k$.
Consider
\begin{equation}\label{eq:triple_defect}
    e:=
    \nabla_{\delta,g}(s_g)s_hs_k + s_g \nabla_{\delta,h}(s_h) s_k + s_g s_h \nabla_k(s_k) - \nabla_{ghk}(s_g s_h s_k)
    \in S_{ghk}.
\end{equation}
This expression can be written in two ways, according to the associations
$(s_g s_h)s_k$ and $s_g(s_h s_k)$.
Using~\eqref{eq:defect_def}, one obtains
\begin{align*}
    e
    &=
    (\nabla_{\delta,gh}(s_g s_h) 
    + 
    \Delta_\delta(g,h)(s_g s_h))s_k
    + 
    (s_g s_h)\nabla_k(s_k) 
    - 
    \nabla_{ghk}(s_g s_h s_k)
    \\
    &=
    \Delta_\delta(gh,k) ((s_g s_h)s_k) + \Delta_\delta(g,h)(s_g s_h)s_k.
\end{align*}
Similarly,
\begin{align*}
    e
    &=
    \nabla_{\delta,g}(s_g)(s_h s_k) 
    + 
    s_g(\nabla_{hk}(s_h s_k) 
    + 
    \Delta_\delta(h,k)(s_h s_k))
    - 
    \nabla_{ghk}(s_gs_hs_k)
    \\
    &=
    \Delta_\delta(g,hk) (s_g(s_h s_k)) + s_g \Delta_\delta(h,k)(s_h s_k).
\end{align*}
Hence,
\begin{equation*}
    \Delta_\delta(gh,k)(s_g s_h s_k) 
    + 
    \Delta_\delta(g,h)(s_g s_h)s_k
    =
    \Delta_\delta(g,hk)(s_g s_h s_k) + s_g \Delta_\delta(h,k)(s_h s_k).
\end{equation*}
Passing to the central element picture, the previous identity becomes
\begin{equation*}
    (\Delta_\delta(gh,k) + \Delta_\delta(g,h) - \Delta_\delta(g,hk)-\beta_g(\Delta_\delta(h,k)))s_g s_h s_k = 0,
\end{equation*}
where the last term was rewritten using the relation
$s_g z=\beta_g(z)s_g$ for $z\in Z(R)$ (see~\eqref{eq:mod_str}.
Since products of the form $s_gs_hs_k$ generate $S_{ghk}$ as a left $R$-module,
it follows that the coefficient vanishes in $Z(R)$:
\begin{equation}\label{eq:deriv_cond_1_again}
    \beta_g(\Delta_\delta(h,k)) + \Delta_\delta(g,hk) - \Delta_\delta(g,h) - \Delta_\delta(gh,k) = 0.
\end{equation}
This is precisely the $2$-cocycle identity for $\Delta_\delta\in C^2(G,Z(R))$ with respect to $\beta$.
\end{proof}

\begin{definition}\label{def:mult_curv}
The cocycle
\begin{equation*}
    \Delta_\delta \in Z^2_\beta(G,Z(R))
\end{equation*}
is called the \emph{multiplicative curvature} associated with the
family $(\nabla_{\delta,g})_{g\in G}$.
Its cohomology class
\begin{equation*}
    [\Delta_\delta] \in H^2_\beta(G,Z(R))
\end{equation*}
is called the \emph{multiplicative curvature class}.
\end{definition}

\begin{remark}\label{rem:mult_curv_class}
The terminology reflects that $\Delta_\delta(g,h)$ measures the failure of the maps $\nabla_{\delta,g}$, $g \in G$, to be compatible with the multiplication in $S$.
Thus $\Delta_\delta$ plays the role of a curvature measuring the obstruction to assembling the $\nabla_g$'s into a graded derivation of $S$.
Equivalently, the class $[\Delta_\delta]$ represents the obstruction to lifting $\delta$ to a graded derivation of $S$, as will be shown in Theorem~\ref{thm:liftder_cohomo} below.
\end{remark}

\begin{lemma}\label{lem:defect_coboundary}
    Let $\xi\in C^1(G,Z(R))$.
    For each $g \in G$ define
    \begin{equation*}
        \eta'(g) := \eta(g) - \xi(g) \alpha_g(1),
        \quad g\in G.
    \end{equation*}
    Let $\nabla'_{\delta,g}$ and $\Delta'_\delta(g,h)$ denote the corresponding maps.
    Then for all $g,h\in G$ one has
    \begin{equation}\label{eq:def_cohomo}
        \Delta'_\delta(g,h) = \Delta_\delta(g,h) - (d\xi)(g,h),
    \end{equation}
    where $(d\xi)(g,h) := \beta_g(\xi(h)) + \xi(g) - \xi(gh)$.
\end{lemma}
\begin{proof}
Let $g,h\in G$, let $s_g\in S_g$, and let $s_h\in S_h$.
Writing $s_g = u^t x_g$ for some $u \in M_{n_g,1}(R)$, the definition of $\nabla'_{\delta,g}$ gives
\begin{equation*}
    \nabla'_{\delta,g}(s_g)
    =
    \delta(u)^t x_g + u^t \eta'(g) x_g
    =
    \nabla_{\delta,g}(s_g) - u^t \xi(g) \alpha_g(1) x_g
    =
    \nabla_{\delta,g}(s_g) - \xi(g) s_g.
\end{equation*}
Thus
\begin{equation*}
    \nabla'_{\delta,g} = \nabla_{\delta,g} - \xi(g) \id_{S_g}.
\end{equation*}
Substituting this into the definition of the defect yields
\begin{align*}
    \Delta'_\delta(g,h)(s_g s_h)
    &=
    \nabla'_{\delta,g}(s_g)s_h + s_g\nabla'_{\delta,h}(s_h) - \nabla'_{\delta,gh}(s_g s_h)
    \\
    &=
    (\nabla_{\delta,g}(s_g)-\xi(g)s_g)s_h
    + 
    s_g(\nabla_{\delta,h}(s_h)-\xi(h) s_h)
    -
    (\nabla_{\delta,gh}(s_g s_h)-\xi(gh)s_g s_h)
    \\
    &=
    \Delta_\delta(g,h)(s_g s_h)
    - 
    \xi(g)s_g s_h
    - 
    s_g \xi(h) s_h
    + 
    \xi(gh)s_g s_h
    \\
    &=
    \Delta_\delta(g,h)(s_g s_h)
    - 
    \xi(g)s_g s_h
    - 
    \beta_g(\xi(h))s_g s_h
    + 
    \xi(gh)s_g s_h
    \\
    &=
    (\Delta_\delta(g,h)-(d\xi)(g,h))s_g s_h.
\end{align*}
Identifying defect maps with their central implementing elements completes the proof.
\end{proof}

\begin{theorem}\label{thm:liftder_cohomo}
    Let $S$ be a strongly $G$-graded ring with principal component $R$, let
    $(n,\alpha,\omega)$ be an associated factor system, and let $\delta: R \to R$ be a derivation.
    Assume that for each $g \in G$ there exists $\eta(g) \in M_{n_g}(R)$ with $\eta(e)=0$
    such that, for all $r \in R$,~\eqref{eq:deriv_cond_1} holds, \ie,
    \begin{equation*}
        [\delta,\alpha_g](r)
        =
        [\eta(g),\alpha_g(r)]
        +
        \alpha_g(r)\delta(\alpha_g(1)).
    \end{equation*}
Then $\delta$ admits a graded lift to $S$ if and only if $[\Delta_\delta] \in H^2_\beta(G,Z(R))$ vanishes.
%cohomology class $[\Delta_\delta] \in H^2_\beta(G,Z(R))$~vanishes.
\end{theorem}
\begin{proof}
(``$\Rightarrow$'')
If $\delta$ admits a graded lift, then $\Delta_\delta(g,h)=0$ for all $g,h\in G$,
by Theorem~\ref{thm:liftder}~\emph{(1)} and Corollary~\ref{cor:delta_0}.
Hence $[\Delta_\delta]=0$ in $H^2_\beta(G,Z(R))$.

(``$\Leftarrow$'')
Conversely, assume that $[\Delta_\delta] = 0$ in $H^2_\beta(G,Z(R))$.
Then $\Delta_\delta$ is a coboundary, so there exists $\xi\in C^1(G,Z(R))$
with $\Delta_\delta = d\xi$.
Define a modified family
\begin{equation*}
    \eta'(g) := \eta(g)-\xi(g)\alpha_g(1), \quad g\in G.
\end{equation*}
For all $g\in G$ and $r\in R$,~\eqref{eq:deriv_cond_1} holds with $\eta'$
in place of $\eta$, and by Lemma~\ref{lem:defect_coboundary} one has
\begin{equation*}
    \Delta'_\delta(g,h) = \Delta_\delta(g,h) - (d\xi)(g,h),
    \quad g,h\in G.
\end{equation*}
Thus $\Delta'_\delta(g,h) = 0$ for all $g,h\in G$.
Applying Corollary~\ref{cor:delta_0} and Theorem~\ref{thm:liftder}~\emph{(2)} yields a graded lift of $\delta$ to $S$.
\end{proof}

\begin{corollary}\label{cor:equivariant_cohomo}
    Let $S$ be a strongly $G$-graded ring with principal component $R$, let
    $(n,\alpha,\omega)$ be an associated factor system, and let $\delta: R \to R$ be a derivation.
    Assume that $\delta(\alpha_g(r)) = \alpha_g(\delta(r))$ for all $g \in G$ and $r \in R$, so that~\eqref{eq:deriv_cond_1} holds with $\eta(g)=0$ for all $g\in G$.
    If $[\Delta_\delta]\in H^2_\beta(G,Z(R))$ vanishes, then $\delta$ admits a graded lift to $S$.
\end{corollary}

\begin{remark}\label{rem:equivariant_defect_general}
Let $g,h \in G$.
In the situation of Corollary~\ref{cor:equivariant_cohomo}, the defect term simplifies to
\begin{equation*}
    \Theta_\delta(g,h) = -\delta(\omega(g,h))
\end{equation*}
(see Lemma~\ref{lem:defect_def_dc}) and therefore, for $u\in M_{n_g,1}(R)$ and $v\in M_{n_h,1}(R)$,
\begin{equation*}
    \Delta_\delta(g,h)\bigl((u^t x_g)(v^t x_h)\bigr)
    =
    -\,u^t \alpha_g(v)^t\,\delta(\omega(g,h))\,x_{gh}.
\end{equation*}

In the crossed product case $S=R\rtimes_{(\alpha,\omega)}G$, this reduces to
\begin{equation*}
    \Delta_\delta(g,h)(u_g u_h)
    =
    -\delta(\omega(g,h))\,u_{gh}.
\end{equation*}
Since $u_g u_h=\omega(g,h)u_{gh}$ and $\Delta_\delta(g,h)$ is left
$R$-linear, it follows that
\begin{equation*}
    \Delta_\delta(g,h)(u_{gh})
    =
    -\omega(g,h)^{-1}\delta(\omega(g,h))\,u_{gh}.
\end{equation*}
Hence, under the central element identification,
\begin{equation*}
    \Delta_\delta(g,h)
    =
    -\omega(g,h)^{-1}\delta(\omega(g,h))
    \in Z(R),
\end{equation*}
which may be interpreted as the (noncommutative) logarithmic derivative
of the cocycle $\omega(g,h)$.

Indeed, the expression $\omega^{-1}\delta(\omega)$ is formally
analogous to the classical Maurer--Cartan form on a Lie group.
\end{remark}

In the case of strongly $\mathbb{Z}$-graded rings the lifting condition simplifies further. 
In particular, the second cocycle condition~\eqref{eq:deriv_cond_2} becomes redundant, and it suffices to
verify the compatibility condition for the generator $1\in\mathbb{Z}$.
This yields the following corollary.

\begin{corollary}\label{cor:Z_lift_from_generator}
    Let $S$ be a strongly $\mathbb{Z}$-graded ring with principal component $R$, let
    $(n,\alpha,\omega)$ be an associated factor system, and let $\delta: R \to R$ be a derivation.
    Assume that there exists $\eta \in M_{n_1}(R)$ such that, for all $r \in R$,
    \begin{equation}\label{eq:Z_generator_condition}
        [\delta,\alpha_1](r)
        =
        [\eta,\alpha_1(r)]
        +
        \alpha_1(r)\delta(\alpha_1(1)).
    \end{equation}
    Then $\delta$ admits a graded lift to $S$.
\end{corollary}

The proof relies on the following auxiliary result, which is the converse of
Lemma~\ref{lem:connection_bimodule}.

\begin{lemma}\label{lem:conn_to_eta}
    Let $S$ be a strongly $G$-graded ring with principal component $R$,
    and let $(n,\alpha,\omega)$ be a factor system associated with a module frame system
    $(x_g,y_g)_{g \in G}$.
    Let $\delta: R \to R$ be a derivation and suppose that $(\nabla_{\delta,g}: S_g \to S_g)_{g \in G}$ is a family of covariant derivatives along $\delta$. 
    Extend each $\nabla_{\delta,g}$, $g \in G$, entrywise to $M_{n_g,1}(S_g)$ and define
    \begin{equation*}
        \eta(g) := \nabla_{\delta,g}(x_g) y_g^t \in M_{n_g}(R),
        \quad g \in G.
    \end{equation*}
    Then, for all $g \in G$ and all $r\in R$, one has
    \begin{equation*}
        [\delta,\alpha_g](r)
        =
        [\eta(g),\alpha_g(r)]
        +
        \alpha_g(r)\delta(\alpha_g(1)).
    \end{equation*}
%    Equivalently,~\eqref{eq:deriv_cond_1} holds for all $g \in G$.
\end{lemma}
\begin{proof}
Let $g \in G$ and let $r \in R$.
Using the relation $x_g r = \alpha_g(r) x_g$, apply the $\delta$-bimodule~connection $\nabla_{\delta,g}$ to both sides.
By the right Leibniz rule,
\begin{equation*}
    \nabla_{\delta,g}(x_g r) = \nabla_{\delta,g}(x_g)r + x_g \delta(r),
\end{equation*}
while the left Leibniz rule gives
\begin{equation*}
    \nabla_{\delta,g}(\alpha_g(r)x_g)
    =
    \delta(\alpha_g(r)) x_g + \alpha_g(r) \nabla_{\delta,g}(x_g).
\end{equation*}
Hence
\begin{equation*}
    \delta(\alpha_g(r)) x_g + \alpha_g(r) \nabla_{\delta,g}(x_g)
    =
    \nabla_{\delta,g}(x_g) r+x_g \delta(r).
\end{equation*}
Multiplying on the right by $y_g^t$, one obtains
\begin{equation*}
    \delta(\alpha_g(r)) \alpha_g(1)
    =
    \eta(g) \alpha_g(r) + \alpha_g(\delta(r)) - \alpha_g(r) \eta(g).
\end{equation*}
Finally, since
\begin{equation*}
    \delta(\alpha_g(r))
    =
    \delta(\alpha_g(r)) \alpha_g(1) + \alpha_g(r) \delta(\alpha_g(1))
\end{equation*}
substituting this identity into the previous equation completes the proof.
\end{proof}

\begin{remark}
    The existence of a family $(\eta(g))_{g\in G}$ satisfying~\eqref{eq:deriv_cond_1} is equivalent to the existence of a family of covariant derivatives $(\nabla_{\delta,g}: S_g \to S_g)_{g \in G}$ along $\delta$.
\end{remark}

\begin{proof}[Proof of Corollary~\ref{cor:Z_lift_from_generator}]
\eqref{eq:Z_generator_condition} is invariant under conjugation of factor systems:
if $(n',\alpha',\omega')$ is conjugate to $(n,\alpha,\omega)$ via
$(v_k,w_k)_{k\in\mathbb Z}$ as in Lemma~\ref{lem:factorsystem}\emph{(\ref{it:conjugate})}, then
\begin{equation*}
    \alpha'_1(r) = v_1 \alpha_1(r) w_1,
    \quad r\in R,
\end{equation*}
and~\eqref{eq:Z_generator_condition} for $\alpha_1$ is equivalent to the corresponding identity for $\alpha'_1$ with $\eta' := v_1 \eta w_1\in M_{n'_1}(R)$.
Consequently, $(n,\alpha,\omega)$ may be replaced by any conjugate factor system.

Choose $n_1\in\mathbb N$ and elements
\begin{equation*}
    x_1 \in M_{n_1,1}(S_1),
    \quad
    y_1 \in M_{1,n_1}(S_{-1}),
\end{equation*}
such that $y_1^t x_1=1$ (existence follows from Lemma~\ref{lem:dade}).
Since $S$ is strongly $\mathbb Z$-graded, multiplication induces $R$-bimodule isomorphisms
\begin{equation*}
    \mu_{m,n}: S_m \otimes_R S_n \to S_{m+n},
    \quad
    \mu_{m,n}(s_m \otimes s_n) = s_m s_n,
\end{equation*}
for all $m,n\in\mathbb Z$.
In particular, for $m\ge 1$ one obtains $R$-bimodule isomorphisms
\begin{equation*}
    \mu_m : S_1^{\otimes_R m} \to S_m,
    \quad
    \mu_m(s_1\otimes\cdots\otimes s_m)=s_1\cdots s_m,
\end{equation*}
and analogously
\begin{equation*}
    \mu_{-m}: S_{-1}^{\otimes_R m} \to S_{-m}.
\end{equation*}
These choices are fixed once and for all.

Set $\eta(1) := \eta$ and define a map $\nabla_1: S_1 \to S_1$ by
\begin{equation*}
    \nabla_1(u^t x_1)
    :=
    \delta(u)^t x_1 + u^t \eta(1) x_1,
    \quad u \in M_{n_1,1}(R).
\end{equation*}
By Lemma~\ref{lem:connection_bimodule},~\eqref{eq:Z_generator_condition}
is exactly what is required for $\nabla_1$ to be a $\delta$-bimodule~connection on $S_1$.

For $m \ge 1$, define an additive map
$\nabla^{\otimes m}:S_1^{\otimes_R m}\to S_1^{\otimes_R m}$ by
\begin{equation*}
    \nabla^{\otimes m}(s_1 \otimes \cdots \otimes s_m)
    :=
    \sum_{j=1}^m
    s_1 \otimes \cdots \otimes s_{j-1}
    \otimes \nabla_1(s_j)
    \otimes s_{j+1} \otimes \cdots \otimes s_m.
\end{equation*}
Since $\nabla_1$ satisfies the two-sided Leibniz rule, this map is compatible with the balancing relations and hence well-defined on $S_1^{\otimes_R m}$; 
moreover it is a $\delta$-bimodule~connection.
Transporting along $\mu_m$ yields $\delta$-bimodule~connections
\begin{equation*}
    \nabla_m := \mu_m \circ \nabla^{\otimes m} \circ \mu_m^{-1}: S_m\to S_m,
    \quad m\ge 1.
\end{equation*}

Since $S$ is strongly $\mathbb Z$-graded, the $R$-bimodules $S_1$ and $S_{-1}$ are inverse to each other, so the multiplication pairing induces an $R$-bimodule isomorphism
\begin{equation*}
    \nu: S_{-1} \to \Hom_R(S_1,R),
    \quad
    \nu(t)(s):=t s,
\end{equation*}
where $\Hom_R(S_1,R)$ denotes right $R$-linear maps and carries the obvious $R$-bimodule structure.
Define an additive map $\nabla_1^\vee : \Hom_R(S_1,R) \to \Hom_R(S_1,R)$ by
\begin{equation*}
    (\nabla_1^\vee f)(s) := \delta(f(s))-f(\nabla_1(s)),
    \quad f \in \Hom_R(S_1,R).
\end{equation*}
A direct verification using the two-sided Leibniz rule for $\nabla_1$ shows that $\nabla_1^\vee$ is a covariant derivative along $\delta$ on $\Hom_R(S_1,R)$.
Transporting along $\nu$ yields a $\delta$-bimodule~connection on~$S_{-1}$,
\begin{equation*}
    \nabla_{-1} : =\nu^{-1} \circ \nabla_1^\vee \circ \nu: S_{-1} \to S_{-1}.
\end{equation*}

For $m \ge 1$, form tensor product connections $(\nabla_{-1})^{\otimes m}$ on $S_{-1}^{\otimes_R m}$
by the same formula as above, and transport along $\mu_{-m}$ to obtain covariant derivatives along $\delta$:
\begin{equation*}
    \nabla_{-m} := \mu_{-m} \circ (\nabla_{-1})^{\otimes m} \circ \mu_{-m}^{-1}: S_{-m} \to S_{-m}.
\end{equation*}

Choose a module frame system $(x_k,y_k)_{k \in \mathbb Z}$ with $x_0=y_0=1$ and whose degree-$1$ part coincides with the fixed pair $(x_1,y_1)$.
For each $k \in \mathbb Z$, define
\begin{equation*}
    \eta(k):=\nabla_k(x_k)\,y_k^t\in M_{n_k}(R).
\end{equation*}
By Lemma~\ref{lem:conn_to_eta},~\eqref{eq:deriv_cond_1} holds for all
$k \in \mathbb Z$ with this choice of $\eta(k)$.
Thus the algebraic hypotheses of Theorem~\ref{thm:liftder_cohomo} are satisfied for
$G = \mathbb Z$.

The remaining cohomological condition is automatic in this case:
since $H^2_\beta(\mathbb Z,Z(R))=0$, the class
$[\Delta_\delta]\in H^2(\mathbb Z,Z(R))_\beta$ vanishes.
It follows that $\delta$ admits a graded lift to $S$.
\end{proof}

\begin{remark}
Corollary~\ref{cor:Z_lift_from_generator} admits the following interpretation. 
A graded derivation of $S$ is completely determined by two pieces of data:
\begin{itemize}
\item 
    a derivation $\delta$ of the principal component $L_0$, and
\item 
    a connection on the $L_0$-bimodule $L_1$ which is compatible with $\delta$.
\end{itemize}
More precisely, the action of the derivation on $L_1$ satisfies a Leibniz-type
rule with respect to the bimodule structure, and this data uniquely determines
the graded derivation on the whole algebra.
\end{remark}

We proceed with a series of immediate corollaries:

\begin{corollary}\label{cor:Z_lift_from_generator_comm}
    Let $S$ be a strongly $\mathbb{Z}$-graded ring with principal component $R$, let
    $(n,\alpha,\omega)$ be an associated factor system, and let $\delta: R \to R$ be a derivation.
    Assume that $\delta \circ \alpha_1 = \alpha_1 \circ \delta$.
    Then $\delta$ admits a graded lift to $S$.
\end{corollary}

\begin{corollary}\label{cor:Z_crossed_lift}
    Let $R$ be a ring, let $\alpha \in \aut(R)$, let $\omega: \mathbb Z \times \mathbb Z \to R^\times$ be a normalized $2$-cocycle, and let $S := R \rtimes_{(\alpha,\omega)} \mathbb Z$ be the associated crossed product.
    A derivation $\delta: R \to R$ admits a graded lift to $S$ if and only if there exists $\eta \in R$ such that
    \begin{equation*}
        \delta(\alpha(r)) - \alpha(\delta(r)) = [\eta,\alpha(r)]
    \end{equation*}
    for all $r \in R$.
    Equivalently, the class of $\delta$ in $\der(R)/\inn(R)$ is fixed by the action of $\alpha$.
\end{corollary}

\begin{corollary}\label{cor:Z_crossed_lift_comm}
    Let $R$ be a commutative ring, let $\alpha \in \aut(R)$, let $\omega : \mathbb Z \times \mathbb Z \to R^\times$ be a normalized $2$-cocycle, and let $S := R \rtimes_{(\alpha,\omega)} \mathbb Z$ be the associated crossed product.
    A derivation $\delta \in \der(R)$ admits a graded lift to $S$ if and only if $\delta \circ \alpha = \alpha \circ \delta$.
%Equivalently,
%\begin{equation}\label{eq:alpha-derivations}
%    \der_\mathrm{ext}(R) = \der(R)_\alpha := \{\delta \in \der(R):\delta \circ \alpha = \alpha \circ \delta\}
%\end{equation}
\end{corollary}

\begin{example}\label{ex:L(1,2)_lift_simple}
Let $E$ be the graph with one vertex $v$ and two loop edges $e_1,e_2$.
The associated Leavitt path algebra $L_{\mathbb C}(1,2)$ is strongly $\mathbb Z$-graded but not a crossed product (see Section~\ref{sec:L(1,2)}).

Using the module frame system from Example~\ref{ex:L(1,2)_factorsystem}, we may take $x_1 = e_1$ and $y_1 = e_1^\ast$ so that $y_1 x_1 = e_1^\ast e_1 = 1$.
Consequently,
\begin{equation*}
    \alpha_1(r) = e_1 r e_1^\ast,
    \quad
    r\in L_0.
\end{equation*}

Let $N_1(\gamma)$ denote the number of occurrences of $e_1$ in a path $\gamma$.
As in Lemma~\ref{lem:alg_action}, the map
\begin{equation*}
    \delta(\alpha\beta^\ast)
    :=
    (N_1(\alpha)-N_1(\beta))\alpha\beta^\ast
\end{equation*}
defines a derivation of $L_0$.

For $r=\alpha\beta^\ast \in L_0$, with $|\alpha|=|\beta|$, one has
\begin{equation*}
\delta(e_1re_1^\ast)
=
\delta(e_1\alpha\beta^\ast e_1^\ast)
=
\bigl(N_1(e_1\alpha)-N_1(e_1\beta)\bigr)e_1\alpha\beta^\ast e_1^\ast
=
e_1\delta(r)e_1^\ast.
\end{equation*}
Hence $\delta(\alpha_1(r))=\alpha_1(\delta(r))$, and therefore the hypothesis of Corollary~\ref{cor:Z_lift_from_generator_comm} holds.

It follows that $\delta$ admits a lift
\begin{equation*}
    \hd: L_{\mathbb C}(1,2) \to L_{\mathbb C}(1,2)
\end{equation*}
given on the spanning elements $\alpha\beta^\ast$ by
\begin{equation*}
    \hd(\alpha\beta^\ast)
    =
    (N_1(\alpha)-N_1(\beta))\alpha\beta^\ast.
\end{equation*}
In particular,
\begin{equation*}
    \hd(e_1)=e_1,
    \quad
    \hd(e_2)=0,
    \quad
    \hd(e_1^\ast)=-e_1^\ast,
    \quad
    \hd(e_2^\ast)=0.
\end{equation*}
In the notation of Section~\ref{sec:L(1,2)}, where graded derivations of $\LConetwo$ are parametrized by matrices
\begin{equation*}
A=
\begin{pmatrix}
a_1 & b\\
c & a_2
\end{pmatrix}
\in M_2(L_0),
\end{equation*}
via
\begin{equation*}
\begin{aligned}
    \delta_A(e_1) &= e_1a_1 + e_2c,
    &\quad
    \delta_A(e_2) &= e_1b + e_2a_2,
    \\
    \delta_A(e_1^\ast) &= -a_1e_1^\ast - be_2^\ast,
    &\quad
    \delta_A(e_2^\ast) &= -ce_1^\ast - a_2e_2^\ast,
\end{aligned}
\end{equation*}
the derivation $\hd$ equals $\delta_A$ for
\begin{equation*}
    A =
    \begin{pmatrix}
        1 & 0\\
        0 & 0
    \end{pmatrix}.
\end{equation*}
\end{example}

The preceding example suggests a structural constraint on derivations commuting with $\alpha_1$. 
More generally, we have the following observation:

\begin{lemma}\label{lem:delta.alpha.com}
    Let $\delta \in \der(L_0)$. 
    Then $\delta \circ \alpha_1 = \alpha_1 \circ \delta$ if and only if there exists $a \in L_0$ such that
    \begin{equation*}
        \delta = \delta_A\vert_{L_0},
        \quad
        A =
        \begin{pmatrix}
            0 & 0 \\
            0 & a
        \end{pmatrix},
    \end{equation*}
    using the notation of Section~\ref{sec:L(1,2)}.
\end{lemma}
\begin{proof}
Assume first that $\delta \circ \alpha_1 = \alpha_1 \circ \delta$.
Then Corollary~\ref{cor:Z_lift_from_generator_comm} implies that $\delta$
admits a graded lift $\hd \in \der_{\mathrm{gr}}(L_{\mathbb C}(1,2))$.
According to Lemma~\ref{lem:der_L(1,2)}, this lift is of the form $\hd = \delta_A$
for some
\begin{equation*}
    A =
    \begin{pmatrix}
        a_1 & b \\
        c & a_2
    \end{pmatrix}
    \in M_2(L_0).
\end{equation*}
Applying the formulas for $\delta_A(e_1)$ and $\delta_A(e_1^\ast)$ from Lemma~\ref{lem:der_L(1,2)}, one finds that
\begin{equation*}
\begin{aligned}
    0
    &=
    (\delta \circ \alpha_1)(r) - (\alpha_1 \circ \delta)(r)
    \\
    &=
    \delta_A(e_1 r e_1^\ast) - e_1 \delta_A(r) e_1^\ast 
    \\
    &=
    \delta_A(e_1) r e_1^\ast + e_1 r \delta_A(e_1^\ast) 
    \\
    &=
    (e_1 a_1 + e_2 c) r e_1^\ast - e_1 r (a_1 e_1^\ast + b e_2^\ast) 
    \\
    &=
    e_1 [a_1,r] e_1^\ast + e_2 c r e_1^\ast - e_1 r b e_2^\ast .
\end{aligned}
\end{equation*}
It follows that $b=c=0$ and $[a_1,r]=0$ for all $r \in L_0$.
Since $Z(L_0) \cong \mathbb{C}$ by Lemma~\ref{lem:L_0}, $a_1=\lambda$ for some $\lambda\in\mathbb{C}$.
Thus
\begin{equation*}
    A =
    \begin{pmatrix}
        \lambda & 0 \\
        0 & a_2
    \end{pmatrix}.
\end{equation*}
The claim now follows from the fact that $\delta_A\vert_{L_0} = \delta_{A - \lambda \id_2}\vert_{L_0}$.

Conversely, let
\begin{equation*}
    A =
    \begin{pmatrix}
        0 & 0 \\
        0 & a
    \end{pmatrix}
\end{equation*}
for some $a \in L_0$, and set $\delta := \delta_A\vert_{L_0}$.
By Lemma~\ref{lem:der_L(1,2)}, $\delta_A(e_1) = \delta_A(e_1^\ast) = 0$, and therefore, for every $r \in L_0$,
\begin{equation*}
    \delta(\alpha_1(r))
    =
    \delta_A(e_1 r e_1^\ast)
    =
    e_1 \delta_A(r) e_1^\ast
    =
    \alpha_1(\delta(r)).
\end{equation*}
Thus $\delta \circ \alpha_1 = \alpha_1 \circ \delta$.
\end{proof}

\begin{remark}
The conclusions above do not depend on the particular choice of module frame system. 
Different choices lead to different associated factor systems, but the resulting structural statements remain the same. 
In practice, however, the complexity of the formulas can vary substantially, and suitable choices may simplify computations considerably.
\end{remark}

\subsection{The Atiyah sequence of a strongly graded ring}\label{sec:Atiyah}

Let $S$ be a strongly $G$-graded ring with principal component $R$.
In this section we associate an Atiyah-type sequence to $S$ and analyze~its constituent  Lie rings, providing a framework in which classical Riemannian concepts - such as Levi–Civita connections - can be extended to this algebraic context.

Denote by
\begin{align*}
    \der_{\mathrm{gr}}(S)
    :=
    \{ \delta \in \der(S) : (\forall g \in G) \, \delta(S_g) \subseteq S_g \}
\end{align*}
the Lie subalgebra of $\der(S)$ consisting of graded derivations.
The restriction map
\begin{equation}\label{eq:res}
    \mathrm{res}: \der_{\mathrm{gr}}(S) \to \der(R),
    \quad
    \delta \mapsto \delta \vert_R
\end{equation}
has kernel
\begin{align*}
    \gau(S)
    :=
    \{\delta \in \der_{\mathrm{gr}}(S): \delta \vert_R = 0\}.
\end{align*}
Let $\der_\mathrm{ext}(R)$ be the Lie subalgebra of $\der(R)$ consisting of derivations that admit a graded lift to $S$.
By definition, $\der_\mathrm{ext}(R)$ coincides with the image of the restriction map~\eqref{eq:res}.
This yields the following short exact sequence of Lie rings,
which we refer to as the \emph{Atiyah sequence} of~$S$:
\begin{equation}\label{eq:AtiyahNC}
    0 \longrightarrow \gau(S) \longrightarrow \der_{\mathrm{gr}}(S)
    \longrightarrow \der_\mathrm{ext}(R) \longrightarrow 0.
\end{equation}

Our aim is to make this sequence more explicit by providing concrete descriptions of its terms.
For any factor system of $S$, Theorem~\ref{thm:liftder}~(2) yields the characterization
\begin{align*}
    \der_\mathrm{ext}(R)
    =
    \{ \delta \in \der(R) : \delta \text{ satisfies the assumptions of Theorem~\ref{thm:liftder}~(2)} \}.
\end{align*}
Moreover, for a skew-group ring $S = R \rtimes_\alpha G$, it is readily seen that $\gau(S)$ is isomorphic to the Abelian Lie ring
\begin{equation*}
    Z^1_\beta(G,Z(R))
    :=
    \{\eta: G \to Z(R) : (\forall g,h \in G) \, \eta(gh) = \eta(g) + \beta_g(\eta(h))\},
\end{equation*}
of crossed $Z(R)$-valued homomorphisms on $G$
(see Remark~\ref{rem:skew}).
Here $\beta: G \to \aut(Z(R))$~denotes the action induced by $S$ (\cf~the discussion around~\eqref{eq:mod_str});
in the skew-group case one has $\beta_g = \alpha_g\vert_{Z(R)}$ for all $g\in G$.
We proceed to show that this identification extends to arbitrary strongly
$G$-graded rings.

%To this end, recall that $S$ induces a $G$-action on $Z(R)$ via a homomorphism $\alpha: G \to \aut(Z(R))$, determined by 
%\begin{equation}\label{eq:mod_str}
%    s z = \alpha_g(z) s
%\end{equation}
%for all $g \in G$, $x \in S_g$, and $z \in Z(R)$ (see, e.g.,~\cite[Lem.~I.3.12]{NaOy82})
%With this in place, we can state the following theorem:

\begin{theorem}
\label{thm:gau}
    The Lie ring $\gau(S)$ is isomorphic to $Z^1_\beta(G,Z(R))$.
\end{theorem}

The proof proceeds by establishing two mutually inverse constructions:

\begin{lemma}\label{lem:gau-to-cocycle}
    Each $\delta \in \gau(S)$ determines an element
    $\eta_\delta \in Z^1_\beta(G,Z(R))$.
\end{lemma}
\begin{proof}
Let $\delta \in \gau(S)$.
Since $\delta$ vanishes on $R$ and is graded, its restriction to each
homogeneous component is an $R$-bimodule homomorphism.
By~\cite[Lem.~I.3.11.2]{NaOy82}, it follows that for~each $g \in G$ there
exists a unique element $\eta_\delta(g) \in Z(R)$ such that
$\delta\vert_{S_g}(s) = \eta_\delta(g)s$ for all $s \in S_g$.
Thus $\delta$ determines a map $\eta_\delta: G \to Z(R)$.

To verify that $\eta_\delta$ satisfies the crossed homomorphism property, let $g,h \in G$ and let $s \in S_g$, and $t \in S_h$.
Then $\delta(st) = \eta_\delta(st) xy$.
On the other hand, using the Leibniz rule and~\eqref{eq:mod_str}
\begin{equation*}
    \delta(st)
    =
    \delta(s) t + s \delta(t)
    =
    \eta_\delta(g) st + s\eta_\delta(h) t
    =
    (\eta_\delta(g) + \beta_g(\eta_\delta(h))) st.
\end{equation*}
Comparing the two expressions for $\delta(st)$ and using strong gradedness
of $S$ yields
\begin{equation*}
    \eta_\delta(gh) = \eta_\delta(g) + \beta_g(\eta_\delta(h)),
\end{equation*}
\ie, $\eta_\delta \in Z^1_\beta(G,Z(R))$.
\end{proof}

\begin{lemma}\label{lem:cocycle-to-gau}
Each $\eta \in Z^1_\beta(G,Z(R))$ determines an element
$\delta_\eta \in \gau(S)$.
\end{lemma}
\begin{proof}
Let $\eta \in Z^1_\beta(G,Z(R))$.
Define a map $\delta_\eta: S \to S$ on homogeneous components by
\begin{equation*}
    \delta_\eta(s) := \eta(g) s,
    \quad
    g \in G, s \in S_g.
\end{equation*}
Clearly, $\delta_\eta\vert_R = 0$.
To verify the Leibniz rule, let $g,h \in G$ and let $s \in S_g$ and $t \in S_h$.
Then
\begin{equation*}
    \delta_\eta(st)
    =
    \eta(gh) st
    =
    (\eta(g) + \beta_g(\eta(h))) st.
\end{equation*}
On the other hand,
\begin{equation*}
    \delta_\eta(s) t + s \delta_\eta(t)
    =
    \eta(g) st + s \eta(h)t
    =
    (\eta(g) + \beta_g(\eta(h))) st,
\end{equation*}
where the last equality follows from~\eqref{eq:mod_str}.
Thus $\delta_\eta \in \gau(S)$.
\end{proof}

A direct computation now shows that the constructions in
Lemmas~\ref{lem:gau-to-cocycle} and~\ref{lem:cocycle-to-gau}
are mutually inverse and compatible with the Lie ring structure,
which establishes Theorem~\ref{thm:gau}.

In some situations it is convenient to work with a Lie subring
$\mathfrak{g} \subseteq \der_\mathrm{ext}(R)$.
Pulling back the Atiyah sequence~\eqref{eq:AtiyahNC} along the inclusion $\iota:\mathfrak{g} \hookrightarrow \der_\mathrm{ext}(R)$ yields the exact sequence
\begin{equation}\label{eq:AtiyahNC_pb}
    0 \longrightarrow \gau(S) \longrightarrow \hat{\mathfrak{g}} \longrightarrow \mathfrak{g} \longrightarrow 0,
\end{equation}
where
\begin{equation*}
    \hat{\mathfrak{g}}
    :=
    \iota^\ast(\der_{\mathrm{gr}}(S))
    =
    \{(\delta,\hd) \in \mathfrak  g\times \der_{\mathrm{gr}}(S)
    : \mathrm{res}(\hd) = \delta\}.
\end{equation*}

\begin{example}\label{ex:at_skew}
    Let $S = R \rtimes_\alpha G$ be a skew group ring.
    Consider the Lie subring
    \begin{equation*}
        \mathfrak{g}
        :=
        \{ \delta \in \der(R) : (\forall g \in G) \, \delta \circ \alpha_g = \alpha_g \circ \delta \} \subseteq \der_\mathrm{ext}(R) .
    \end{equation*}
    By Example~\ref{ex:skew_lift}, each $\delta \in \mathfrak{g}$ admits a canonical
    graded lift $\hd \in \der_{\mathrm{gr}}(S)$ given, for all $g \in G$ and $r \in R$, by $\hd(r u_g)=\delta(r)u_g$.
    Moreover, the assignment
    \begin{equation*}
        \mathfrak{g} \to \der_{\mathrm{gr}}(S),
        \quad
        \delta \mapsto \hd,
    \end{equation*}
    is a Lie ring homomorphism.
    Consequently, the pull-back sequence~\eqref{eq:AtiyahNC_pb} splits in this case.
\end{example}

%For the next example, recall that each element in a complex group ring $\mathbb{C}[G]$ can be uniquely written as a sum $\sum_{g\in G} f_g u_g$ with only finitely many non-zero coefficients $f_g \in \C$ and the Dirac functions
%\begin{equation*}
%    u_g :G \to \mathbb{C}, 
%    \quad 
%    u_g(g')
%    :=
%    \begin{cases}
%        1 \quad \text{if} \quad g = g'
%        \\
%        0 \quad \text{otherwise}.
%    \end{cases}
%\end{equation*}

\begin{example}\label{ex:heisenberg_der}
Let $H_3$ denote the discrete Heisenberg group, realized as the semidirect product $\mathbb{Z}^2 \rtimes \mathbb{Z}$, where the action of $\mathbb{Z}$ on $\mathbb{Z}^2$ is given by $k.(m,n) := (m, km+n)$.
Then the group ring $\C[H_3]$ admits the description as the skew group ring
\begin{equation*}
    \C[H_3] \cong R \rtimes_\alpha \mathbb Z,
    \quad
    R := \C[\mathbb Z^2] \cong \C[u^{\pm1},v^{\pm1}],
\end{equation*}
where $\alpha \in \aut(R)$ is determined by $\alpha(u) = uv$ and $\alpha(v) = v$.

Since $R$ is commutative, Corollary~\ref{cor:Z_crossed_lift_comm} implies
\begin{equation*}
    \der_\mathrm{ext}(R)
    =
    \{\delta \in \der(R) : \delta \circ \alpha = \alpha \circ \delta\}.
\end{equation*}
Every derivation of $R$ is of the form
\begin{equation*}
    \delta = f \delta_u + g \delta_v,
    \quad
    f,g \in R,
\end{equation*}
where $\delta_u := u \frac{\partial}{\partial u}$ and $\delta_v := v \frac{\partial}{\partial v}$.
We claim that
\begin{equation*}
    \der_\mathrm{ext}(R)
    =
    \C[v^{\pm1}] \delta_u .
\end{equation*}

Indeed, evaluating the commutation condition $\delta\circ\alpha=\alpha\circ\delta$ at $v$ yields
\begin{equation*}
    \delta(\alpha(v)) = \delta(v) = gv,
    \quad
    \alpha(\delta(v)) = \alpha(gv) = \alpha(g)v,
\end{equation*}
so $g=\alpha(g)$, and hence $g\in\C[v^{\pm1}]$.

Evaluating at $u$ yields
\begin{equation*}
    \delta(\alpha(u))
    =
    \delta(uv)
    =
    (f+g)uv,
    \qquad
    \alpha(\delta(u))
    =
    \alpha(f)uv,
\end{equation*}
so $\alpha(f)=f+g$.
Iterating yields $\alpha^n(f)=f+ng$.
Consider the decomposition
\begin{equation*}
    R = \bigoplus_{k\in\mathbb Z} \C[v^{\pm1}] u_k
\end{equation*}
and let $p_0: R \to \C[v^{\pm1}]$ denote the projection onto the $u_0$–component.
Since $\alpha(u)=uv$ preserves the $u$–degree, one has $p_0 \circ \alpha^n = p_0$.
Hence
\begin{equation*}
    p_0(f)
    =
    p_0(\alpha^n(f))
    =
    p_0(f+ng)
    =
    p_0(f)+ng,
\end{equation*}
which forces $g=0$.
Thus $\alpha(f)=f$, and hence $f\in\C[v^{\pm1}]$.

Consequently, 
\begin{equation*}
    \der_\mathrm{ext}(R)
    =
    \C[v^{\pm1}] \delta_u,
\end{equation*}
as claimed.
In particular, the Atiyah sequence~\eqref{eq:AtiyahNC} for $S=\C[H_3]$ takes the form
\begin{equation*}
    0 \longrightarrow R \longrightarrow \der_{\mathrm{gr}}(S) \longrightarrow \C[v^{\pm1}] \delta_u \longrightarrow 0,
\end{equation*}
which splits by Example~\ref{ex:at_skew}.
\end{example}

\begin{example}
Let $E$ be a finite directed graph consisting of a single cycle of length $n$,
\begin{equation*}
    v_1 \xrightarrow{e_1} v_2 \xrightarrow{e_2} \cdots
    \xrightarrow{e_{n-1}} v_n \xrightarrow{e_n} v_1 .
\end{equation*}
The associated Leavitt path algebra admits the well-known description
\begin{equation*}
    L_{\mathbb C}(E) \cong M_n(\mathbb C[t,t^{-1}]).
\end{equation*}
The degree-zero component is generated by the vertex idempotents,
\begin{equation*}
    L_0 = \operatorname{span}_{\mathbb{C}}\{v_1,\dots,v_n\}.
\end{equation*}
As the vertices are mutually orthogonal idempotents, it follows that
\begin{equation*}
L_0 \cong \mathbb{C}^n,
\end{equation*}
which is a finite commutative $C^\ast$-algebra. 
Consequently, every derivation of $L_0$ is trivial. 
It follows that the graded derivations of $L_{\mathbb{C}}(E)$ coincide with the gauge derivations. 
More precisely, by Theorem~\ref{thm:gau} one obtains
\begin{equation*}
    \der_{\mathrm{gr}}(L_{\mathbb{C}}(E)) \cong \mathbb{C}^n .
\end{equation*}
\end{example}

\begin{example}\label{ex:L(1,2)_atiyah}
Let $E$ be the graph with one vertex $v$ and two loop edges $e_1,e_2$.
The associated Leavitt path algebra $L_{\mathbb C}(1,2)$ is strongly $\mathbb Z$-graded but not a crossed product (see Section~\ref{lem:notcrossed}).

Lemma~\ref{lem:L_0} shows that $Z(L_0)\cong\mathbb{C}$. 
Hence Theorem~\ref{thm:gau} implies that $\gau(L_{\mathbb{C}}(1,2)) \cong \mathbb{C}$, and the Atiyah sequence~\eqref{eq:AtiyahNC} for $L_{\mathbb C}(1,2)$ becomes
\begin{equation*}
    0 \longrightarrow \mathbb{C} 
    \longrightarrow \der_{\mathrm{gr}}(L_{\mathbb{C}}(1,2))
    \longrightarrow \der_{\mathrm{ext}}(L_0)
    \longrightarrow 0.
\end{equation*}

Corollary~\ref{cor:lieaut} identifies $\der_{\mathrm{gr}}(L_{\mathbb{C}}(1,2))$ with $M_2(L_0)$ via the assignment $A \mapsto \delta_A$. 
The kernel of the projection in the Atiyah sequence corresponds to scalar matrices, and therefore
\begin{equation*}
    \der_{\mathrm{ext}}(L_0)\cong M_2(L_0)/\mathbb{C}.
\end{equation*}
In particular, $\der_{\mathrm{ext}}(L_0)\neq 0$.

More concretely, consider the Lie subalgebra
\begin{equation*}
    \mathfrak{g}_{\alpha_1}
    =
    \left\{
    \delta_A|_{L_0} :
    A =
    \begin{pmatrix}
        0 & 0 \\
        0 & a
    \end{pmatrix},
    \, a \in L_0
    \right\}
    \subseteq \der(L_0).
\end{equation*}
By Lemma~\ref{lem:delta.alpha.com}, these derivations commute with $\alpha_1$
introduced in Example~\ref{ex:L(1,2)_factorsystem}.
Corollary~\ref{cor:Z_lift_from_generator_comm} therefore shows that they admit graded lifts to $L_{\mathbb C}(1,2)$.
Thus $\mathfrak{g}_{\alpha_1}$ identifies with a Lie subalgebra of
$\der_{\mathrm{ext}}(L_0)$.
\end{example}

\subsection{The Atiyah curvature of a strongly graded ring}\label{sec:atiyah_curvature}

Let $S$ be a strongly $G$-graded ring with principal component $R$.
In this section we introduce the Atiyah curvature associated with a section of the Atiyah sequence and study its algebraic properties.
This curvature measures the failure of a chosen lift of derivations of $R$ to form a Lie homomorphism.
We then relate it to the induced connections on the homogeneous components and indicate how it enters characteristic-class constructions.

Let $\mathfrak{g}\subseteq \der_\mathrm{ext}(R)$ be a Lie subring, and consider a (not necessarily Lie) section
\begin{equation*}
    \sigma:\mathfrak{g} \to \hat{\mathfrak{g}},
    \quad
    \mathrm{res} \circ \sigma = \mathrm{id}_{\mathfrak{g}},
\end{equation*}
of the pull-back sequence~\eqref{eq:AtiyahNC_pb}.

\begin{definition}\label{def:atiyah_curvature}
The \emph{Atiyah curvature} (or \emph{Lie curvature}) of the section $\sigma$ is the map
\begin{equation*}
    F_\sigma: \mathfrak{g} \times \mathfrak{g} \to \gau(S),
    \quad
    F_\sigma(\delta_1,\delta_2)
    :=
    [\sigma(\delta_1),\sigma(\delta_2)]-\sigma([\delta_1,\delta_2]).
\end{equation*}
\end{definition}

It is readily verified that $F_\sigma$ is biadditive and alternating, \ie,
\begin{equation*}
    F_\sigma(\delta,\delta) = 0,
    \quad
    F_\sigma(\delta_1,\delta_2) = -F_\sigma(\delta_2,\delta_1)
\end{equation*}
for all $\delta,\delta_1,\delta_2 \in \mathfrak{g}$.
Moreover, $\sigma$ is a Lie ring homomorphism, equivalently, the pull-back sequence~\eqref{eq:AtiyahNC_pb} splits as Lie rings, if and only if $F_\sigma = 0$.

Recall that the multiplicative curvature $\Delta_\delta \in Z^2_\beta(G,Z(R))$ measures the obstruction to lifting a single derivation $\delta \in \der(R)$ to a graded derivation of $S$ (see Remark~\ref{rem:mult_curv_class}).
For derivations $\delta \in \mathfrak{g}$ such lifts exist by definition.
The Atiyah curvature $F_\sigma$ measures the obstruction to choosing these lifts \emph{coherently} so that the resulting section $\sigma: \mathfrak{g} \to \hat{\mathfrak{g}}$ becomes a Lie homomorphism.

Note that $\der_{\mathrm{gr}}(S)$ acts on $\gau(S)$ by commutators, \ie,
\begin{equation*}
    \delta.\eta:=[\delta,\eta],
    \quad
    \delta \in \der_{\mathrm{gr}}(S), \eta \in \gau(S).
\end{equation*}
Using the Jacobi identity in $\der_{\mathrm{gr}}(S)$ applied to
$\sigma(\delta_1),\sigma(\delta_2),\sigma(\delta_3)$,
for $\delta_1,\delta_2,\delta_3 \in \mathfrak{g}$, and expanding
\begin{equation*}
    [\sigma(\delta_i),\sigma(\delta_j)]
    =
    \sigma([\delta_i,\delta_j]) + F_\sigma(\delta_i,\delta_j),
\end{equation*}
one obtains the following abstract Bianchi identity:

\begin{lemma}\label{lem:bianchi}
$F_\sigma$ satisfies the (Chevalley--Eilenberg) $2$-cocycle identity
\begin{gather*}
    F_\sigma([\delta_1,\delta_2],\delta_3)
    +
    F_\sigma([\delta_2,\delta_3],\delta_1)
    +
    F_\sigma([\delta_3,\delta_1],\delta_2)
    \\
    =
    \sigma(\delta_1).F_\sigma(\delta_2,\delta_3)
    +
    \sigma(\delta_2).F_\sigma(\delta_3,\delta_1)
    +
    \sigma(\delta_3).F_\sigma(\delta_1,\delta_2).
\end{gather*}
\end{lemma}

%Via Theorem~\ref{thm:gau}, $F_\sigma$ may be regarded as a $Z^1_\beta(G,Z(R))$-valued $2$-cocycle on $\mathfrak{g}$.

To relate this abstract construction to the graded structure of $S$, we now express the Atiyah curvature in terms of the induced connections on the homogeneous components.

Each lifted derivation $\hat{\delta} := \sigma(\delta)$, for $\delta \in \mathfrak{g}$,
restricts to a map
\begin{equation*}
    \nabla_{\delta,g} := \sigma(\delta)\vert_{S_g}: S_g \to S_g,
    \quad
    g \in G.
\end{equation*}
Thus the family $(\nabla_{\delta,g})_{\delta\in\mathfrak{g}}$ defines a $\mathfrak{g}$-bimodule connection $\nabla_g$ on $S_g$.
Its curvature is given by
\begin{equation*}
    R_{\nabla_g}(\delta_1,\delta_2)
    :=
    [\nabla_{\delta_1,g},\nabla_{\delta_2,g}]
    -
    \nabla_{[\delta_1,\delta_2],g},
    \quad
    \delta_1,\delta_2\in\mathfrak{g}.
\end{equation*}
For $s\in S_g$ one then has
\begin{equation*}
    F_\sigma(\delta_1,\delta_2)(s)
    =
    R_{\nabla_g}(\delta_1,\delta_2)(s).
\end{equation*}
Hence the Atiyah curvature records the curvature of the induced connections on the homogeneous components.

Using a module frame system for $S$, the induced connection and its curvature admit the following explicit matrix description.

\begin{lemma}\label{lem:atiyah_curvature_factor}
    Let $S$ be a strongly $G$-graded ring with principal component $R$, and let $(n,\alpha,\omega)$ be a factor system associated with a module frame system $(x_g,y_g)_{g \in G}$.
    Let $\mathfrak{g} \subseteq \der_{\mathrm{ext}}(R)$ be a Lie subring, and let $\sigma:\mathfrak{g} \to \hat{\mathfrak{g}}$ be a section of the corresponding pull-back extension~\eqref{eq:AtiyahNC_pb}.

    For $\delta\in\mathfrak{g}$ and $g\in G$, define
    \begin{equation*}
        \eta_\sigma(\delta,g)
        :=
        \sigma(\delta)(x_g)y_g^t
        \in M_{n_g}(R)\alpha_g(1).
    \end{equation*}
    Then the induced connection on $S_g$ is given by
    \begin{equation*}
        \nabla_{\delta,g}(u^t x_g)
        :=
        \delta(u)^t x_g + u^t \eta_\sigma(\delta,g) x_g,
        \quad
        u \in M_{n_g,1}(R),
    \end{equation*}
    and its curvature satisfies
    \begin{equation*}%\label{eq:curvature_component_formula}
        R_{\nabla_g}(\delta_1,\delta_2)(u^t x_g)
        =
        u^t \Omega_\sigma(\delta_1,\delta_2,g) x_g,
    \end{equation*}
    where
    \begin{align*}
        \Omega_\sigma(\delta_1,\delta_2,g)
        &:=
        \delta_1(\eta_\sigma(\delta_2,g))
        -
        \delta_2(\eta_\sigma(\delta_1,g))
        \\
        &\quad
        +
        [\eta_\sigma(\delta_1,g),\eta_\sigma(\delta_2,g)]
        -
        \eta_\sigma([\delta_1,\delta_2],g)
        \in M_{n_g}(R)\alpha_g(1).
    \end{align*}
\end{lemma}
\begin{proof}
The formula for $\nabla_{\delta,g}$ follows from Theorem~\ref{thm:liftder}.

Now let $\delta_1,\delta_2 \in \mathfrak{g}$ and let $u\in M_{n_g,1}(R)$.
A direct computation gives
\begin{align*}
    \nabla_{\delta_1,g}\nabla_{\delta_2,g}(u^t x_g)
    &=
    \nabla_{\delta_1,g}(\delta_2(u)^t x_g + u^t\eta_\sigma(\delta_2,g) x_g)
    \\
    &=
    \delta_1(\delta_2(u))^t x_g
    + \delta_2(u)^t\eta_\sigma(\delta_1,g) x_g
    \\
    &\quad
    + \delta_1(u)^t\eta_\sigma(\delta_2,g) x_g
    + u^t \delta_1(\eta_\sigma(\delta_2,g)) x_g
    \\
    &\quad
    + u^t \eta_\sigma(\delta_2,g) \eta_\sigma(\delta_1,g) x_g.
\end{align*}
Interchanging $\delta_1$ and $\delta_2$ and subtracting, the mixed terms cancel, and one obtains
\begin{align*}
    R_{\nabla_g}(\delta_1,\delta_2)(u^t x_g)
    &=
    u^t(
    \delta_1(\eta_\sigma(\delta_2,g))
    -
    \delta_2(\eta_\sigma(\delta_1,g))
    \\
    &\qquad\qquad
    +
    [\eta_\sigma(\delta_1,g),\eta_\sigma(\delta_2,g)]
    -
    \eta_\sigma([\delta_1,\delta_2],g)
    ) x_g.
    \qedhere
\end{align*}
\end{proof}

\begin{remark}
Since $F_\sigma(\delta_1,\delta_2) \in \gau(S)$ for all $\delta_1,\delta_2 \in \mathfrak{g}$, its restriction to $S_g$ is an $R$-bimodule endomorphism and therefore is given by multiplication by a unique central element of $R$.
\end{remark}

The Grassmann connection on $S_g$ is of the same form as the
connections appearing in Lemma~\ref{lem:atiyah_curvature_factor}, with
\begin{equation*}
    \eta(\delta,g) = \delta(\alpha_g(1)).
\end{equation*}
The corresponding curvature can therefore be obtained by specializing the general formula of the lemma.

\begin{remark}\label{ex:grassmann_curvature}
Let $S$ be a strongly $G$-graded ring with principal component $R$,
and let $(n,\alpha,\omega)$ be a factor system associated with a module frame system
$(x_g,y_g)_{g \in G}$.
Let $\mathfrak{g} \subseteq \der(R)$ be a Lie subring.
For each $g \in G$ and each $\delta \in \mathfrak{g}$ consider the left covariant derivative
\begin{equation*}
    \nabla_{\delta,g}: S_g \to S_g,
    \quad
    \nabla_{\delta,g}(u^t x_g)
    :=
    \delta(u)^t x_g + u^t \delta(\alpha_g(1))x_g,
    \quad
    u \in M_{n_g,1}(R).
\end{equation*}
By Remark~\ref{rem:grassmann}, this is the covariant derivative along $\delta$
induced by the Grassmann connection on the finitely generated projective
left $R$-module $S_g$.

Its curvature is given by
\begin{equation*}
    R_{\nabla_g}(\delta_1,\delta_2)(u^t x_g)
    =
    u^t [\delta_1(\alpha_g(1)),\delta_2(\alpha_g(1))] x_g
\end{equation*}
for all $\delta_1,\delta_2 \in \mathfrak{g}$ and $u\in M_{n_g,1}(R)$.

Indeed, let $u \in M_{n_g,1}(R)$.
A direct computation gives
\begin{align*}
    \nabla_{\delta_1,g}\nabla_{\delta_2,g}(u^t x_g)
    &=
    \nabla_{\delta_1,g}(\delta_2(u)^t x_g + u^t\delta_2(\alpha_g(1)) x_g)
    \\
    &=
    \delta_1\delta_2(u)^t x_g
    + \delta_2(u)^t\delta_1(\alpha_g(1)) x_g
    + \delta_1(u)^t\delta_2(\alpha_g(1)) x_g
    \\
    &\quad
    + u^t\delta_1\delta_2(\alpha_g(1)) x_g
    + u^t\delta_2(\alpha_g(1))\delta_1(\alpha_g(1)) x_g.
\end{align*}
Interchanging $\delta_1$ and $\delta_2$ and subtracting, one obtains
\begin{align*}
    [\nabla_{\delta_1,g},\nabla_{\delta_2,g}](u^t x_g)
    &=
    [\delta_1,\delta_2](u)^t x_g
    + u^t [\delta_1,\delta_2](\alpha_g(1)) x_g
    \\
    &\quad
    + u^t[\delta_1(\alpha_g(1)),\delta_2(\alpha_g(1))]x_g .
\end{align*}
Since
\begin{equation*}
    \nabla_{[\delta_1,\delta_2],g}(u^t x_g)
    =
    [\delta_1,\delta_2](u)^t x_g
    + u^t[\delta_1,\delta_2](\alpha_g(1))x_g,
\end{equation*}
it follows that
\begin{equation*}
    R_{\nabla_g}(\delta_1,\delta_2)(u^t x_g)
    =
    u^t[\delta_1(\alpha_g(1)),\delta_2(\alpha_g(1))]x_g.
\end{equation*}
\end{remark}

\begin{remark}\label{rem:grassmann_flat}
If $\delta \in \der(R)$ commutes with $\alpha_g$, then $\delta(\alpha_g(1)) = \alpha_g(\delta(1))=0$, and hence
\begin{equation*}
    R_{\nabla_g}(\delta_1,\delta_2) = 0
\end{equation*}
for all $\delta_1,\delta_2 \in \der(R)$.
In particular, the Grassmann connection on $S_g$ is flat along any pair of derivations that commute with $\alpha_g$.
\end{remark}

We next describe how the Atiyah curvature changes when the section of the Atiyah sequence is replaced:

\begin{lemma}\label{lem:curvature_change_section}
Let $S$ be a strongly $G$-graded ring with principal component $R$,
and let~$\mathfrak{g} \subseteq \der_{\mathrm{ext}}(R)$ be a Lie subring.
Suppose $\sigma,\sigma': \mathfrak{g} \to \hat{\mathfrak{g}}$ are sections of the pull-back sequence~\eqref{eq:AtiyahNC_pb}.
Then there exists a map $\psi: \mathfrak{g} \to \gau(S)$ such that $\sigma' = \sigma + \psi$.
Moreover, their Atiyah curvatures satisfy
\begin{equation*}
    F_{\sigma'}(\delta_1,\delta_2)
    =
    F_\sigma(\delta_1,\delta_2)
    +
    d_\sigma\psi(\delta_1,\delta_2)
    +
    [\psi(\delta_1),\psi(\delta_2)],
\end{equation*}
for all $\delta_1,\delta_2 \in \mathfrak{g}$, where
\begin{equation*}
    d_\sigma\psi(\delta_1,\delta_2)
    =
    [\sigma(\delta_1),\psi(\delta_2)]
    -
    [\sigma(\delta_2),\psi(\delta_1)]
    -
    \psi([\delta_1,\delta_2]).
\end{equation*}
\end{lemma}
\begin{proof}
Since $\mathrm{res} \circ \sigma = \mathrm{res} \circ \sigma' = \mathrm{id}_{\mathfrak{g}}$,
their difference takes values in the kernel of $\mathrm{res}$.
Therefore $\psi(\delta) := \sigma'(\delta)-\sigma(\delta)$ belongs to $\gau(S)$
for all $\delta \in \mathfrak{g}$, showing $\sigma' = \sigma + \psi$.

Using the definition of $F_\sigma(\delta_1,\delta_2)$, and substituting $\sigma'=\sigma+\psi$, a direct computation yields
\begin{equation*}
    F_{\sigma'}
    =
    F_\sigma + d_\sigma\psi + [\psi,\psi],
\end{equation*}
which is the claimed identity.
\end{proof}

\begin{remark}\label{rem:curvature_change_matrix}
Let $\sigma,\sigma'$ be sections as in Lemma~\ref{lem:curvature_change_section}
and write $\sigma' = \sigma + \psi$ with $\psi: \mathfrak{g} \to \gau(S)$.
For $g\in G$ and $\delta \in \mathfrak{g}$, define
\begin{equation*}
    A_\psi(\delta,g) := \psi(\delta)(x_g)y_g^t \in M_{n_g}(R)\alpha_g(1).
\end{equation*}
Then the corresponding connection matrices satisfy
\begin{equation*}
    \eta_{\sigma'}(\delta,g)
    =
    \eta_\sigma(\delta,g)
    +
    A_\psi(\delta,g).
\end{equation*}
Consequently, the curvature matrices are related by
\begin{equation*}
    \Omega_{\sigma'}
    =
    \Omega_\sigma
    +
    d_{\eta_\sigma}A_\psi
    +
    [A_\psi,A_\psi],
\end{equation*}
where $d_{\eta_\sigma}$ denotes the covariant differential determined by
$\eta_\sigma$.
Thus changing the section of the Atiyah sequence modifies the induced
connections on the components $S_g$ by a gauge term and transforms the
curvature according to the usual gauge transformation law.
\end{remark}

For the following discussion we assume that all algebras are defined over
a field $k$ of characteristic zero.

The Atiyah curvature also fits naturally into the framework of
Lecomte's generalization of the Chern--Weil homomorphism for Lie algebra
extensions~\cite{Lec85} (see also~\cite{Wa18}).
To apply this construction, let $V$ be a $\mathfrak{g}$-module and view $V$ as a $\hat{\mathfrak{g}}$-module via the restriction map.
Consider
\begin{equation*}
    \Sym^p(\gau(S),V)
\end{equation*}
the space of symmetric $p$-linear maps $f:\gau(S)^p \to V$.
Let $\Sym^p(\gau(S),V)^{\hat{\mathfrak{g}}}$ denote the subset consisting of those maps which are invariant under
the natural action of $\hat{\mathfrak{g}}$, \ie,
\begin{equation*}
    x.f(\eta_1,\ldots,\eta_p)
    =
    \sum_{i=1}^p
    f(\eta_1,\ldots,[x,\eta_i],\ldots,\eta_p)
\end{equation*}
for all $x \in \hat{\mathfrak{g}}$ and $\eta_1,\ldots,\eta_p \in \gau(S)$.

Lecomte’s Chern--Weil construction then yields, for each $p\in\mathbb N_0$,
a natural map
\begin{equation*}
    C_p:
    \Sym^p(\gau(S),V)^{\hat{\mathfrak{g}}}
    \to
    H^{2p}(\mathfrak{g},V),
    \quad
    f \mapsto \frac{1}{p!}[f_\sigma],
\end{equation*}
which is independent of the choice of the section $\sigma$.
Here $f_\sigma \in C^{2p}(\mathfrak{g},V)$ is the cochain obtained by
evaluating $f$ on $p$ copies of the curvature $F_\sigma$ and then
alternating, \ie,
\begin{equation*}
    f_\sigma(\delta_1,\ldots,\delta_{2p})
    =
    \frac{1}{2^p}
    \sum_{\pi \in S_{2p}}
    \operatorname{sgn}(\pi)
    f
    (
    F_\sigma(\delta_{\pi(1)},\delta_{\pi(2)}),
    \ldots,
    F_\sigma(\delta_{\pi(2p-1)},\delta_{\pi(2p)})
    ).
\end{equation*}

\begin{example}
We once more consider the graph $E$ with one vertex $v$ and two loop edges $e_1,e_2$, and the associated Leavitt path algebra $L_{\mathbb{C}}(1,2)$.
Recall that $L_{\mathbb C}(1,2)$ is strongly $\mathbb Z$-graded but not a crossed product (see Section~\ref{sec:L(1,2)}).
Our aim is to use this algebra to construct a Lie algebra cohomology class arising from the Atiyah curvature and Lecomte's Chern--Weil homomorphism.

As shown in Example~\ref{ex:L(1,2)_atiyah}, the Atiyah sequence for $L_{\mathbb{C}}(1,2)$ yields a central extension of Lie rings
\begin{equation*}
    0 \longrightarrow \mathbb{C}
    \longrightarrow \der_{\mathrm{gr}}(L_{\mathbb{C}}(1,2))
    \longrightarrow \der_{\mathrm{ext}}(L_0)
    \longrightarrow 0.
\end{equation*}

Let $\mathfrak{g} \subseteq \der_{\mathrm{ext}}(L_0)$ be a Lie subring and let
$\sigma: \mathfrak{g} \to \der_{\mathrm{gr}}(L_{\mathbb C}(1,2))$
be a section of the corresponding pull-back sequence~\eqref{eq:AtiyahNC_pb}.
Since the kernel of the extension is central and identified with $\mathbb{C}$, the Atiyah curvature
\begin{equation*}
    F_\sigma(\delta_1,\delta_2)
    =
    [\sigma(\delta_1),\sigma(\delta_2)]
    -
    \sigma([\delta_1,\delta_2]),
    \qquad
    \delta_1,\delta_2 \in \mathfrak{g},
\end{equation*}
takes values in $\mathbb{C}$ and defines an alternating $2$-cocycle on $\mathfrak{g}$.

To apply Lecomte's Chern--Weil construction, consider the trivial $\hat{\mathfrak g}$-module $V = \mathbb{C}$.
For $p=1$, one may choose
\begin{equation*}
    f = \mathrm{id}_{\mathbb{C}}
    \in
    \Sym^1(\gau(L_{\mathbb{C}}(1,2)),\mathbb{C})^{\hat{\mathfrak g}} .
\end{equation*}
The associated Lecomte class is then
\begin{equation*}
    C_1(f)=[F_\sigma]\in H^2(\mathfrak{g},\mathbb{C}).
\end{equation*}
Thus, in this example the first Lecomte class is simply the $2$-cohomology class represented by the pull-back sequence~\eqref{eq:AtiyahNC_pb}
\begin{equation*}
    0 \longrightarrow \mathbb{C} \longrightarrow \hat{\mathfrak{g}} \longrightarrow \mathfrak{g} \longrightarrow 0.
\end{equation*}
The vanishing or nonvanishing of this class determines whether the extension splits.
It remains an interesting problem to determine for which Lie subalgebras $\mathfrak{g} \subseteq \der_{\mathrm{ext}}(L_0)$ this occurs.

\end{example}

%\begin{remark}
%A natural strategy for producing a nontrivial class is to choose a two-dimensional Abelian Lie subalgebra $\mathfrak{g} = \mathbb{C} \delta_1 \oplus \mathbb{C} \delta_2 \subseteq \der_{\mathrm{ext}}(L_0)$.
%Since $\mathfrak{g}$ is Abelian, one has
%\begin{equation*}
%    H^2(\mathfrak{g},\mathbb{C})
%    \cong
%    \Lambda^2\mathfrak{g}^\ast
%    \cong
%    \mathbb{C},
%\end{equation*}
%so any nonzero alternating $2$-cocycle already represents a nontrivial cohomology class.

%In this situation, $F_\sigma(\delta_1,\delta_2) = [\sigma(\delta_1),\sigma(\delta_2)]$,
%because $[\delta_1,\delta_2] = 0$.
%Hence a nontrivial class would arise if one could find commuting
%derivations $\delta_1,\delta_2 \in \der_{\mathrm{ext}}(L_0)$ together with a section $\sigma$ such that their graded lifts fail to commute, \ie, $[\sigma(\delta_1),\sigma(\delta_2)] \neq 0$.
%However, Lemma~\ref{XXX}~shows that this strategy cannot succeed in the case of $L_{\mathbb{C}}(1,2)$.
%Equivalently, one seeks commuting derivations on the principal component $L_0$ whose chosen lifts to graded derivations of $L_{\mathbb{C}}(1,2)$ differ from a Lie homomorphism by a nonzero scalar.

\begin{remark}
Lecomte's Chern--Weil construction provides a useful tool for detecting when a strongly
graded ring is not a crossed product. 
Indeed, if one can exhibit a Lie subalgebra of
derivations for which a nontrivial Lecomte class exists, this shows that the Atiyah
sequence does not admit a Lie section in the given setting, suggesting that the graded
structure is unlikely to arise from a crossed product.
\end{remark}

\appendix

\section{Leavitt path algebras}\label{sec:LPA}

Let $k$ be a field and let $E=(E_0,E_1,s,r)$ be a directed graph with vertex set $E_0$, edge set $E_1$, and source and range maps $s,r:E_1\to E_0$.
In addition, let $E_1^\ast:=\{e^\ast : e\in E_1\}$ be a formal copy of $E_1$, whose elements are called ghost edges.

The \emph{Leavitt path algebra} $L_k(E)$ is the $k$-algebra generated by $E_0\cup E_1\cup E_1^\ast$
subject to the relations, for all $v,w\in E_0$ and $e,f\in E_1$,
\begin{enumerate}[label=(\arabic*),wide=0pt,leftmargin=*]
\item $vw=\delta_{v,w}\,v$.
\item $s(e)e=e=e\,r(e)$.
\item $r(e)e^\ast=e^\ast=e^\ast s(e)$.
\item $e^\ast f=\delta_{e,f}\,r(e)$.
\item For every \emph{regular vertex} $v$ (\ie, $0<|s^{-1}(v)|<\infty$), $\sum_{e\in s^{-1}(v)} ee^\ast = v$.
\end{enumerate}
A Leavitt path algebra is called finite if both $E_0$ and $E_1$ are finite sets. Furthermore, a \emph{real path of length $n$} is a sequence of edges $e_1\cdots e_n$ such that $r(e_i)=s(e_{i+1})$ for $i=1,\ldots,n-1$; analogously, a \emph{ghost path} is a sequence of ghost edges. We let $P_n(E)$ denote the set of all real paths of length $n$.

The algebra $L_k(E)$ carries a natural $\mathbb{Z}$-grading defined on generators by
\begin{equation*}
    \deg(v) = 0, 
    \quad 
    \deg(e) = 1, 
    \quad 
    \deg(e^\ast) = -1,
\end{equation*}
for all $v\in E_0$ and $e\in E_1$. If $k=\mathbb{C}$, then $L_\mathbb{C}(E)$ is a graded ${}^\ast$-algebra with involution defined on generators by
\begin{equation*}
    v^\ast :=v ,
    \quad 
    e^\ast := e^\ast\, \text{(the ghost edge of $e$)}, 
    \quad 
    (e^\ast)^\ast := e,
\end{equation*}
for all $v\in E_0$ and $e\in E_1$.

In this paper we mainly consider Leavitt path algebras that are unital and strongly $\mathbb{Z}$-graded.
By~\cite{h:graded.struct.lpa,LuOi22}, this class corresponds to directed graphs $E$ that are row-finite, have no sinks, and satisfy Condition~(Y); in particular, every vertex emits at least one edge.
Hence, for each $v\in E_0$ and every $n\ge1$, there exists a real path $p$ of length $n$ with $s(p)=v$.

\subsection{The Leavitt path algebra \texorpdfstring{$L_{\complex}(1,2)$}{L(C,1,2)}}\label{sec:L(1,2)}

Let $E$ be the directed graph with one single vertex and two loop edges $e_1,e_2$:

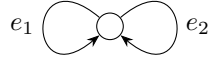
\begin{center}
\begin{tikzpicture}[>=Stealth]
  \node[circle, draw, minimum size=10pt, inner sep=0pt] (v) at (0,0) {};
  \draw[->] (v) to[out=140,in=220,looseness=14] node[left]  {$e_1$} (v);
  \draw[->] (v) to[out=40, in=320,looseness=14] node[right] {$e_2$} (v);
\end{tikzpicture}
\end{center}
The associated Leavitt path algebra $L_{\complex}(1,2) := L_{\complex}(E)$ is the unital ${}^\ast$-algebra over $\complex$ generated by $e_1,e_2$ subject to the relations
\begin{align*}
    e_1^\ast e_1 = e_2^\ast e_2 = 1, 
    \quad
    e_1^\ast e_2 = e_2^\ast e_1 = 0, 
    \quad
    e_1 e_1^\ast + e_2 e_2^\ast = 1.
\end{align*}
It is linearly spanned by elements $\alpha \beta^\ast$, where $\alpha,\beta$ are real paths in $E$.
Writing $|\alpha|$ for the length of a path, the canonical $\integers$-grading is given by
\begin{equation*}
    \deg(\alpha\beta^\ast) = |\alpha|-|\beta| ,
\end{equation*}
and
\begin{equation*}
    L_{\complex}(1,2)=\bigoplus_{n\in\integers} L_n, 
    \quad
    L_n = \operatorname{span}_{\complex}\{\alpha\beta^\ast:\ |\alpha|-|\beta| = n\}.
\end{equation*}
Since $E$ has no sinks, $L_{\complex}(1,2)$ is strongly $\integers$-graded. 

% For $\mu\in\{1,2\}^n$ we set
% \begin{align*}
%     e_{\mu} = e_{\mu_1}e_{\mu_2}\cdots e_{\mu_n}.
% \end{align*}
% \ptcomj{Do we need this notation?}

% Note that for $\alpha\beta^\ast\in L_1$ one may write
% \begin{align*}
%     a = \sum_{k}c_k\alpha_k\beta_k^\ast
% \end{align*}

Next, we note that the principal component of $\LConetwo$ admits a convenient description as a direct limit of matrix algebras:

\begin{lemma}\label{lem:L_0}
    The principal component $L_0$ is the inductive limit
    \begin{equation*}
        L_0 \cong \varinjlim M_{2^n}(\mathbb{C}).
    \end{equation*}
    In particular, $Z(L_0) \cong \mathbb{C}$.
\end{lemma}
\begin{proof}
For $n\ge 0$, let $P_n(E)$ denote the set of paths in $E$ of length $n$.
Since $E$ has two loop~edges, $|P_n(E)| = 2^n$. 
Define
\begin{equation*}
    A_n := \operatorname{span}_{\mathbb{C}}\{\alpha\beta^\ast : \alpha,\beta \in P_n(E)\} \subseteq L_0.
\end{equation*}
Now, let $\alpha,\beta,\gamma,\delta \in P_n(E)$. 
Using the Cuntz--Krieger relations one finds that
\begin{equation*}
    (\alpha\beta^\ast)(\gamma\delta^\ast)
    =
    \begin{cases}
        \alpha\delta^\ast & \text{if } \beta = \gamma,
        \\
        0 & \text{otherwise}.
    \end{cases}
\end{equation*}
Consequently, $\{\alpha\beta^\ast : \alpha,\beta\in P_n(E)\}$ form a system of matrix units indexed by $P_n(E)$, and hence
\begin{equation*}
    A_n \cong M_{|P_n(E)|}(\mathbb{C}) = M_{2^n}(\mathbb{C}).
\end{equation*}
Next we show that $(A_n)_{n\ge0}$ is increasing.
For $\alpha,\beta\in P_n(E)$, the relation $e_1e_1^\ast+e_2e_2^\ast=1$~gives
\begin{equation*}
    \alpha\beta^\ast
    =
    \alpha(e_1e_1^\ast+e_2e_2^\ast)\beta^\ast
    =
    (\alpha e_1) (\beta e_1)^\ast + (\alpha e_2) (\beta e_2)^\ast .
\end{equation*}
Since $\alpha e_i,\beta e_i\in P_{n+1}(E)$ for $i=1,2$, it may be concluded that
$A_n\subseteq A_{n+1}$.

Finally, every homogeneous element of degree zero in
$L_{\mathbb C}(E)$ is a finite linear combination of elements
$\alpha\beta^\ast$ with $|\alpha|=|\beta|$. 
Hence every element of $L_0$
belongs to $A_n$ for some sufficiently large $n$, and therefore
\begin{equation*}
    L_0 = \bigcup_{n\ge0} A_n .
\end{equation*}
From $A_n\cong M_{2^n}(\mathbb{C})$ and $A_n\subseteq A_{n+1}$ it follows that
\begin{equation*}
    L_0 \cong \varinjlim M_{2^n}(\mathbb{C}).
\end{equation*}
For each $n$, $Z(A_n) \cong \mathbb{C}$, and the inclusions
$A_n\subseteq A_{n+1}$ are unital. 
Hence~$Z(L_0) \cong \mathbb C$.
\end{proof}

The preceding lemma shows that $L_0$ is an AF-algebra, hence in particular directly finite. 
This can now be used to show that the degree-one component contains no invertible elements:

\begin{lemma}\label{lem:notcrossed}
    $L_{\mathbb C}(1,2)$ has no invertible homogeneous element of degree $1$. 
    In particular, $L_{\complex}(E)$ is not a crossed product.
\end{lemma}
\begin{proof}
Suppose, for contradiction, that $u \in L_1$ is invertible and set $a := u^{-1} e_1$ and $b := u^{-1} e_2$.
Then $a,b\in L_0$.
Using the defining relations of $L_{\mathbb C}(1,2)$, one obtains
\begin{equation*}
    a^\ast u^\ast u a = 1,
    \quad
    b^\ast u^\ast u b = 1,
    \quad \text{and} \quad 
    a^\ast u^\ast u b = 0.
\end{equation*}
Since $L_0$ is directly finite, the first two identities imply $aa^\ast u^\ast u = 1$ and $bb^\ast u^\ast u = 1$.
Hence
\begin{equation*}
    1
    =
    (aa^\ast u^\ast u) (bb^\ast u^\ast u)
    =
    aa^\ast u^\ast u bb^\ast u^\ast u
    =
    a(a^\ast u^\ast u b) b^\ast u^\ast u
    =
    0,
\end{equation*}
a contradiction. Hence, the homogeneous component $L_1$ does not contain any invertible elements, which implies by Lemma~\ref{lem:j.oi} that $\LConetwo$ is not a crossed product. 
\end{proof}

Thus $L_{\mathbb C}(1,2)$ is not a crossed product.
It is nevertheless useful to describe the structure of the degree-one component as a module over the principal component. 
The following lemma provides such a description:

\begin{lemma}\label{lem:rank2}
    $L_1$ is generated as a free right $L_0$-module by $\{e_1,e_2\}$, \ie, $L_1$ is of rank 2 as a right $L_0$-module.
\end{lemma}
\begin{proof}
Since $e_1e_1^\ast + e_2e_2^\ast = 1$ one has
\begin{equation*}
    x 
    = 
    (e_1e_1^\ast + e_2e_2^\ast)x
    = 
    e_1(e_1^\ast x) + e_2(e_2^\ast x)
\end{equation*}
for any $x\in L_1$. It follows that $\{e_1,e_2\}$ generates $L_1$ as a right $L_0$-module.

To prove linear independence, suppose that $e_1 r + e_2 s = 0$ for some $r,s \in L_0$. Multiplying from the left by $e_1^\ast$ and $e_2^\ast$ gives
\begin{equation*}
    e_1^\ast(e_1 r + e_2 s) = r = 0
    \quad \text{and} \quad
    e_2^\ast(e_1 r + e_2 s) = s = 0,
\end{equation*}
implying that $r = s = 0$, and that $\{e_1,e_2\}$ is a basis of $L_1$.
\end{proof}

We now turn to derivations on $\LConetwo$. 
We begin by showing that $(\complex^{\times})^2$ acts naturally on $\LConetwo$, giving rise to derivations:

\begin{lemma}\label{lem:alg_action}
    There is an algebraic action
    \begin{equation*}
        \alpha:(\mathbb{C}^\times)^2 \to \operatorname{Aut}(L_{\mathbb C}(1,2))
    \end{equation*}
    given on generators by
    \begin{equation*}
    \alpha_{(z,w)}(v)=v,
    \quad
    \alpha_{(z,w)}(e_1)=ze_1,
    \quad \text{and} \quad 
    \alpha_{(z,w)}(e_2)=we_2,
    \end{equation*}

    Moreover, the map $\hd: L_{\mathbb C}(1,2) \to L_{\mathbb C}(1,2)$ defined on monomials by
    \begin{equation*}
        \hd(\alpha\beta^\ast)
        =
        (N_1(\alpha) - N_1(\beta)) \alpha\beta^\ast
    \end{equation*}
    is a graded derivation, where $N_1(\gamma)$ denotes the number of occurrences of $e_1$ in a path $\gamma$.  
    It is the~infinitesimal generator of the above action in the first coordinate, \ie,
    \begin{equation*}
        \hd(x)
        =
        \left.\frac{d}{dz}\right\vert_{z=1}\alpha_{(z,1)}(x).
        \end{equation*}
    In particular, $\hd$ restricts to a derivation on $L_0$.
\end{lemma}
\begin{proof}
By the universal property of $L$, it suffices to verify that the assignment
on generators preserves the defining relations.
%$e_i^\ast e_j = \delta_{ij}v$, $i,j\in\{1,2\}$, and $e_1 e_1^\ast + e_2 e_2^\ast = v$.
For $z_1 = z$ and $z_2 = w$, one has
\begin{equation*}
    \alpha_{(z,w)}(e_i^\ast) \alpha_{(z,w)}(e_j)
    =
    z_i^{-1} z_j e_i^\ast e_j
    =
    \delta_{ij}v,
\end{equation*}
and
\begin{equation*}
    \alpha_{(z,w)}(e_1)\alpha_{(z,w)}(e_1^\ast)
    +
    \alpha_{(z,w)}(e_2)\alpha_{(z,w)}(e_2^\ast)
    =
    e_1 e_1^\ast + e_2 e_2^\ast
    =
    v.
\end{equation*}
Thus $\alpha_{(z,w)} \in \operatorname{Aut}(L_{\mathbb C}(1,2))$ for all
$(z,w) \in (\mathbb C^\times)^2$.

Let $\alpha,\beta$ be finite paths. 
Since $\alpha_{(z,1)}$ acts multiplicatively and rescales each occurrence of $e_1$ by $z$ and
each occurrence of $e_1^\ast$ by $z^{-1}$,
\begin{equation*}
    \alpha_{(z,1)}(\alpha\beta^\ast)
    =
    z^{N_1(\alpha) - N_1(\beta)} \alpha\beta^\ast.
\end{equation*}
Differentiating at $z=1$ gives
\begin{equation*}
    \left.\frac{d}{dz}\right\vert_{z=1}\alpha_{(z,1)}(\alpha\beta^\ast)
    =
    (N_1(\alpha) - N_1(\beta))\alpha\beta^\ast .
\end{equation*}
Hence
\begin{equation*}
    \hd(x)
    =
    \left.\frac{d}{dz}\right\vert_{z=1}\alpha_{(z,1)}(x)
\end{equation*}
for all monomials $x$, and therefore for all $x\in L_{\mathbb C}(1,2)$ by linearity.
That $\hd$ is a derivation follows from the fact that each $\alpha_{(z,1)}$, $z \in \mathbb{C}^\times$, is an algebra automorphism. 
Moreover, $\hd$ is graded, because each $\alpha_{(z,1)}$, $z \in \mathbb{C}^\times$, preserves the canonical $\mathbb Z$-grading of $L_{\mathbb C}(1,2)$.

If $x = \alpha\beta^\ast\in L_0$, then $|\alpha|=|\beta|$, and so
$\hd(x) \in L_0$. 
Thus $\hd$ restricts to a derivation of $L_0$.
\end{proof}

In fact, it turns out that there is an easy description of all graded derivations of $\LConetwo$.

\begin{lemma}\label{lem:der_L(1,2)}
    For every matrix
    \begin{equation*}
        A=
        \begin{pmatrix}
            a_1 & b
            \\
            c & a_2
        \end{pmatrix}
        \in M_2(L_0)
    \end{equation*}
    there exists $\delta_A \in \der_{\mathrm{gr}}(L_{\mathbb{C}}(1,2))$
    determined by
    \begin{equation*}
    \begin{aligned}
        \delta_A(e_1) &= e_1a_1 + e_2c,
        &\quad
        \delta_A(e_2) &= e_1b + e_2a_2,
        \\
        \delta_A(e_1^\ast) &= -a_1e_1^\ast - be_2^\ast,
        &\quad
        \delta_A(e_2^\ast) &= -ce_1^\ast - a_2e_2^\ast .
    \end{aligned}
    \end{equation*}
    Conversely, every $\delta\in \der_{\mathrm{gr}}(L_{\mathbb{C}}(1,2))$ is of the form $\delta = \delta_A$ for a unique $A\in M_2(L_0)$.
\end{lemma}
\begin{proof}
%Write $e := (e_1,e_2)$.
Let
\begin{equation*}
    A=
    \begin{pmatrix}
        a_1 & b
        \\
        c & a_2
    \end{pmatrix}
    \in M_2(L_0)
\end{equation*}
and define $\delta_A$ on the generators by the above formulas. 
Writing $e := (e_1,e_2)$, one has $\delta_A(e) = eA$ and $\delta_A(e^\dagger) = -Ae^\dagger$.
Since
\begin{equation*}
    e^\dagger e =
    \begin{pmatrix}
        1 & 0\\
        0 & 1
    \end{pmatrix}
    \quad \text{and} \quad
    ee^\dagger = 1,
\end{equation*}
it follows that
\begin{equation*}
    \delta_A(e^\dagger)e + e^\dagger\delta_A(e)
    =
    -Ae^\dagger e + e^\dagger e A = 0
\end{equation*}
and
\begin{equation*}
    \delta_A(e)e^\dagger + e\delta_A(e^\dagger)
    =
    eAe^\dagger - eAe^\dagger = 0.
\end{equation*}
Hence $\delta_A$ is compatible with the defining relations and therefore extends to a graded derivation of~$L_{\mathbb{C}}(1,2)$.

Conversely, let $\delta \in \der_{\mathrm{gr}}(L_{\mathbb{C}}(1,2))$.
Since $L_1$ is generated as a free right $L_0$-module by $\{e_1,e_2\}$ by Lemma~\ref{lem:rank2}, there exists $A \in M_2(L_0)$ such that 
\begin{equation*}
    \delta_A(e_1) = e_1a_1 + e_2c
    \quad \text{and} \quad
    \delta_A(e_2) = e_1b + e_2a_2.
\end{equation*}
Writing
\begin{equation*}
    A=
    \begin{pmatrix}
    a_1 & b
    \\
    c & a_2
\end{pmatrix},
\end{equation*}
one obtains
\begin{equation*}
    \delta(e_1)=e_1a_1+e_2c
    \quad \text{and} \quad 
    \delta(e_2)=e_1b+e_2a_2.
\end{equation*}
It follows that 
\begin{align*} 
    0 
    &= 
    \delta(e_1^\ast e_1) = \delta(e_1^\ast)e_1 + e_1^\ast\delta(e_1) 
    = 
    \delta(e_1^\ast)e_1 + a_1, 
    \\ 
    0 
    &= 
    \delta(e_1^\ast e_2) 
    = 
    \delta(e_1^\ast)e_2 + e_1^\ast\delta(e_2) = \delta(e_1^\ast)e_2 + b. 
\end{align*} 
Multiplying the first identity on the right by $e_1^\ast$ and the second by $e_2^\ast$ gives
\begin{equation*} 
    \delta(e_1^\ast) = -a_1e_1^\ast - be_2^\ast. 
\end{equation*} 
A similar computation shows that
\begin{equation*} 
    \delta(e_2^\ast) = -ce_1^\ast - a_2e_2^\ast. 
\end{equation*}
Now, let $A,A'\in M_2(L_0)$ such that $\delta_A=\delta_{A'}$ and write
\begin{align*}
    A = 
    \begin{pmatrix}
    a_1 & b
    \\
    c & a_2
    \end{pmatrix}
    \qquad
    A' = 
    \begin{pmatrix}
    a_1' & b'
    \\
    c' & a_2'
    \end{pmatrix}.
\end{align*}
Since $\delta_A(e_1)=\delta_{A'}(e_1)$ one finds that
\begin{align*}
    e_1a_1 + e_2c = e_1a_1' + e_2c'    
\end{align*}
which, by multiplying with $e_1^\ast$ from the left, yields $a_1=a_1'$. Similarly, one shows that $a_2=a_2'$, $c=c'$ and $b=b'$. Hence, $\delta_{A}=\delta_{A'}$ implies that $A=A'$.
\end{proof}

\noindent
Note that there are many $\delta_A\in\Dergr(\LConetwo)$ restricting to nontrivial derivations on $L_0$. For instance, considering
\begin{equation*}
    \delta_A(e_1e_2^\ast) = e_1(a_1-a_2)e_2^\ast + e_2ce_2^\ast - e_1ce_1^\ast
\end{equation*}
(in the notation of Lemma~\ref{lem:der_L(1,2)}) gives a nonzero element of $L_0$ in general.

\begin{remark}\label{rem:der_L(1,2)}
The gauge derivations $\gau(L_{\mathbb{C}}(1,2)) \cong \mathbb{C}$, described in Example~\ref{ex:L(1,2)_atiyah}, are realized by scalar matrices $A = \lambda \id_2$ with $\lambda \in \mathbb{C}$.
In this case one obtains
\begin{equation*}
    \delta_A(\alpha\beta^\ast)
    =
    \lambda (|\alpha|-|\beta|)\alpha\beta^\ast
    =
    \lambda \deg(\alpha\beta^\ast)\alpha\beta^\ast,
\end{equation*}
and, in particular, $\delta_A(r)=0$ for all $r\in L_0$.
Moreover, hermitian derivations are precisely those $\delta_A$ for which $A^\dagger = -A$. 
\end{remark}

Furthermore, it is easy to express the Lie algebra structure of graded derivations in terms of their representation by elements of $M_2(L_0)$.

\begin{lemma}\label{lem:commutator}
    Let $A_1,A_2 \in M_2(L_0)$. 
    Then
    \begin{equation*}
        [\delta_{A_1},\delta_{A_2}]
        =
        \delta_{[A_1,A_2]+\delta_{A_1}(A_2)-\delta_{A_2}(A_1)} .
    \end{equation*}
\end{lemma}
\begin{proof}
Set $\delta_1 := \delta_{A_1}$, $\delta_2 := \delta_{A_2}$, and $e := (e_1,e_2)$.
Since $\delta_i(e) = eA_i$, $i \in \{1,2\}$, one has
\begin{equation*}
    \begin{aligned}
        [\delta_1,\delta_2](e)
        &= \delta_1(eA_2)-\delta_2(eA_1)
        \\
        &= eA_1A_2 + e\delta_1(A_2) - eA_2A_1 - e\delta_2(A_1)
        \\
        &= e([A_1,A_2]+\delta_1(A_2)-\delta_2(A_1)).
    \end{aligned}
\end{equation*}
Thus
\begin{equation*}
[\delta_1,\delta_2](e)
=
\delta_{[A_1,A_2]+\delta_{A_1}(A_2)-\delta_{A_2}(A_1)}(e).
\end{equation*}
A similar computation shows that
\begin{equation*}
[\delta_1,\delta_2](e^\dagger)
=
\delta_{[A_1,A_2]+\delta_{A_1}(A_2)-\delta_{A_2}(A_1)}(e^\dagger).
\end{equation*}
This establishes the result.
\end{proof}

\begin{corollary}\label{cor:lieaut}
    The assignment $M_2(L_0) \ni A \mapsto \delta_A \in \der_{\mathrm{gr}}(L_{\mathbb{C}}(1,2))$ defines an isomorphism of Lie algebras from $(M_2(L_0),[\cdot,\cdot]_{\mathrm{gr}})$ onto $\der_{\mathrm{gr}}(L_{\mathbb C}(1,2))$, where
    \begin{equation*}
        [A_1,A_2]_{\mathrm{gr}}
        :=
        [A_1,A_2]+\delta_{A_1}(A_2)-\delta_{A_2}(A_1)
    \end{equation*}
    for all $A_1,A_2 \in M_2(L_0)$.
\end{corollary}

% For $\lambda\in\complex$ let
% \begin{align*}
%     A_\lambda = \lambda 
%     \begin{pmatrix}
%         0 & 1 \\ 1 & 0
%     \end{pmatrix}
% \end{align*}
% giving 
% \begin{align*}
%     &\delta_{A_\lambda}(e_1) = \lambda e_2\qquad
%     \delta_{A_\lambda}(e_2) = \lambda e_1\\
%     &\delta_{A_\lambda}(e_1^\ast) = -\lambda e_2^\ast\qquad
%     \delta_{A_\lambda}(e_2^\ast) = -\lambda e_1^\ast
% \end{align*}
% satisfying $\delta_{A_{\lambda}}\notin\gau(\LConetwo)$. For $\lambda,\mu\in\complex$ one checks that 
% \begin{align*}
%     [A_\lambda,A_{\mu}]=\delta_{A_{\lambda}}(A_{\mu})
%     =\delta_{A_{\mu}}(A_{\lambda})=0
% \end{align*}
% giving $[\delta_{A_{\lambda}},\delta_{A_{\mu}}]=0$; for convenience we set $\delta_{\lambda}=\delta_{A_{\lambda}}$ and $\delta_{\mu}=\delta_{A_{\mu}}$.  Letting $\g\subseteq\Derext(L_0)$ be the abelian Lie algebra generated by $\delta_{\lambda}\vert_{L_{0}},\delta_{\mu}\vert_{L_0}$
% we define $\sigma:\g\to\Dergr(\LConetwo)$ by 
% \begin{align*}
%     &\sigma(\delta_\lambda\vert_{L_0}) = A_{\lambda}\qquad
%     \sigma(\delta_{\mu}\vert_{L_0}) = \delta_{B}\qquad
%     B = 
%     \begin{pmatrix}
%         z & \mu \\ \mu & z
%     \end{pmatrix}
% \end{align*}
% for $0\neq z\in\complex$.

\bibliographystyle{alpha}
\bibliography{rings}

\begin{thebibliography}{SW17b}

\bibitem[AH26]{ah:on.the.existence.LC}
J.~Arnlind and V.~Hildebrandsson.
\newblock On the existence of noncommutative {L}evi-{C}ivita connections in
  derivation based calculi.
\newblock {\em J. Geom. Phys.}, 222:105764, 2026.

\bibitem[Arn24]{a:lc.kronecker}
J.~Arnlind.
\newblock Noncommutative {R}iemannian geometry of {K}ronecker algebras.
\newblock {\em J. Geom. Phys.}, 199:Paper No. 105145, 22, 2024.

\bibitem[Bae34]{Baer34}
R.~Baer.
\newblock Erweiterung von {Gruppen} und ihren {Isomorphismen}.
\newblock {\em Math. Z.}, 38:375--416, 1934.

\bibitem[Dad80]{Dade1980}
E.~C. Dade.
\newblock Group-graded rings and modules.
\newblock {\em Math. Z.}, 174(3):241--262, 1980.

\bibitem[EM42]{MacEil42}
S.~Eilenberg and S.~MacLane.
\newblock Group extensions and homology.
\newblock {\em Ann. Math. (2)}, 43:757--831, 1942.

\bibitem[EM47]{EilenbergMacLane}
S.~Eilenberg and S.~MacLane.
\newblock Cohomology theory in abstract groups. {II}. {G}roup extensions with a
  non-{A}belian kernel.
\newblock {\em Ann. of Math. (2)}, 48:326--341, 1947.

\bibitem[Haz13]{h:graded.struct.lpa}
R.~Hazrat.
\newblock The graded structure of {L}eavitt path algebras.
\newblock {\em Israel J. Math.}, 195(2):833--895, 2013.

\bibitem[Lec85]{Lec85}
P.~B.~A. Lecomte.
\newblock Sur la suite exacte canonique associ\'ee \`a un fibr\'e principal.
\newblock {\em Bull. Soc. Math. France}, 113(3):259--271, 1985.

\bibitem[LO22]{LuOi22}
P.~Lundstr\"{o}m and J.~\"{O}inert.
\newblock Strongly graded {L}eavitt path algebras.
\newblock {\em J. Algebra Appl.}, 21(07):2250141, 2022.

\bibitem[ML95]{MacLane95}
Saunders Mac~Lane.
\newblock {\em Homology}.
\newblock Class. Math. Berlin: Springer-Verlag, reprint of the 3rd corr. print.
  1975 edition, 1995.

\bibitem[NvO82]{NaOy82}
C.~Nastasescu and F.~van Oystaeyen.
\newblock {\em Graded ring theory}, volume~28 of {\em North-Holland Math.
  Libr.}
\newblock Elsevier (North-Holland), Amsterdam, 1982.

\bibitem[Sch25]{Schreier1925}
O.~Schreier.
\newblock \"{U}ber die {E}rweiterung von {G}ruppen. {II}.
\newblock {\em Abh. Math. Sem. Univ. Hamburg}, 4(1):321--346, 1925.

\bibitem[SW17a]{SchWa15}
K.~Schwieger and S.~Wagner.
\newblock {Part} {I}, {Free} actions of compact {A}belian groups on
  {C$^*$}-algebras.
\newblock {\em Adv. Math.}, 317:224--266, 2017.

\bibitem[SW17b]{SchWa16}
K.~Schwieger and S.~Wagner.
\newblock {Part} {II}, {Free} actions of compact groups on {C$^*$}-algebras.
\newblock {\em J. Noncommut. Geom.}, 11(2):641--688, 2017.

\bibitem[SW17c]{SchWa17}
K.~Schwieger and S.~Wagner.
\newblock {Part III}, {Free} actions of compact quantum groups on
  {C$^*$}-algebras.
\newblock {\em SIGMA}, 13(062):19 pages, 2017.

\bibitem[SW22]{SchWa21}
Kay Schwieger and Stefan Wagner.
\newblock An {Atiyah} sequence for noncommutative principal bundles.
\newblock {\em SIGMA}, 18(015):22 pages, 2022.

\bibitem[Wag19]{Wa18}
S.~Wagner.
\newblock Secondary characteristic classes of {L}ie algebra extensions.
\newblock {\em Bull. Soc. Math. France}, 147(3):443--453, 2019.

\end{thebibliography}

\end{document}